\documentclass[twoside,11pt]{amsart}

\RequirePackage{amssymb,amsmath,amsthm,amsfonts}
\usepackage[square,numbers]{natbib}
\RequirePackage{geometry}
\RequirePackage[shortlabels]{enumitem}
\RequirePackage{comment}
\RequirePackage{bm}
\RequirePackage{url}            
\RequirePackage{booktabs}       
\RequirePackage{nicefrac}       
\RequirePackage{microtype}      
\RequirePackage{dsfont}
\RequirePackage{accents}
\RequirePackage{bbm}
\RequirePackage{color}

\RequirePackage{ulem}
\RequirePackage{breakcites}     
\RequirePackage{graphicx}
\RequirePackage{xcolor}

\RequirePackage{algorithm}
\RequirePackage[noend]{algpseudocode}

\RequirePackage[breaklinks=true,colorlinks,citecolor=blue,urlcolor=blue]{hyperref}

\theoremstyle{plain}
\newtheorem{theorem}{Theorem}
\newtheorem{proposition}{Proposition}
\newtheorem{lemma}{Lemma}
\newtheorem{corollary}{Corollary}
\theoremstyle{remark}
\newtheorem{remark}{Remark}
\newtheorem{definition}{Definition}

\newcommand{\Cov}[0]{\mathsf{Cov}}
\newcommand{\Var}[0]{\mathsf{Var}}
\newcommand{\iid}[0]{\overset{iid}{\sim}}

\newcommand{\ind}[0]{\mathds{1}}

\newcommand{\calB}[0]{\mathcal{B}}
\newcommand{\calC}[0]{\mathcal{C}}

\newcommand{\calS}[0]{\mathcal{S}}

\newcommand{\supp}{\mathsf{supp}}
\newcommand{\E}[0]{\mathbb{E}}

\newcommand{\R}[0]{\mathbb{R}}

\newcommand{\Prob}[0]{\mathbb{P}}

\renewcommand{\tilde}{\widetilde}
\renewcommand{\hat}{\widehat}

\newcommand{\vertiii}[1]{{\left\vert\kern-0.25ex\left\vert\kern-0.25ex\left\vert #1 
    \right\vert\kern-0.25ex\right\vert\kern-0.25ex\right\vert}}
\newcommand\numberthis{\addtocounter{equation}{1}\tag{\theequation}}

\newcommand{\cA}{\mathcal{A}}
\newcommand{\cB}{\mathcal{B}}
\newcommand{\cC}{\mathcal{C}}

\newcommand{\cF}{\mathcal{F}}

\newcommand{\cH}{\mathcal{H}}

\newcommand{\cK}{\mathcal{K}}

\newcommand{\cN}{\mathcal{N}}

\newcommand{\cP}{\mathcal{P}}

\newcommand{\cS}{\mathcal{S}}
\newcommand{\cT}{\mathcal{T}}
\newcommand{\cU}{\mathcal{U}}

\newcommand{\cX}{\mathcal{X}}

\newcommand{\EE}{\mathbb{E}}

\newcommand{\GG}{\mathbb{G}}

\newcommand{\NN}{\mathbb{N}}

\newcommand{\PP}{\mathbb{P}}

\newcommand{\RR}{\mathbb{R}}

\newcommand{\ZZ}{\mathbb{Z}}
\newcommand*{\dd}{\, \mathrm{d}}
\newcommand{\vasti}{\bBigg@{3.5 }}
\newcommand{\vast}{\bBigg@{4}}
\newcommand{\Vast}{\bBigg@{5}}
\newcommand{\Vastt}{\bBigg@{7}}
\makeatother
\newcommand{\be}{\begin{equation}}
\newcommand{\ee}{\end{equation}}
\newcommand{\ba}{\begin{align}}
\newcommand{\ea}{\end{align}}
\newcommand{\baa}{\begin{align*}}
\newcommand{\eaa}{\end{align*}}
\newcommand{\argmin}{\mathop{\mathrm{argmin}}}

\newcommand*{\gauss}{\varphi_\sigma}

\newcommand*{\Gauss}{\mathcal{N}_\sigma}

\newcommand{\wass}{\mathsf{W}_1}
\newcommand{\law}{\mathsf{Law}}
\newcommand{\Wp}{\mathsf{W}_p}
\newcommand{\GWp}{\mathsf{W}_p^{\mspace{1mu}\sigma}}

\newcommand{\lip}[1]{\big\|#1\big\|_{\mathsf{Lip}_{1,0}}}
\newcommand{\ip}[1]{\big<#1\big>}

\newcommand{\infgwass}{\inf_{\theta\in\Theta}\mathsf{W}_1^{\mspace{1mu}\sigma}(P,Q_\theta)}
\newcommand{\infgwassn}{\inf_{\theta\in\Theta}\mathsf{W}_1^{\mspace{1mu}\sigma}(P_n,Q_\theta)}

\newcommand{\gwass}{\mathsf{W}_1^{\mspace{1mu}\sigma}}

\newcommand{\gwasssmall}{\mathsf{W}_1^{\mspace{1mu}\sigma_n}}
\newcommand{\gwassemp}{\mathsf{W}_1^{\mspace{1mu}\sigma}(P_n,P)}
\newcommand{\gwassempsmall}{\mathsf{W}_1^{\mspace{1mu}\sigma_n}(P_n,P)}

\newcommand{\convp}{\overset{\PP}{\to}}

\DeclareMathOperator{\Card}{Card}

\DeclareMathOperator{\diam}{\mathsf{diam}}

\begin{document}

\title[Limit Distribution Theory for the Smooth 1-Wasserstein Distance]{Limit Distribution Theory for the Smooth 1-Wasserstein Distance with Applications}

\author[Sadhu]{Ritwik Sadhu}
\address[Ritwik Sadhu]{Department of Statistics and Data Science, Cornell University.}
\email{rs2526@cornell.edu}

\author[Goldfeld]{Ziv Goldfeld}
\address[Ziv Goldfeld]{Department of Electrical and Computer Engineering, Cornell Univeristy.}
\email{goldfeld@cornell.edu}

\author[Kato]{Kengo Kato}
\address[Kengo Kato]{Department of Statistics and Data Science, Cornell University.}
\email{kk976@cornell.edu}



\begin{abstract}%
The smooth 1-Wasserstein distance (SWD) $\gwass$ was recently proposed as a means to mitigate the curse of dimensionality in empirical approximation, while preserving the Wasserstein structure. Indeed, SWD exhibits parametric convergence rates in any data dimension and inherits the metric and topological structure of the classic 1-Wasserstein distance. To facilitate principled learning and inference with the SWD, 
this work conducts a thorough statistical study thereof, encompassing a high-dimensional limit distribution theory for empirical $\gwass$, bootstrap consistency, concentration inequalities, and Berry-Esseen type bounds. 
We provide sharp characterization of the dependence of convergence rates on the smoothing parameter $\sigma$ and account for regimes when it decays with $n$ at a sufficiently slow rate. 
As applications of the limit distribution theory, we study two-sample testing and implicit generative modeling (when both the data distribution and the model are sampled) via $\gwass$ minimum expected distance estimation (MEDE). 
We establish asymptotic validity of homogeneity testing based on $\gwass$, while for MEDE,
we prove measurability, almost sure convergence, and limit distributions for optimal estimators and their corresponding $\gwass$ error. Our~results suggest that the SWD is well suited for high-dimensional statistical learning and inference.

\smallskip

\noindent \textbf{Keywords.} intrinsic dimension, limit theorems, multivariate homogeneity testing, optimal transport, smooth Wasserstein distance, strong approximation.
\end{abstract}

\maketitle


\section{Introduction}

Optimal transport (OT) is a versatile framework for quantifying discrepancy between probability distributions. It has a longstanding history in mathematics, dating back to the Monge in the 18th century \citep{mongeOT1781}. Recently, OT has emerged as a tool of choice for a myriad of machine learning tasks, encompassing clustering \citep{ho2017multilevel}, computer vision \citep{rubner2000earth,sandler2011nonnegative,li2013novel}, generative modeling \citep{arjovsky2017wasserstein,gulrajani2017improved}, natural language processing \citep{alvarez2018gromov,yurochkin2019hierarchical,grave2019unsupervised}, and model ensembling \citep{dognin2019wasserstein}. This wide adoption was driven by recent computational advances \citep{cuturi2013sinkhorn, solomon2015convolutional, genevay2016stochastic} and the rich geometric structure with which OT-based discrepancy measures, like Wasserstein distances, endow the space of probability measures.
However, Wasserstein distances suffer from the so-called curse of dimensionality (CoD) when distributions are estimated from samples, with empirical convergence rates of $n^{-1/d}$, whenever the dimension is $d>3$. This slow rate renders performance guarantees in terms of $\wass$ all but vacuous when $d$ is large. It is a also a major roadblock towards a more delicate statistical analysis concerning limit distributions, bootstrapping, and second order rates, all of which are central for valid statistical inference.


Recently, smoothing the underlying distribution via convolution with a Gaussian kernel was proposed as a means for mitigating the CoD while preserving the Wasserstein structure. Specifically, for a smoothing parameter $\sigma>0$ and $p\geq 1$, consider $\GWp(P,Q):=\Wp(P\ast\Gauss,Q\ast\Gauss)$, where $\Gauss=\cN(0,\sigma^2\mathrm{I}_d)$ is the $d$-dimensional isotropic Gaussian measure of parameter $\sigma$, $P\ast\Gauss$ is the convolution of $P$ and $\Gauss$, and $\Wp$ is the regular $p$-Wasserstein distance. In contrast to the CoD rates, \citet{Goldfeld2020convergence} showed that several popular statistical divergences, including $\wass$ and $\mathsf{W}_2$, enjoy fast empirical convergence of $n^{-1/2}$ once distributions are smoothed by $\Gauss$. The followup work \citet{Goldfeld2020GOT} focused on $\gwass$ and showed that the smooth distance inherits the metric and topological structure of its classic counterpart. That work also established regularity properties of $\gwass$ in $\sigma$, and quantified the smoothing as $\big|\wass(P,Q)-\gwass(P,Q)\big|=O(\sigma\sqrt{d})$. The structural properties and parametric empirical convergence rates were extended to general $p\geq 1$ in \cite{nietert2021smooth}. 
Other follow-up works explored relations between $\GWp$ and maximum mean discrepancies \citep{zhang2021convergence}, analyzed its rate of decay as $\sigma\to \infty$ \citep{chen2021asymptotics}, considered it in the context of sequential estimation of divergences \cite{manole2021sequential}, and adopted it as a performance metric for nonparametric mixture model estimation \citep{han2021nonparametric}, demonstrating accelerated (polynomial) MLE convergence rates.
Motivated by the above, herein we conduct an in-depth statistical study of the Gaussian smoothed Wasserstein distance (SWD), exploring high dimensional limit distributions of the empirical distance under the null and the alternative, bootstrapping, Berry-Esseen type bounds, vanishing $\sigma$ analysis, as well as several applications.


\subsection{Contributions}\label{SUBSEC:contributions}

This work extends our earlier conference paper \citep{Goldfeld2020limit_wass}. That paper derived the first limit distribution result for the empirical SWD that holds in arbitrary dimensions, and applied it to study minimum distance estimation (MDE) under $\gwass$. Alongside the fast empirical convergence of $\gwass$ in expectation \citep{Goldfeld2020convergence,Goldfeld2020GOT}, this limit distribution result further demonstrated the enhanced statistical properties of $\gwass$. However, the analysis in \citet{Goldfeld2020limit_wass} was limited to the one-sample setting under the null, not accounting for statistical questions concerning two-sample approximation, speed of convergence to the distributional limit, sharpness of expectation bounds in $\sigma$, adaptation to the intrinsic dimensionality of the data, and MDE when only samples from the model are available. These aspects are key for valid statistical inference, and serve to motivate the present work.


Our first main result herein derives the limit distribution of the smooth empirical distance $\sqrt{n}\gwass(P_n,P)$ in any dimension $d$, and characterizes sharp dependence of the expected empirical convergence rate on $\sigma$
(Theorem \ref{thm: CLT for W1}). 
Here, $P_n:=n^{-1}\sum_{i=1}^n\delta_{X_i}$ is the empirical measure based on $n$ independently and identically distributed (i.i.d.) samples, $X_1,\ldots,X_n$, from $P$. The analysis relies on the Kantorovich-Rubinstein (KR) duality \citep{villani2008optimal}, which allows representing $\sqrt{n}\gwass(P_n,P)$ as a supremum of an empirical process indexed by the class of 1-Lipschitz functions convolved with a Gaussian density. We prove that this function class is $P$-Donsker (i.e., satisfies the uniform central limit theorem (CLT)) under a polynomial moment condition on $P$. By the continuous mapping theorem, we conclude that $\sqrt{n}\gwass(P_n,P)$ converges in distribution to the supremum of a tight Gaussian process. We then extend this result to the two-sample case $\sqrt{mn/(m+n)}\gwass(P_n,Q_m)$, providing distributional limits under both the null ($P=Q$) and the alternative ($P \neq Q$). To enable approximation of the distributional limits, we prove consistency of the nonparametric bootstrap (Corollary \ref{cor: bootstrap}). Our high-dimensional SWD limit distribution theory contrasts the classic $\wass$ case, for which similar results are known only when $d=1$ \citep{del1999central}. 
Indeed, while 1-Lipschitz functions are generally not Donsker for $d>1$, smoothing them via the Gaussian convolution shrinks the class and enables controlling its complexity by uniformly bounding higher order derivatives. As a consequence, empirical SWD is amenable to analysis via empirical process theory\footnote{The reader is referred to, e.g., \citet{LeTa1991,VaWe1996,GiNi2016} as useful references on modern empirical process theory.}, using which we derive our results.


Given the limit distribution theory, it is natural to ask how accurate the distributional approximations are. To that end we explore a Berry-Esseen type analysis of empirical SWD, as well as its empirical bootstrap analogue (Theorems \ref{thm: SWD small sigma Berry-Esseen type rate} and \ref{thm: Empirical Bootstrap Berry-Esseen type rate} in Section \ref{sec: Berry-Esseen type results}). Specifically, we bound the absolute difference between $\sqrt{n}\gwassemp$ and its distributional limit in probability, and as a corollary, obtain a bound on the Prokhorov distance between their distributions. These results also hold when the smoothing parameter $\sigma_n$ vanishes with $n$ at a slow enough rate, an analysis motivated by by noise annealing techniques used in machine learning. These results provide quantitative bounds on the speed of convergence of $\sqrt{n}\gwassemp$ (or its bootstrap approximation) towards its asymptotic distribution. 
Our results also imply that the 
expected empirical SWD satisfies $\E\big[\gwass (P_{n},P)\big] \lesssim \sigma^{-d/2+1} n^{-1/2}\log n$. 
This is reminiscent of \citet[Proposition 1]{Goldfeld2020convergence}, but with a sharper dependence on $\sigma$ in terms of $d$. In Section \ref{sec: low dim case}, we further refine this bound to show that $d$ can be replaced with an appropriate notion of \textit{intrinsic} dimension of the population distribution. 
These results, together with the stability of $\gwass$ with respect to (w.r.t.)~$\sigma$ \citep[Lemma 1]{Goldfeld2020GOT}, enable leveraging SWD to analyze classic $\wass$ and derive convergence rates thereof depending on intrinsic dimension, as described next.

We consider three applications of the derived theory: (i) dependence of the classic empirical $\wass$ on the intrinsic dimension of the support; (ii) two-sample homogeneity testing; and (iii) generative modeling via minimum expected distance estimation (MEDE) with the SWD. For the first problem, we derive a bound on $\EE\big[\wass(P_n,P)\big]$ that adapts to the (smooth) Wasserstein intrinsic dimension of $P$ (see \citet{weed2019sharp}). Our result matches known sharp convergences rates from \citet{weed2019sharp}, while removing the compact support assumption previous works imposed  \citep{dudley1969speed,boissard2014,weed2019sharp}; similar results for unbounded supports but under a different notion of intrinsic dimensionality are available in \citet{singh2018minimax,lei2020}. 
For SWD two-sample homogeneity testing, we calibrate critical values via the bootstrap and prove consistency of the resulting test. Lastly, we consider MEDE as a formulation of implicit generative modeling when only samples are available from both the data distribution and the model (as is often the case in practice, e.g., generative adversarial networks (GANS) \cite{nowozin2016f,arjovsky2017wasserstein,gulrajani2017improved}). We establish measurability and strong consistency of the estimator $\hat{\theta}_{m,n}\in\argmin_{\theta\in\Theta}\EE\big[\gwass(P_n,Q_{\theta,m})\big|X_1\ldots,X_n\big]$, along with almost sure convergence of the associated minimal distance. Next, we characterize the high-dimensional limits of MEDE solutions and their $\gwass$ error, 
establishing $n^{-1/2}$ convergence rates for both quantities in arbitrary dimensions. Similar MEDE limit distribution results for $\wass$ are only known for $d=1$, bottlenecked by the lack of high-dimensional limit distribution theory for the classic distance \citep{bassetti2006minimum,bernton2019parameter} (see also \citet{belili1999estimate,bassetti2006asymptotic}). 
Our results pose $\gwass$ as a potent tool for high-dimensional learning and inference, and demonstrate its utility for analyzing classic Wasserstein distances.


\subsection{Related Works}
The problems of empirical approximation, limit distributions, and MDE studied here for $\gwass$ have received considerable attention under the classic $\wass$ framework. Since $\wass$ metrizes weak convergence and convergence of first moments in $\cP_1(\RR^d)$, Varadarajan's theorem (see, for example, Theorem 11.4.1 in \citet{Du2002}) yields $\wass(P_n,P)\to0$ as $n\to\infty$ almost surely (a.s.) for any $P \in \cP_1(\R^d)$. The convergence rate of the expected value, first studied by Dudley \citep{dudley1969speed}, is well understood with sharp rates known in all dimensions.\footnote{Except $d=2$, where a log factor is possibly missing.} In particular, $\E\big[\wass(P_n,P)\big]\lesssim n^{-1/d}$ provided that $P$ has sufficiently many moments \citep{FournierGuillin2015}; see also \citet{bolley2007,boissard2011,dereich2013,boissard2014,lei2020,chizat2020faster,manole2021sharp} for extensions of this result to other orders of the Wasserstein distance and general Polish spaces. 
Notably, this rate suffers from the CoD. 
Imposing structural assumptions on the distributions may accelerate convergence rates \citep{boissard2014,weed2019sharp,niles2019estimation,hutter2021minimax,weed2019minimax,manole2021plugin,deb2021rates}, but such assumptions are hard to verify in practice.

Despite the comprehensive account of expected $\wass(P_n,P)$, limiting distribution results for a scaled version thereof are scarce. In fact, a limit distribution characterization analogous to that presented herein for the SWD (i.e., assuming only finiteness of moments) is known for the classic distance only when $d=1$. Specifically, Theorem 2 in \citet{GineZinn1986} yields that $\mathsf{Lip}_{1}(\R)$ is a $P$-Donsker class if (and only if) $\sum_{j} P\big([-j,j]^{c}\big)^{1/2}<\infty$.
Combining with KR duality and the continuous mapping theorem, we have $\sqrt{n} \wass(P_{n},P) \stackrel{w}{\to} \|G_{P}\|_{\mathsf{Lip}_{1}(\R)}
$ for some tight Gaussian process $G_{P}$ in $\ell^{\infty}(\mathsf{Lip}_{1}(\R))$. An alternative derivation of the one-dimensional limit distribution is given in \citet{del1999central} based on the fact that $\wass$ equals the $L^1$ distance between cumulative distribution functions (CDFs) when $d=1$. However, these arguments do not generalize to higher dimensions. When $d \ge 2$, the function class $\mathsf{Lip}_{1}(\R^{d})$ is not Donsker; if it was, then $\EE[\wass (P_n,P)]$ would be of order $O(n^{-1/2})$, contradicting known lower bounds \citep{FournierGuillin2015}. 

For $p$-Wasserstein distances with general $p \geq 1$, the limit distribution of $\sqrt{n} \mathsf{W}_{p}^p(P_n,P)$, when $P$ is supported on a finite or a countable set, was derived in \citet{Sommerfeld2018} and \citet{Tameling2019}, respectively. Their proof relies on the directional functional delta method \citep{Dumbgen1993}. Also related is \citet{del2019central}, where asymptotic normality of $\sqrt{n}\big(\mathsf{W}_2^2(P_n,Q)-\EE\big[\mathsf{W}_2^2(P_n,Q)\big]\big)$ in an arbitrary dimension is derived under the alternative $P\ne Q$; see also \cite{del2021central} for the extension to general $p > 1$. Notably, the centering term is the expected empirical distance (and not the population one), which is undesirable since such results do not give confidence intervals for $\mathsf{W}_2(P,Q)$.  The recent preprint \citet{manole2021plugin} addressed this gap and established a CLT for $\sqrt{n}\big(\mathsf{W}_2^2(\hat{P}_n,Q)-\mathsf{W}_2^2(P,Q)\big)$ under the alternative $P \ne Q$, but for a wavelet-based estimator $\hat{P}_n$ of $P$ (as opposed to the empirical distribution), while assuming that the ambient space is $[0,1]^d$ and imposing smoothness and boundedness assumptions on the Lebesgue densities of $P,Q$. 
The limit distribution of the empirical $2$-Wasserstein distance under the null $Q=P$ is known only when $d=1$ \citep{delbarrio2005}. The key observation there is that when $d=1$, the empirical $2$-Wasserstein distance coincides with the $L^2$ distance between the empirical and population quantile functions, resulting in a non-Gaussian limit. When $d=1$, a CLT for the empirical $p$-Wasserstein distance with the population centering was derived in \cite{del2019central1D} for $p \ge 2$. Their argument relies on the quantile expression of the Wasserstein distance in $d=1$ and some techincal results from \cite{delbarrio2005}.
Another related work is \citet{Goldfeld2020limit_fdiv}, where limit distribution results for smooth empirical total variation distance and $\chi^2$ divergence were derived based on the CLT in Banach spaces. None of the techniques employed in these works are applicable in our case, which requires a different analysis as described in~Section~\ref{sec: limit}.

Two approaches to alleviate the CoD besides smoothing are slicing \citep{rabin2011wasserstein} and entropic regularization \citep{cuturi2013sinkhorn}. Empirical convergence under sliced distances follows their rates when $d=1$ \citet{nadjahi2020statistical}, which amounts to $n^{-1/2}$ for sliced $\wass$ \citep{niles2019estimation,lin2021projection}. A sliced 1-Wasserstein MDE and MEDE analyses, covering questions similar to those considered here, was provided in \citet{nadjahi2019asymptotic}. The authors proved a limit distribution~result for empirical sliced $\wass$, but their assumptions included a Donsker-type theorem that makes the conclusion immediate. Entropic regularization of OT also speeds up empirical convergence rates in some cases. Specifically, the rate of two-sample empirical convergence under entropic OT (EOT) is $n^{-1/2}$ for smooth costs (thus excluding entropic $\wass$) with compactly supported distributions \citep{genevay2019sample}, or quadratic cost with sub-Gaussian distributions \citep{mena2019statistical}.

A CLT for empirical EOT under quadratic cost was also derived in \citet{mena2019statistical} (see also \citet{bigot2019central,klatt2020empirical}). This result is similar to that of \citet{BarrioLoubes2019} for the classic $\mathsf{W}_2$ as it is centered by the expected empirical EOT, and not the population one. 
While both sliced Wasserstein distances and EOT can be efficiently computed (sliced distances amount to simple one-dimensional formulas while EOT employs Sinkhorn iterations \citep{cuturi2013sinkhorn,altschuler2017near}), sliced distances offer a poor proxy of classic ones \citep[Lemma 5.1.4]{bonnotte2013unidimensional}, while EOT forfeits the Wasserstein metric and topological structure \citep{feydy2018interpolating,bigot2019central}. On the other hand, $\gwass$ preserves the Wasserstein structure and is within an additive $2\sigma\sqrt{d}$ gap from $\wass$ \citep[Lemma 1]{Goldfeld2020GOT}. Computational aspects of the SWD, however, are still premature and call for further exploration (see Section \ref{SEC:summary}).


\section{Background and Preliminaries}

\subsection{Notation}
Let $\| \cdot \|$ denote the Euclidean norm, and $x \cdot y$, for $x,y \in \R^{d}$, designate the inner product. The diameter of a set $X\subset \RR^d$ is $\mathsf{diam}(X):=\sup_{x,y\in X} \|x-y\|$. In a metric space $(S,d)$, $B(x,\epsilon)$ denotes the open ball of radius $\epsilon$ around $x$: if $x$ is the origin in $\R^d$ it is omitted. For a subset $A$ of a metric space, we denote its $\epsilon$-blowup as $A^\epsilon := \{x \in S: d(x,A) \leq \epsilon\}$.  We write $a\lesssim_{x} b$ when $a\leq C_x b$ for a constant $C_x$ that depends only on $x$ ($a\lesssim b$ means the hidden constant is absolute), and $a_n \ll b_n$ for $\lim_{n\to\infty} a_n/b_n \to 0$. We denote by $a \vee b$ and $a\wedge b$ the maximum and minimum of $a$ and $b$, respectively.

We denote by $(\Omega,\mathcal{A},\PP)$ the probability space on which all random variables are defined. The class of Borel probability measures on $\RR^d$ is $\cP(\RR^d)$. The subset of measures with finite first moment is denoted by $\cP_1(\RR^d)$, i.e., $P\in\cP_1(\RR^d)$ whenever $P\|x\|<\infty$. The convolution of $P,Q\in\cP(\RR^d)$ is $P\ast Q(A) := \int\int\mathds{1}_{A}(x+y)\dd P(x)\dd Q(y)$, where $\mathds{1}_A$ is the indicator of $A$. The convolution of measurable functions $f,g$ on $\R^{d}$ is $f \ast g(x) = \int f(x-y)\,g(y) \dd y$. We also recall that $\Gauss:=\cN(0,\sigma^2 \mathrm{I}_d)$, and use $\gauss(x) = (2\pi \sigma^{2})^{-d/2} e^{-\|x\|^{2}/(2\sigma^{2})}$, $x \in \R^{d}$, for the Gaussian density. The law of a random variable $X$ is denoted by $\law(X) = \PP \circ X^{-1}$. For any measure $Q$ on a measurable space $(S,\calS)$ and a measurable function $f:S\to\RR$, we write $Qf := \int_{S} f \dd Q$, whenever the corresponding integral exists. 
Let $\smash{\stackrel{w}{\to}}$, $\smash{\stackrel{d}{\to}}$, and $\smash{\stackrel{\Prob}{\to}}$ denote weak convergence of probability measures, convergence in distribution of random variables, and convergence in probability,~respectively: the first two are sometimes used interchangeably.

For a non-empty set $\cT$, let $\ell^\infty(\cT)$ be the space of all bounded functions $f:\cT\to\RR$, equipped with the uniform norm $\|f\|_\cT:= \sup_{t\in \cT} \big|f(t)\big|$. We write $\mathsf{Lip}_{1}(\R^{d}) := \{ f: \R^{d} \to \R : |f(x) - f(y) | \le \| x - y \|\,,\ \forall x,y \in \R^{d} \}$ for the set of functions on $\R^{d}$ whose Lipschitz constant is bounded by $1$, and define $\mathsf{Lip}_{1,0}(\R^{d}) := \big\{ f\,\in \mathsf{Lip}_{1}(\R^{d}) : f(0) = 0\big\}$. 
When $d$ is clear from the context we use the shorthands $\mathsf{Lip}_1$ and $\mathsf{Lip}_{1,0}$. We write $N(\epsilon, \cF, \mathsf{d})$ for the $\epsilon$-covering number of a function class $\cF$ w.r.t. a metric $\mathsf{d}$, and $N_{[\,]}(\epsilon, \cF, \mathsf{d})$ for the bracketing number. The $L^p(P)$ norm of a function $f$ on $(S,\cS)$ is $\big(\int_S |f|^p \dd P\big)^{1/p}$; $P$ is omitted from the norm/space notation when it is the Lebesgue measure. We use $N(\epsilon, \cF, L^p(P))$ and $N(\epsilon, \cF, \|\cdot\|_{L^p(P)})$ interchangeably. 
For a measurable set $A$, we use~$\ind_A$ for its indicator function, and $\cF|_A$ for the restriction of a function class $\cF$ to subset $A$ of its~domain.

\subsection{The 1-Wasserstein Distance}
The Kantorovich OT problem \citep{kantorovich1942translocation} between $P,Q\in\cP(\R^{d})$ with cost $c(x,y)$ seeks to minimize \(\int c(x,y)\dd\pi(x,y)\) over $\pi \in \Pi(P,Q)$,
where $\Pi(P,Q)$ is the set of couplings (or transport plans) between $P$ and $Q$. The 1-Wasserstein distance ($\wass$) \citep[Chapter 6]{villani2008optimal} is the minimum value achieved for the Euclidean distance cost $c(x,y)=\|x-y\|$:
\[
\wass(P,Q) = \inf_{\pi \in \Pi(P,Q)} \int_{\R^d \times \R^d} \|x - y\|\ \dd\pi(x,y).
\]

The Wasserstein space $\big(\cP_1(\RR^d),\wass\big)$ is a metric space, as $\wass$ metrizes weak convergence together with convergence of first moments \citep{villani2008optimal}. KR duality \citep{kantorovitch1958space} gives rise to the following representation of the 1-Wasserstein distance
\begin{equation*}
    \wass(P,Q) = \inf_{f \in \mathsf{Lip}_1} \int_{\R^d}  f\, \dd(P-Q).
\end{equation*}
This form is invariant to shifting the function $f$ by a constant. Consequently we may assume without loss of generality that $f(0)=0$, and optimize over the class $\mathsf{Lip}_{1,0}$ instead. 
We subsequently leverage this dual form to derive the limit distribution theory for the smooth version of the distance.

\subsection{Smooth 1-Wasserstein Distances}
The smooth 1-Wasserstein distance levels out local irregularities in the distributions via convolution with the Gaussian kernel. It is defined as 
\[
\gwass(P,Q):=\wass(P\ast\Gauss,Q\ast\Gauss),
\]
where $\Gauss$ is the $d$-dimensional isotropic Gaussian measure with parameter $\sigma$. \citet{Goldfeld2020GOT} showed that $\gwass$ inherits the metric and topological structure of $\wass$. Stability w.r.t. the smoothing parameter was also established in that work, showing that 
\begin{equation}
\big|\wass^{\mspace{1mu}\sigma}(P,Q)-\wass^{\mspace{1mu}\tau}(P,Q)\big|\leq 2\sqrt{d|\sigma^2-\tau^2|}.\label{EQ:stability}
\end{equation}
This enabled proving regularity properties of $\gwass(P,Q)$ in $\sigma$, such as continuity, monotonicity, and $\Gamma$-convergence.
In addition, \citet{Goldfeld2020convergence} showed that $\gwass$ alleviates the CoD, converging as $\EE\big[\gwass(P_n,P)\big] = O(n^{-1/2})$ in all dimensions under a sub-Gaussian condition on $P$. In the next section, sharpen the dependence on the expectation bound on $\sigma$, characterize the limit distributions of (scaled) empirical $\gwass$ under the null and the alternative, prove that the limits can be accurately estimated via the bootstrap, and establish concentration inequalities. Later, we move to a vanishing $\sigma$ analysis, Berry-Esseen type bounds, and applications.


\section{Limit Distribution Theory for Smooth Wasserstein Distance}
\label{sec: limit}

In this section we characterize the limit distribution of $\sqrt{n}\gwass (P_{n},P)$ in all dimensions. We also derive consistency of the bootstrap as a means for computing the limit distribution, establish concentration inequalities for $\gwass (P_{n},P)$, and then explore the two-sample setting.

\subsection{One-Sample Limit Distribution}
To state the result, some definitions are due. For $P\in\cP_2(\RR^d)$, let $G_{P}^{\mspace{1mu}\sigma}:=\big(G_{P}^{\mspace{1mu}\sigma}(f)\big)_{f\in\mathsf{Lip}_{1,0}}$ be a centered Gaussian process on $\mathsf{Lip}_{1,0}$ with covariance $\E\big[G_{P}^{\mspace{1mu}\sigma}(f)\,G_{P}^{\mspace{1mu}\sigma}(g)\big] = \Cov_{P} (f\ast\gauss,g\ast\gauss)$, where $f,g, \in \mathsf{Lip}_{1,0}$. One may verify that $|f\ast\gauss (x)| \le \| x \| + \sigma \sqrt{d}$ (cf. Section \ref{SUBSEC:CLT for W1 proof}), so that $P|f\ast\gauss |^{2} < \infty$, for all $f \in \mathsf{Lip}_{1,0}$ (which ensures that the said covariance function is well-defined). With that, we are ready to state the theorem.

\begin{theorem}[One-sample limit distribution]
\label{thm: CLT for W1}
Let $\R^{d} = \bigcup_{j=1}^{\infty} I_{j}$ be a partition of $\R^{d}$ into bounded convex sets with nonempty interior such that $K:=\sup_{j}\diam(I_{j})<\infty$ and $M_{j}:=\sup_{I_{j}} \|x\|\geq 1$ for all $j$. If $P\in\cP_2(\RR^d)$ satisfies
\begin{equation}
\label{eq: condition W1}
\sum_{j=1}^{\infty} M_{j}P(I_{j})^{1/2} < \infty,
 \end{equation}
then there exists a version of $G_P^{\mspace{1mu}\sigma}$ that is tight in $\ell^{\infty}(\mathsf{Lip}_{1,0})$, and denoting the tight version by the same symbol $G_{P}^{\mspace{1mu}\sigma}$, we have $\sqrt{n}\gwass(P_{n},P)  \stackrel{d}{\to}  \|G_{P}^{\mspace{1mu}\sigma}\|_{\mathsf{Lip}_{1,0}} =: L_P^{\mspace{1mu}\sigma}$. In addition, for any $\alpha > d/2$, we have 
\begin{equation}
\E\big[\gwass (P_{n},P)\big]\lesssim_{\alpha,d,K}  \frac{\sigma^{-\alpha+1}}{\sqrt{n}} \sum_{j=1}^{\infty} M_{j}P(I_{j})^{1/2}.\label{EQ:expected_bound}
\end{equation}

\end{theorem}

The proof, given below in Section \ref{SUBSEC:CLT for W1 proof}, uses KR duality to translate the Gaussian convolution in the measure space to the convolution of Lipschitz functions with a Gaussian density. It is then shown that the class of Gaussian-smoothed Lipschitz functions is $P$-Donsker by bounding the metric entropy of the function class restricted to each $I_{j}$. As the next corollary shows, the dependence on $\sigma$ in \eqref{EQ:expected_bound} can be tightened to $\sigma^{-d/2 + 1}$ up to a $\log n$ term.

\begin{corollary}[Expectation bound]\label{cor: sharp gwass rate}
Under Condition \eqref{eq: condition W1}, we have
\begin{equation}
\E\big[\gwass (P_{n},P)\big] \lesssim_{d,K}  \frac{\sigma^{-d/2+1}\log n}{\sqrt{n}} \sum_{j=1}^{\infty} M_{j}P(I_{j})^{1/2}.\label{EQ:expected_sharp}
\end{equation}
\end{corollary}
Corollary \ref{cor: sharp gwass rate} is proven in Appendix \ref{SUBSEC: sharp gwass rate proof}. Combining \eqref{EQ:expected_sharp} with the stability result~\eqref{EQ:stability} recovers the (correct) $n^{-1/d}$ bound on the classic empirical $\wass$, up to logarithmic factors. Specifically, taking $\sigma=n^{-1/d} \big ( \log n \sum_{j=1}^{\infty} M_{j}P(I_{j})^{1/2} \big )^{2 / d}$, we have
\[
\begin{split}
\E\big[\wass (P_{n},P)\big] &= \E\big[\gwass (P_{n},P)\big] + \E\big[\gwass (P_{n},P) - \wass (P_n,P)\big] \\
&\lesssim_{d,K}  \inf_{\sigma \in (0,1)}\frac{\sigma^{-d/2+1}\log n}{\sqrt{n}} \sum_{j=1}^{\infty} M_{j}P(I_{j})^{1/2} + \sigma \\
&\lesssim_{d,K} n^{-1/d} (\log n)^{2/d} \Big ( \sum_{j=1}^{\infty} M_{j}P(I_{j})^{1/2} \Big )^{\frac{2}{d}}. 
\end{split}
\]

A bound similar to \eqref{EQ:expected_bound} and \eqref{EQ:expected_sharp} was derived in Proposition 1 of \citet{Goldfeld2020convergence} with explicit dependence on $\sigma$ and $d$. However, their dependence on $\sigma$ is not sharp enough to imply the known polynomial rates for the empirical unsmoothed distance.

\begin{remark}[Discussion of Condition \eqref{eq: condition W1}]
\label{rem:moment_condition}
Let $\{ I_{j} \}$ consist of cubes with side length $1$ and integral lattice points as vertices. One may then obtain the bound
\[
\sum_{j=1}^{\infty} M_{j} P(I_{j})^{1/2} \lesssim_{d} \sum_{k=1}^{\infty} k^{d} P\big( \| x \|_{\infty} > k \big)^{1/2}  \lesssim \int_{1}^{\infty} t^{d} P\big(\| x \|_{\infty} > t \big)^{1/2}\dd t,
\]
which is finite (by Markov's inequality) if there exists $\epsilon>0$ such that $P|x_{j}|^{2(d+1)+\epsilon} < \infty$ for all $j$. Proposition 1 in \citet{Goldfeld2020convergence} shows that $\E\big[\gwass (P_{n},P)\big] =  O(n^{-1/2})$ whenever $P$ is sub-Gaussian. Theorem \ref{thm: CLT for W1} thus relaxes this moment condition, in addition to deriving a limit distribution. For $d=1$, the above condition can be relaxed to $\int_{0}^\infty \sqrt{P(|x| > t)}\,\dd t < \infty$ via a tighter argument with disjoint unit cubes, matching the condition from \citep[Section 3.1]{bobkov2019one} for empirical $\wass$.

\end{remark}

\subsection{Bootstrap Consistency for One-Sample Limit Distribution}

We next show that the cumulative distribution function (CDF) of the limiting variable in Theorem \ref{thm: CLT for W1} can be estimated via the bootstrap \citep[Chapter 3.6]{VaWe1996}. To that end we first need to guarantee that the distribution of $L_P^{\mspace{1mu}\sigma}$ has a Lebesgue density that is sufficiently regular (see Appendix \ref{SUPP: bootstrap proof}).  

\begin{lemma}[Distribution of $L_P^{\mspace{1mu}\sigma}$]
\label{lem: limit variable}
Assume that Conditions \eqref{eq: condition W1} holds and that $P$ is not a point mass. Then the distribution of $L_P^{\mspace{1mu}\sigma}$ is absolutely continuous w.r.t. the Lebesgue measure and its density is positive and continuous on $(0,\infty)$ except for at most countably many points. 
\end{lemma}

To define the bootstrap estimate, let $X_{1}^{B},\dots,X_{n}^{B}$ be i.i.d. from $P_n$ conditioned on $X_1,\ldots,X_n$, and set $P_{n}^{B}:= n^{-1}\sum_{i=1}^{n}\delta_{X_{i}^{B}}$
as the bootstrap empirical distribution. 
Let $\PP^{B}$ be the probability measure induced by the bootstrap (i.e., 
the conditional probability given $X_1,X_2,\dots$). We have the following bootstrap consistency.

\begin{corollary}[Bootstrap consistency]
\label{cor: bootstrap}
If Condition \eqref{eq: condition W1} holds and $P$ is not a point mass, then 
\[
\sup_{t \ge 0}  \left|\PP^{B}\big(\sqrt{n}\gwass (P_{n}^{B},P_{n})  \le t \big)- \PP\big( L_{P}^{\mspace{1mu}\sigma} \le t\big) \right| \to 0\quad \mbox{a.s.}
\]
\end{corollary}
This corollary, together with continuity of the distribution function of $L_{P}^{\mspace{1mu}\sigma}$ implies that for $\hat{q}_{1-\alpha} := \inf \big\{ t \ge 0 : \PP^{B} \big(\sqrt{n}\gwass (P_{n}^{B},P_{n}) \le t\big) \ge 1-\alpha \big\}$ (which can be computed numerically), we have $\PP\big( \sqrt{n}\gwass (P_{n},P) > \hat{q}_{1-\alpha}\big) = \alpha+o(1)$.

\subsection{Concentration Inequalities}


Next, we consider quantitative concentration inequalities for $\gwass (P_{n},P)$, under different moment conditions on $P$. For $\alpha > 0$, let $\| \xi \|_{\psi_{\alpha}} := \inf \{ C > 0 : \E[e^{(|\xi|/C)^{\alpha}}] \le 2 \}$ be the Orlicz $\psi_{\alpha}$-norm for a real-valued random variable $\xi$ (if $\alpha \in (0,1)$, then $\| \cdot \|_{\psi_{\alpha}}$ is a quasi-norm). We focus on the unbounded support case since a concentration inequality for compactly supported distributions was derived in \citet{Goldfeld2020limit_wass}.

\begin{proposition}[Concentration inequalities]
\label{prop: concentration}
Let $X\sim P$ with $\supp(P)=\cX\subseteq \R^d$ and suppose one of the following conditions holds:
(i) $\big\| \| X \| \big\|_{\psi_{\alpha}} < \infty$ for some $\alpha \in (0,1]$;
(ii) $P\|x\|^{q} < \infty$ for some $q \in [1,\infty)$.
Then, for all $\eta \in (0,1)$, $\delta > 0$, there exist constants $C := C_{\eta,\delta}$, such that for all $t > 0$:
\be
\label{eq:gwass_concentration}
\PP\Big(\gwass(P_n,P) > (1+\eta)\EE\big[\gwass(P_n,P)\big] + t\Big) \leq \exp \left (-\frac{nt^{2}}{C\big(P\|x\|^{2}+\sigma^{2}d\big)} \right ) + a(n,t),
\ee
where
\[a(n,t) = \begin{cases}
3 \exp \left ( -\left ( \frac{nt}{C \left((\log(1+n))^{1/\alpha} \left\|\| X \|  \right\|_{\psi_{\alpha}}+\sigma \sqrt{d}\right)} \right)^{\alpha}\right )&, \text{ under (i)} \vspace{2mm}\\
C\big(n^{1/q}\E[ \| X \|^{q}] + \sigma^{q} d^{q/2}\big)(nt)^{-q}&,\text{ under (ii)}
\end{cases}.\]
\end{proposition}
The proof of the above result can be found in Appendix \ref{SUBSEC: concentration}.

\begin{remark}[Comparison to \citet{Goldfeld2020GOT}]
Corollary 1 of \citet{Goldfeld2020GOT} establishes the following Gaussian concentration of $\gwass(P_n,P)$ when $P$ has bounded support:
\be
\label{eq: concentration_ineq_compact}
\PP \big (  \gwass (P_{n},P) \ge \E\big[\gwass (P_{n},P)\big]+ t \big ) \le e^{-\frac{2nt^{2}}{\diam (\mathcal{X})^{2}}},\quad \forall t>0.
\ee
Herein, we generalize this result to unbounded support. The term $a(n,t)$ accounts for different concentration behavior of $\gwass(P_n, P)$ for large $t$. We expect that the logarithmic term in Case (i) can be removed via refinements akin to \citet{fournier2015rate}. 
\end{remark}

\subsection{Two-Sample Limit Distributions and Bootstrap}

Theorem \ref{thm: CLT for W1} and Corollary \ref{cor: bootstrap} can be extended to the two-sample case, i.e., accounting for $\gwass(P_n,Q_m)$, where $P_n$ and $Q_m$ correspond to possibly different population measures $P$ and $Q$. We first treat the null case ($P=Q$) and then consider the alternative ($P\neq Q$). Let $N := n + m$ and consider the statistic
\[
W_{m, n} := \sqrt{\frac{mn}{N}}\gwass(P_n,Q_m) = \sqrt{\frac{mn}{N}}\|P_n - Q_m\|_{\cF_\sigma},
\]
where $\cF_\sigma:= \left\{ f\ast\gauss : f \in \mathsf{Lip}_{1,0} \right\}$ (cf. \eqref{eq: smooth Wasserstein}), and define its bootstrap version as follows. Let $H_N$ be the pooled empirical measure constructed from i.i.d. samples $X_1, \dots, X_n \sim P$ and $Y_1, \dots, Y_m \sim Q$. Letting $Z_1, \dots, Z_N$ be a bootstrap sample from $H_N$, construct bootstrap measures $P_n^B :=n^{-1}\sum_{i=1}^n \delta_{Z_i}$, $Q_m^B := m^{-1}\sum_{i=n+1}^N \delta_{Z_i}$, and let 
\[W_{m,n}^B = \sqrt{\frac{mn}{N}} \gwass(P_n^B, Q_m^B).\]
The following corollary shows that when $P = Q$, $W_{m,n}$ converges in distribution to $L_P^{\mspace{1mu}\sigma}$, and the bootstrap distribution of $W_{m,n}^B$ is also consistent for the same limit (see Appendix~\ref{subsec: 2 sample limit proof} for the~proof).

\begin{corollary}[Two-sample limit distribution and bootstrap under the null]
\label{cor: 2 sample limit}
\ \ Let\\ $P,Q\in\cP(\RR^d)$ be such that $\sum_{j=1}^{\infty} M_{j}P(I_{j})^{1/2} ,\sum_{j=1}^{\infty} M_{j}Q(I_{j})^{1/2} < \infty$ and assume that $P$ is non-degenerate ($M_j$ is as in Theorem~\ref{thm: CLT for W1}). 
For $m,n \in \NN$, set $N = m+n$, and assume $n/N \to \lambda \in (0,1)$ as $m,n \to \infty$. The following hold:
\begin{enumerate}[(i)]
    \item If $P=Q$, then $W_{m,n} \stackrel{d}{\to} L_P^{\mspace{1mu}\sigma}$, and $W_{m,n} \stackrel{\PP}{\to} \infty$, otherwise.
    \item For $H:= \lambda P + (1-\lambda)Q$, we have $\law\big(W_{m,n}^B|X_1,\dots,X_n,Y_1,\dots,Y_m\big) \stackrel{w}{\to} \law \big ( L_H^{\mspace{1mu}\sigma} \big ),$ $\PP-$a.s. When $P=Q=H$ we further have
\[\sup_{t\in\RR}\,\Big|\PP^B\big(W_{m,n}^B \leq t\big) - \PP\big(L_P^{\mspace{1mu}\sigma} \leq t\big)\Big| \to 0\quad \text{a.s.}\]
\end{enumerate}
\end{corollary}

Evidently, the $W_{m,n}$ statistic diverges when $P\neq Q$. Thus, to treat the alternative case (i.e., when distributions are different), consider the centered statistic 
\[\overline{W}_{m,n} := \sqrt{\frac{mn}{m+n}} \big(\gwass(P_n,Q_m) - \gwass(P,Q)\big).\]
Theorem 6.1 of \citet{carcamo2020directional} uses directional differentiability of the supremum norm and the functional delta method to derive limit distributions for integral probability metrics (IPMs) \citep{zolotarev1983probability,muller1997integral} over Donsker classes. The $\overline{W}_{m,n}$ statistic satisfies this condition, since, up to centering, it is an IPM over $\cF_\sigma$, which is $P$- and $Q$-Donsker under appropriate moment conditions. This leads to the following result.

\begin{corollary}[Two-sample limit distribution under the alternative]
\label{cor: 2 sample limit alt}
Assume the conditions of Corollary \ref{cor: 2 sample limit} and set
$\mathsf{d}_\sigma(f,g):= \sqrt{P(f\ast\varphi_\sigma-g\ast\varphi_\sigma)^2} + \sqrt{Q(f\ast\varphi_\sigma-g\ast\varphi_\sigma)^2}$. Then $(\mathsf{Lip}_{1,0},\mathsf{d}_\sigma)$ is a totally bounded pseudometric space, and we have
\be\label{eq: two sample limit alt}\overline{W}_{m,n} \stackrel{d}{\to} \|G_{P,Q,\lambda}^{\mspace{1mu}\sigma}\|_{\bar{M}_\sigma} =: L^{\mspace{1mu}\sigma}_{P,Q,\lambda},\ee 
where $G_{P,Q,\lambda}^{\mspace{1mu}\sigma} = \sqrt{1-\lambda}\,\tilde G_P^{\mspace{1mu}\sigma} - \sqrt{\lambda}\,\tilde G_Q^{\mspace{1mu}\sigma}$ and
\[
\bar{M}_\sigma := \Big\{\bar{f} \in \overline{\mathsf{Lip}}_{1,0}:\, (P - Q)\,\bar{f}\ast\varphi_\sigma = \sup\nolimits_{f\in\mathsf{Lip}_{1,0}} (P - Q)\,f\ast\varphi_\sigma\Big\}.
\]
Here $\overline{\mathsf{Lip}}_{1,0}$ is the $\mathsf{d}_\sigma$-completion of $\mathsf{Lip}_{1,0}$, and $\tilde G_P^{\mspace{1mu}\sigma}$, $\tilde G_Q^{\mspace{1mu}\sigma}$ are the extensions of $G_P^{\mspace{1mu}\sigma}$, $G_Q^{\mspace{1mu}\sigma}$ to the completion, respectively.
\end{corollary}

This corollary follows from an application of Theorem 6.1 in \citet{carcamo2020directional}.

\begin{remark}[Relation between Corollaries \ref{cor: 2 sample limit} and \ref{cor: 2 sample limit alt}]
Under $P=Q$, $\tilde G_P^{\mspace{1mu}\sigma}$ and $\tilde G_Q^{\mspace{1mu}\sigma}$ are independent copies of the same Gaussian process, $\mathsf{d}_\sigma(f,g) = 2 \mathsf{d}_P(f,g) := 2 \|f\ast\gauss - g\ast\gauss\|_{L^2(P)}$, and $\bar{M}_\sigma = \overline{\mathsf{Lip}}_{1,0}$. In this case, the distribution of $G_{P,Q,\lambda}^\sigma$ above coincides with the limit distribution found in Corollary~\ref{cor: 2 sample limit}. When $P\neq Q$, Corollary \ref{cor: 2 sample limit} shows that the bootstrap distribution of $W^B_{m,n}$ (a symmetrized bootstrap version of $W_{m,n}$) consistently estimates the null limit distribution. For testing $H_0:\,P=Q$ against the alternative $H_1:\,P\neq Q$, the quantiles of $W^B_{m,n}$ can be used to choose critical values for the test. In contrast, Corollary \ref{cor: 2 sample limit alt} may be used to find confidence intervals for $W_1^\sigma(P,Q)$ using samples from the two distributions $P$, $Q$
.
\end{remark}

\begin{remark}[Uniqueness of potentials] 
\label{rem:unique_potential} If the set $\bar{M}_\sigma$ of optimal potentials for $P,Q$ is a singleton up to additive constants (i.e., the optimal potential is unique), then the limit distribution in \eqref{eq: two sample limit alt} is a univariate Gaussian. However, no such uniqueness result is known for $\wass$ in general, even for smooth measures, since the underlying cost function is not strictly convex. Recent results on uniqueness of Kantorovich potentials are given in \citet{bernton2021entropic,staudt2022uniqueness}. Theorem 2 of \citet{staudt2022uniqueness} shows uniqueness of potentials under certain regularity conditions on the optimal coupling and the cost. However $c_1(x,y) := \|x - y\|$ is non differentiable at $x = y$, which violates Condition (3) of their result when, for instance, $P$ has a density that is non-zero everywhere and $\supp(Q)$ is compact with non-empty interior. Theorem B.2 of \citet{bernton2021entropic} also requires differentiability of the cost and superlinear growth of $c(x,y)=h(x - y)$, neither of which holds for $c_1$.
\end{remark}

We next consider the bootstrap analog of the centered statistic $\overline{W}_{m,n}$. This is motivated by the fact that the distribution of the limit random variable $L^{\mspace{1mu}\sigma}_{P,Q,\lambda}$ under the alternative is non-pivotal in the sense that it depends on the population distributions $P$ and $Q$, which are unknown in practice. However, the naive bootstrap is inconsistent in estimating the alternative limit distribution in general. To state the result, let  $\tilde{P}_n^B, \tilde{Q}_m^B$ be bootstrap samples from $X_1, \dots, X_n$ and $Y_1, \dots, Y_m$, respectively, and define the bootstrap estimate
\[
\overline{W}_{m,n}^B = \sqrt{\frac{mn}{m+n}} \left ( \gwass (\tilde{P}_n^B, \tilde{Q}_m^B) - \gwass(P_n, Q_m) \right ).
\]

\begin{proposition}[Bootstrap inconsistency under the alternative]
\label{prop: two_sample_bootstrap_alt}
Under the conditions of Corollary~\ref{cor: 2 sample limit alt}, we have the following unconditional limit of $\bar{W}_{m,n}^B$.
\[
\law \left ( \overline{W}_{m,n}^B \right ) \overset{w}{\to} \law \left ( \| G^{\mspace{1mu}\sigma}_{P,Q,\lambda} + \tilde{G}^{\mspace{1mu}\sigma}_{P,Q,\lambda}\|_{\bar{M}_\sigma} - \| G^{\mspace{1mu}\sigma}_{P,Q,\lambda}\|_{\bar{M}_\sigma} \right ),
\]
where $G^\sigma_{P,Q,\lambda}$ is given in \eqref{eq: two sample limit alt} and $\tilde{G}^\sigma_{P,Q,\lambda}$ is an independent copy thereof.
\end{proposition}

The proof of Proposition~\ref{prop: two_sample_bootstrap_alt} is given in Appendix~\ref{subsec: two_sample_bootstrap_alt}. The unconditional limit in Proposition~\ref{prop: two_sample_bootstrap_alt} is incompatible with the conditional limit distribution in Corollary~\ref{cor: 2 sample limit alt} unless the Gaussian process $G_{P,Q,\lambda}^\sigma$ is a.s. constant on $\bar{M}_\sigma$. A sufficient condition that ensures this, and consequently, consistency of the naive bootstrap, is when the set of optimal potentials $\bar M_\sigma$ is a singleton up to an additive constant. However, conditions to ensure this are not known in general (cf. Remark~\ref{rem:unique_potential}).

Even if the naive bootstrap is inconsistent, the alternative limit distribution can be estimated via the rescaled bootstrap scheme specified in \cite{Dumbgen1993}. We describe below an alternative approach, based on subsampling \citep{politis2008k}, to estimate the quantiles of the alternative limit distribution and obtain confidence intervals for $\gwass(P,Q)$ from samples. 

To describe the subsampling scheme, first let $1 \leq a_n \leq n$, $1 \leq b_m \leq m$ be sequences of positive integers, and set $r_{m,n}:= \sqrt{mn/(m + n)}$. Take
\[A_n\sim \mathrm{Unif}\binom{\{1,.\ldots,n\}}{a_n}\quad ;\quad B_m\sim \mathrm{Unif}\binom{\{1,.\ldots,m\}}{b_m},\]
where $\binom{S}{k}$ designates the set of all $k$-subsets of a set $S$. Define  $P_n^\mathsf{Sub}:=a_n^{-1}\sum_{i\in A_n}X_i$ and $Q_m^\mathsf{Sub}:=b_m^{-1}\sum_{i\in B_m}Y_i$ as the empirical measures of the subsamples. Lastly, let $J_{m,n}(x)$ be the CDF of $\law \big ( r_{a_n, b_m} \big( \gwass(P_n^\mathsf{Sub}, Q_m^\mathsf{Sub}) - \gwass(P_n, Q_m) \big) \,\big|\,X_1, \dots, X_n, Y_1, \dots, Y_m \big )$, and $J_{\lambda}(x)$ be the CDF of $\law \big ( L^{\mspace{1mu}\sigma}_{P,Q,\lambda} \big )$. Then, $J_{m,n}$ acts as the subsampling estimator for $J_\lambda$, and the following result shows that it is a consistent estimator (under both $H_0$ and $H_1$).


\begin{lemma}[Consistency of subsampling]
\label{lem: subsample_consistency_H1}
Assume the conditions of Corollary \ref{cor: 2 sample limit alt}  
and suppose that $a_n \vee b_m \to \infty$, $a_n/(a_n + b_m) \to \lambda$, $a_n/n \wedge b_m/m \to 0$, and $r_{a_n,b_m}/r_{m,n} \to 0$. Then $J_{m,n}(x) \convp J_\lambda(x)$ for every continuity point $x$ of $J_\lambda$.
\end{lemma}

This lemma follows from a modification of \citet[Theorem 3.1]{politis2008k}, so as to account for the dependence of the limit distribution on $\lambda = \lim_{m \vee n \to \infty} n/(m+n)$. Nevertheless, the proof is all but identical, and is therefore skipped.

\medskip
As a consequence of Lemma \ref{lem: subsample_consistency_H1}, we obtain confidence intervals for the SWD.

\begin{corollary}[Consistent confidence intervals under the alternative]
Under the\\ conditions of Lemma~\ref{lem: subsample_consistency_H1}, an asymptotic $(1-\alpha)$ confidence set for $\gwass(P,Q)$ is given by
\[
\left[ \gwass(P_n, Q_m) + \frac{J_{m,n}^{-1}(\alpha/2)}{r_{m,n}},\  \gwass(P_n, Q_m) + \frac{J_{m,n}^{-1}(1 - \alpha/2)}{r_{m,n}} \right ].
\]
\end{corollary}


\section{Berry-Esseen Type Bounds for Vanishing Smoothing Parameter}
\label{sec: Berry-Esseen type results}

Given a limit distribution theory, Berry-Esseen type bounds quantify the accuracy of the distributional approximation. The classic Berry-Esseen theorem fulfills this role for the CLT, specifying the rate at which the sample mean CDF (of standardized i.i.d. samples) converges towards the standard normal distribution in the sup-norm \citep{berry1941accuracy,esseen1942liapunov}. Specifically, let $F_n(x) := \PP \big(n^{-1/2}\sum_{i=1}^n \xi_i \le x\big)$ for i.i.d. $\xi_1,\dots,\xi_n$ with mean zero and variance one, and take $\Phi$ as the CDF of $\cN_1$. The Berry-Esseen theorem yields
\[\sup_{x \in \R} \big|F_n(x) - \Phi(x)\big| \leq \frac{C \EE\big[|\xi_1|^3\big]}{\sqrt{n}},\]
where the constant $C$ above was refined multiple times throughout the years (see \citet{shevtsova2011absolute} for a review and a recent proof). 
This result has seen many extensions and generalizations, e.g., to Banach spaces \citep{rhee1986uniform} and similar approximation bounds for empirical and bootstrap processes over bounded and/or VC-type function classes \citep{massart1989strong,ChChKa2014,ChChKa2016}.

In this section, we derive Berry-Esseen type bounds for finite-sample approximations of the limiting distribution from Theorem \ref{thm: CLT for W1}, by either the (scaled) empirical SWD or its bootstrap analogue. 
Motivated by noise annealing techniques used in machine learning, our results allow the smoothing parameter $\sigma$ to vanishes as $n \to \infty$ at a sufficiently slow rate. Recalling that $\big|\gwass(P,Q)-\wass(P,Q)\big|\leq 2\sigma\sqrt{d}$, the vanishing $\sigma$ analysis enables viewing the SWD as a proxy of classic $\wass$. The discrepancy between the limit distribution and its non-asymptotic approximation is measured by the Prokhorov distance, given~by
\[\rho(P,Q):= \inf\big\{\epsilon > 0\ :\ P(A) \leq Q(A^\epsilon) + \epsilon\,,\ \forall A \in \cB(\R^d)\big\}.\]
With some abuse of notation, we write $\rho(X,Y)$ instead of $\rho(P,Q)$, for $X\sim P$ and $Y\sim Q$.


\medskip 

We first consider approximating 
the limiting variable $L_P^{\mspace{1mu}\sigma}$ from Theorem \ref{thm: CLT for W1} by the empirical SWD $\sqrt{n}\gwassemp$, for small $\sigma$ values. To simplify notation, we introduce the following quantities: for $d\geq 2$ and any $\alpha > d/2 $, define $a := \alpha - 1,\,b := d/(2\alpha)$.

\begin{theorem}[Berry-Esseen type bound]
\label{thm: SWD small sigma Berry-Esseen type rate}
Suppose $P\|x\|_\infty^{l} = M< \infty$ for some $l > 4\alpha-d+3$ and let $n^{-p} \leq \sigma_n \le 1$ for some $p < \frac{1-b}{8ab}$. Then:
\begin{enumerate}[i),leftmargin = 1.3\parindent]
    \item there exist versions $V_n$ and $W_n$ of $L^{\mspace{1mu}\sigma_n}_P$ and $\sqrt{n}\gwassempsmall$, respectively, 
    such that for any sequence $r_n \gg n^{-(1-b)/\left(2(3+b)\right)}\sigma_n^{- 4ab/(3+b)}$, we have $V_n - W_n = o_{\PP}(r_n)$;
    \item the Prokhorov distance is bounded as $\rho\big(\sqrt{n}\gwassempsmall,\,L_P^{\mspace{1mu}\sigma_n}\big) \lesssim_{\alpha,d,M,l} n^{-(1-b)/8} \sigma_n^{-ab}$.
\end{enumerate}
\end{theorem}
The proof of Theorem \ref{thm: SWD small sigma Berry-Esseen type rate} leverages the methodology of \citet{ChChKa2014} for approximating of the supremum of an empirical process over a VC-type class by that of a Gaussian process. Several adaptations of their technique are needed to account for the fact our function class is not VC-type.




\begin{remark}[Fast Berry-Esseen rates]
If $\alpha = 2d^2$, then $r_n\sim O(n^{-1/6 + \epsilon_d})\sigma_n^{-4(d^2 \mspace{-1mu} - \mspace{-1mu} 1)}$,
where $\epsilon_d$ is small for large $d$. Hence, if for instance, $\sigma_n$ decays at an inverse logarithmic rate and $d$ is sufficiently large, we have $r_n = o(n^{-1/6 + \epsilon})$ for some small $\epsilon > 0$. Consequently, when the population distribution $P$ has sufficiently high moments (say, it is subexponential), and $\sigma_n$ is appropriately chosen, we can obtain dimension free Berry-Esseen rates, in the sense  $n$ and $d$ are decoupled (although $d$ can still enter through the constants). This is~similar to the rate found in \citet{ChChKa2014}, for smaller (VC-type) function classes.
\end{remark}

\begin{remark}[Other distances]
With further restriction on $P$ (e.g., anti-concentration assumptions akin to \citet{chernozhukov2014anti}), the bound in Claim (ii) can be derived under a metric stronger than Prokhorov. However, such an analysis is quite technical and does not shed new light on the problem, and so we decided to omit it.
\end{remark}

In Corollary \ref{cor: bootstrap}, it was shown that the empirical bootstrap is consistent in estimating the limiting distribution of SWD. The next result provides Berry-Esseen rates for the bootstrap estimate of $L_P^{\mspace{1mu}\sigma}$, under slightly stronger moment conditions.


We establish the following analogue of Theorem \ref{thm: SWD small sigma Berry-Esseen type rate} for the empirical bootstrap.

\begin{theorem}[Berry-Esseen type bound for bootstrap]
\label{thm: Empirical Bootstrap Berry-Esseen type rate}
Under the same conditions as Theorem~\ref{thm: SWD small sigma Berry-Esseen type rate},
\begin{enumerate}[i)]
    \item there exist random variables $W_n$ and $V_n$ such that 
    \[
    \law(W_n|X_{1:n}) \mspace{-3mu} = \mspace{-3mu} \law\big(\sqrt{n}\gwasssmall(P_n^B,P_n)\big|X_{1:n}\big)
    \]
    and $\law(V_n|X_{1:n}) \mspace{-3mu} = \mspace{-3mu} \law\big(L_P^{\mspace{1mu}\sigma_n}\big)$, and for any sequence $r_n \gg n^{-(1-b)/(2(3+b))} \sigma_n^{-a}$, we have
    $    
    |V_n - W_n| = o_{\PP}(r_n)$.
    \item the Prokhorov distance is bounded as
    \[\rho\big(\law\left(\sqrt{n}\gwasssmall(P_n^B,P_n)\middle|\,X_{1:n}\right), \law\left(L_P^{\sigma_n}\right)\big) = O_\PP\big(n^{-\frac{1-b}{4(3+b)}}\sigma_n^{-a/2}\big).
    \]
\end{enumerate}
\end{theorem}

The empirical bootstrap process on $\cF_\sigma$ is defined as:
\be \GG_n^*(f) = \sqrt{n}(P_n^B - P_n)f,\ f \in \cF_\sigma. \ee
Recall that $P_n^B$ is the empirical measure formed from a bootstrap sample from $P_n$, as defined earlier. Then, $\sqrt{n}\gwass(P_n^B,P_n) = \|\GG_n^*\|_{\cF_\sigma}$. The proof of this theorem then follows along similar lines to those of Theorem \ref{thm: SWD small sigma Berry-Esseen type rate} above. The main difference is that conditional versions of the technical lemmas are needed. 
An outline of the argument is provided in Section \ref{SUBSEC: SWD Empirical Bootstrap proof outline}, with the detailed proof deferred to Appendix \ref{SUBSEC: SWD Empirical Bootstrap Berry-Esseen proof}.

\begin{remark}[Measurability and existence of versions]
The existence of versions $W_n$ and $V_n$ is guaranteed under the assumption that the underlying probability space is rich enough to contain a $\mathsf{U}[0,1]$ random variable independent of the sample observations (cf. \citet{ChChKa2016}, Page 3). This can be ensured by taking the product probability space $\prod_{i=1}^{\infty} \big(\R^{d}, \cB(\R^{d}),P\big)$ (in which the samples $X_1,X_2,\ldots$ are coordinate projections) and extending it as $(\Omega, \cA, \Prob) = \left [ \prod_{i=1}^{\infty} \big(\R^{d}, \cB(\R^{d}),P\big)\right] \times \big([0,1],\cB([0,1]), \mathrm{Leb}\big)$, where $\mathrm{Leb}$ denotes the Lebesgue measure on $[0,1]$.
\end{remark}


\section{Dependence on Intrinsic Dimension}
\label{sec: low dim case}

The rate of convergence of the empirical 1-Wasserstein distance $\wass(P_n,P)$ adapts to the intrinsic dimension of the population distribution \citep{dudley1969speed,boissard2014,niles2019estimation,weed2019sharp,lei2020, singh2018minimax}. 
For the SWD, while the rate is parametric irrespective of dimension, we next show that the dependence on the intrinsic dimension is captured by the Gaussian smoothing parameter~$\sigma$. 



\subsection{Empirical Convergence and Intrinsic Dimension}

The notion of intrinsic dimension used in the aforementioned papers and in this work employs covering numbers. The idea is to leverage the fact that the $\epsilon$-covering number of a set of intrinsic dimension $s$ will scale as $c\epsilon^{-s}$, where $c$ may depend on the size and the dimension of the ambient space. To state our intrinsic dimension definition, for any $P\in\cP(\RR^d)$, denote its $(\epsilon,\tau)$-covering number by
\[
N_\epsilon(P, \tau):= \inf\big\{N(\epsilon, S,\|\cdot\|)\ :\ S \in \cB(\R^d),\ P(S) \geq 1 - \tau\big\},
\]
and let $d_\epsilon(P, \tau) = -\frac{\log N_\epsilon(P,\tau)}{\log \epsilon}$ be the corresponding $(\epsilon,\tau)$-dimension.

\begin{definition}[Upper smooth 1-Wasserstein dimension]\label{DEF:gwass_dim}
For a probability measure $P \in \cP(\R^d)$, the upper smooth 1-Wasserstein dimension is defined as:
\[
\dim^*_{\mathsf{SW}_1}(P) := \inf\left\{s \in (2,\infty)\ :\ \limsup_{\epsilon \to 0} d_\epsilon\Big(P,\epsilon^{s^2/(s-2)}\Big) \leq s\right \}.
 \]
\end{definition}

The above definition is stricter than regular upper 1-Wasserstein dimension \citep{dudley1969speed,weed2019sharp} in the sense that for a given $\epsilon$, it considers sets that exclude a smaller probability mass (i.e., larger sets) for determining dimensionality. A detailed comparison between the two as well as relations to Minkowski dimension are provided in the next subsection. The stricter notion of dimensionality enables treating empirical convergence of distributions with unbounded support, as stated next.

\begin{theorem}[Empirical convergence and intrinsic dimension]
\label{thm: low dim rate for exp}
Let $P \mspace{-3mu} \in \mspace{-3mu} \cP(\R^d)$ be such that $P\|x\|^s < \infty$ and $\dim^*_{\mathsf{SW}_1}(P) < s$. Then, for all $n > 1$,
\[\EE\left[\sqrt{n}\gwassemp\right] \leq C \sigma^{-\frac{s}{2}+1}(\log n)^{\frac 32},\]
where $C$ is a constant independent of $\sigma$ and $n$.
\end{theorem}

Theorem \ref{thm: low dim rate for exp} follows from a stronger version that we establish in Section \ref{SUBSEC: low dim rate proof}, where the assumption that $\dim^*_{\mathsf{SW}_1}(P) < s$ is relaxed to a condition based on entropy. The key proof idea used here is to derive sharp estimates of the entropy of $\cF_\sigma$ restricted to small covering balls provided by the entropy bound. These bounds are then combined to provide a bound on the $L^2(P)$ entropy of $\cF_\sigma$, and the result follows by the application of a maximal inequality. 

\medskip

We conclude this subsection with a sufficient primitive condition on $P$ for $\dim^*_{\mathsf{SW}_1}(P)\leq s$ to hold. The next lemma states that this is the case for sub-exponential distributions supported on an affine $s$-dimensional subspace of $\R^d$ for $d \geq s$.

\begin{lemma}[Sufficient primitive condition]
\label{lem: affine support entropy bound}
Let $2 < s \leq d$. Suppose that $P\in\cP(\R^d)$ is supported on an affine space $H_s = \{x \in \RR^d\ : x - x_0 \in V_s\}$ for some $x_0 \in \R^d$, where $V_s$ is an $s$-dimensional linear subspace of $\RR^d$. Also suppose that $Z = \|X - x_0\|$ for $X \sim P$ has Orlicz $\psi_{\beta}$-norm $\|Z\|_{\psi_\beta} < \infty$, for some $\beta > 0$. Then, $\dim^*_{\mathsf{SW}_1}(P) \leq s$ and we have
\[\EE\big[\sqrt{n}\gwassemp\big] \leq C \sigma^{-\frac{s}{2}+1}(\log n)^{\frac{3\beta+1}{2\beta}},\]
where $C$ is independent of $n$ and $\sigma$.
\end{lemma}

\subsection{Implications to Classic 1-Wasserstein Distance}
\label{SUBSEC:wass_intrinsic}

Recall that combining Corollary \ref{cor: sharp gwass rate} with the bound on the gap between $\gwass$ and $\wass$ \citep[Lemma 1]{Goldfeld2020GOT}, recovers the known $n^{-1/d}$ rate for $\E\big[\wass (P_{n},P)\big]$ up to logarithmic factors. We can adopt a similar strategy in conjunction with Theorem \ref{thm: low dim rate for exp} to get an analogous result in terms of intrinsic dimension.

\begin{corollary}[Sharp $\wass$ rate under intrinsic dimension]
\label{cor: sharp wass rate low dim}
Under the assumptions of Theorem~\ref{thm: low dim rate for exp}, we have
$\E\big[ \wass (P_n,P) \big] \leq C n^{-1/s}$, where $C$ is a constant independent of $n$.
\end{corollary}
A direct computation suggests the presence of a log factor, but we can eliminate it at the cost of a constant dependent on $\dim^*_{\mathsf{SW}_1}(P)$ (see Appendix \ref{sec: low dim case proofs} for a proof). The dependence of the constants on $P$ is solely through its moment and $\dim_{\mathsf{SW1}}^*(P)$. Our bound is tight in the sense that if $P$ is supported on a regular set of dimension $s$ (see Section \ref{subsec: dim_discussion}), then Corollary 2 of \citet{weed2019sharp} implies
$\EE\big[W_1(P_n,P)\big] \gtrsim n^{-1/t}$, for any $t < s = \dim^*_{\mathsf{SW}_1}(P)$. A~similar bound was derived in \citet[Theorem 1]{weed2019sharp} under a compact support assumption, but Corollary \ref{cor: sharp wass rate low dim} establishes this rate for distributions with unbounded support. Other empirical convergence bounds for $\wass$ on unbounded spaces have been recently derived in \citet{singh2018minimax,lei2020}, but the dependence of their dimension definition on $P$ is implicit. 
Nevertheless, an application of \citet[Theorem 1]{singh2018minimax} recovers the $n^{-1/s}$ rate but under more stringent conditions, e.g., such as those in Lemma~\ref{lem: affine support entropy bound}.


\begin{remark}[Berry-Essen bounds w.r.t. intrinsic dimension]
Similar modification\\ of Theorems \ref{thm: SWD small sigma Berry-Esseen type rate} and \ref{thm: Empirical Bootstrap Berry-Esseen type rate} in the low intrinsic dimension regime are possible. The resulting rate bounds would be the same, with the only change being a faster feasible decay of the smoothing parameter. The overall rate will remain nearly $n^{-1/6}$, and for that reason, we omit the details of such refinements.
\end{remark}

\subsection{Relation to Minkowski and Wasserstein Dimension}
\label{subsec: dim_discussion}
The study of empirical convergence rates based on box-counting or Minkowski dimension and variants thereof dates back to \citet{dudley1969speed}, with more recent results available in \citet{boissard2014,weed2019sharp}. These notions of dimensionality are restricted to compact or bounded sets. We next compare these notions of intrinsic dimensions to our Definition \ref{DEF:gwass_dim}, under the appropriate boundedness assumptions.

\begin{definition}[Minkowski dimension]
\label{def: minkowski}
The Minkowski dimension $\dim_\mathsf{M}(S)$ of bounded sets
$S\subseteq\R^d$ is
\[\dim_\mathsf{M}(S) := \limsup_{\epsilon\to 0} \frac{\log N(\epsilon, S, \|\cdot\|_2)}{-\log\epsilon}.\]
For a compactly supported $P\in\cP(\R^d)$, we also write $\dim_\mathsf{M}(P) := \dim_\mathsf{M}\big(\supp(P)\big).$
\end{definition}

\begin{definition}[Upper Wasserstein dimension \citep{weed2019sharp}]
The upp-\\er 1-Wasserstein dimension for a compactly supported $P\in\cP(\R^d)$ is
\[\dim_{\mathsf{W}_1}^*(P) :=\inf\left\{s \in (2,\infty):\limsup_{\epsilon\to 0}d_\epsilon\left(P,\epsilon^{s/(s-2)}\right) \leq s\right\}.\]
\end{definition}

An extension of \citet[Proposition 2]{weed2019sharp} gives the following.
\begin{lemma}[Dimension comparison]
If $\dim_\mathsf{M}(P) \geq 2$, then
\[\dim_{\mathsf{W}_1}^*(P) \leq \dim^*_{\mathsf{SW}_1}(P) \leq \dim_\mathsf{M}(P).\]
\end{lemma}
Thus, the upper smooth 1-Wasserstein dimension lies between the unsmoothed variant and the classic Minkowski dimension. However, for probability measures with sufficiently well behaved supports, \citet[Proposition 8]{weed2019sharp} shows that $\dim_{\mathsf{W}_1}^*(P) = \dim_\mathsf{M}(P)$, as long as $\dim_\mathsf{M}(P) \geq 2$. Such supporting sets are called regular sets, as defined in \citet[Definition 12.1]{graf2007foundations}:

\begin{definition}[Regular sets \citep{graf2007foundations}]
A compact set $S \subset \R^d$ is called regular of dimension $s$ if there exist $c > 0, r_0 > 0$ such that
its Hausdorff measure $\cH^s(S)$ satisfies
$\frac{1}{c} r^d \leq \cH^s\big(S \cap B(x,r)\big) \leq c r^d$, for all $x \in S$ and $r\in(0,r_0)$.
\end{definition}

Proposition 9 in \citet{weed2019sharp} provides a list of common regular sets of dimension $s$. This includes, in particular, compact smooth $s$-dimensional manifolds and non-empty compact convex sets supported on an affine $s$-dimensional subspaces. Since $\dim_{\mathsf{W}_1}^*(P) = \dim_\mathsf{M}(P)$ for any $P$ that is supported on a regular set as defined above, we have $\dim_{\mathsf{W}_1}^*(P) = \dim^*_{\mathsf{SW}_1}(P) = \dim_\mathsf{M}(P)$ in such cases.


\section{Applications}
\label{sec: applications}

\subsection{Two-Sample Homogeneity Testing} 

The nonparametric two-sample homoegeneity testing problem seeks to detect difference between two distributions on the same measurable (usually Euclidean) space given only samples from them. For $P,Q\in\cP(\R^d)$, let $X_1,\dots,X_n \sim P$ and $Y_1, \dots, Y_m \sim Q$ be independent samples from these distributions, respectively. Based on these samples we wish to test whether $P$ and $Q$ are the same (null hypothesis) or not (alternative), i.e., 
\[H_0: P = Q \ \ \text{versus}\ \ H_1: P \neq Q.\]

For univariate distributions, the standard choice for two-sample homogeneity testing is the Kolmogorov-Smirnov test, which rejects the null hypothesis for large values $\|F_n - G_m\|_\infty$, where $F_n$ and $G_m$ are the corresponding empirical CDFs. Generalizations to multivariate samples take some care to account for the dependence between coordinates, which typically removes the distribution-free property of the limiting statistic. Herein, we propose a multivariate test statistic based on SWD.

Consider a general sequence of 2-sample tests $\phi_{m,n}:(\RR^d)^{n+m}\to\{0,1\}$ defined as 
\[\phi_{m,n}(X_1,\dots,X_n,Y_1,\dots,Y_m) := \begin{cases}
1,&\text{if }D_{m,n} > c_{m,n},\\
0,&\text{otherwise},
\end{cases}\]
where $D_{m,n} := D_{m,n}(X_1,\dots,X_n, Y_1,\dots,Y_m)$ is a sequence of statistics, and $c_{m,n}$ are real numbers. Here, an output of 1 corresponds to rejecting the null. We say that $\phi_{n,m}$ is \textit{asymptotically of level} $\alpha \in [0,1]$, if $\limsup_{m,n \to \infty}\PP(D_{m,n} > c_{m,n}) \leq \alpha$, for all possible distributions of the statistic $D_{m,n}$ under $H_0$. We say that $\phi_{n,m}$ is \textit{consistent at asymptotic level} $\alpha$ if it is asymptotically level $\alpha$ and $\liminf_{m,n \to \infty}\PP(D_{m,n} > c_{m,n}) = 1$ for all possible distributions of the statistic $D_{m,n}$ under $H_1$.

Corollary \ref{cor: 2 sample limit} allows us to use the bootstrap distribution of $W_{m,n}^B$ for setting critical values. This results in a consistent two sample test.

\begin{proposition}[Two-sample testing consistency]
\label{prop: 2 sample test}
Under the setting of Corollary \ref{cor: 2 sample limit}, consider the problem of testing $H_0: P = Q$ versus $H_1: P \neq Q$ using the rule 
\[T_{m,n}(X_1,\dots,X_n,Y_1,\dots,Y_m) := \begin{cases}
1,&\text{if }W_{m,n} > c^B_\alpha,\\
0,&\text{otherwise},
\end{cases}\] where $c^B_\alpha$ is the $(1-\alpha)$-th quantile of $W_{m,n}^B$. Then, this test is consistent at asymptotic level~$\alpha$.
\end{proposition}

The result follows immediately from Corollary \ref{cor: 2 sample limit} and the discussion that follows it. 

\begin{remark}[Comparison to classic Wasserstein testing]
OT distances have been\\ previously considered for both goodness-of-fit \citep{delbarrio2005,hallin2021multivariate} and homogeneity \citep{ramdas2017} testing. These works could only provide critical values for the tests only when $d=1$ due to the lack of a multivariate limit distribution theory for $\mathsf{W}_p$. Here, we showed that adopting SWD as the figure of merit for testing overcomes this bottleneck. 
\end{remark}

\subsection{Minimum Expected Smooth Wasserstein Estimation}
\label{subsec: mede}
Generative modeling concerns learning a parametrized model $\{Q_\theta\}_{\theta \in \Theta}$ that best approximates a data distribution $P$, based on samples from $P$. Proximity between $P$ and $Q_\theta$ is typically measured by a chosen statistical distance, which gives rise to the MDE problem \citep{pollard1980minimum}. Adopting SWD as the distance metric and using the empirical distribution $P_n$ as an estimate for $P$ in the MDE problem gives rise to the minimum SWD estimation (M-SWE) problem
\begin{equation}
\inf_{\theta \in \Theta} \gwass(P_n,Q_\theta),\label{EQ:MDE}
\end{equation}
with an optimal estimate $\hat\theta_n\in\argmin_{\theta\in\Theta}\gwass(P_n,Q_\theta)$. Results on M-SWE appeared in \cite{Goldfeld2020limit_wass}, and are reproduced here in Appendix~\ref{subsec: mde}.

In practice, the MDE problem in \eqref{EQ:MDE} is oftentimes intractable due to complexity of the density function of $Q_\theta$. Indeed, this is the case for many implicit generative models, such as GANs \citep{nowozin2016f,arjovsky2017wasserstein,gulrajani2017improved}. Nevertheless, samples $Y_1, \dots, Y_m$ from the model $Q_\theta$ can be readily generated via simulation in such scenarios. This suggests replacing $Q_\theta$ with its empirical proxy $Q_{\theta, m}:= \frac{1}{m} \sum_{i=1}^n \delta_{Y_i}$ and adopting the minimum expected distance estimation (MEDE) framework \citep{bernton2019parameter, nadjahi2019asymptotic}
\begin{equation}
\label{eq: mede}
\min_{\theta \in \Theta} \EE \big [ \gwass(P_n, Q_{\theta, m}) \,\big|\, X_{1:n} \big ],
\end{equation}
where we assume that the $X_1,\ldots,X_n\sim P$ and $Y_1, \dots, Y_m \sim Q_\theta$ samples are independent, and use the shorthand $X_{1:n} = (X_1, \dots, X_n)$. We term this problem minimum expected SWD estimation (M-ESWE) and denote its solution by
\begin{equation}
\label{eq: mede_sol}
\hat\theta_{m,n} \in \argmin_{\theta \in \Theta} \EE \big [ \gwass(P_n, Q_{\theta, m}) \,\big|\, X_{1:n} \big ].
\end{equation}
The expectations in \eqref{eq: mede} and \eqref{eq: mede_sol} are w.r.t. the samples $Y_1, \dots, Y_m$ from the model. This accounts for the fact that $m$ can typically be as large as we like, in which case one might hope that the conditional expectation is close to $\gwass(P_n,Q_\theta)$ from \eqref{EQ:MDE}. We also use the shorthands M-ESWD for the minimized ESWD in \eqref{eq: mede_sol}, and M-ESWE for the estimator $\hat\theta_{m,n}$ achieving the minimized ESWD.

\medskip 
In what follows, $P$ and $\{Q_\theta\}_{\theta\in\Theta}$ are all assumed to be in $\cP_1(\RR^d)$, and the parameter space $\Theta \subset \R^{d_{0}}$ is assumed to be compact with nonempty interior. We will show existence, measurability, and consistency of the M-ESWE, as well as find the limit distribution of $\hat\theta_{m,n}$, when $m = m(n)$ is chosen so that $m \gg n$. We start from existence and measurability. 

\begin{theorem}[M-ESWE existence and measurability]
\label{thm:M_ESWE_measurable}
Fix $m \in \NN$. Assume the map $\theta \mapsto Q_\theta$ is continuous w.r.t. the weak topology. Then, for every $n \in \NN$, there exists a measurable function $\hat\theta_n(\omega)$ such that:
\[
\hat\theta_n(\omega) \in \argmin_{\theta \in \Theta} \EE \big [ \gwass ( P_n, Q_{\theta, m}) \,\big|\, X_{1:n} \big ](\omega) \neq \emptyset.
\]
\end{theorem}
The proof of Theorem \ref{thm:M_ESWE_measurable} is given in Appendix~\ref{subsec:M_ESWE_measurable_proof}. It relies on extending the lower-semicontinuity property to the expected SWD (Lemma~\ref{lem:eswd_continuity}), and is otherwise similar to the analogous result on M-SWE (Theorem~\ref{thm:MSWD_measurable}).

\medskip
Next, we show that if the number of samples from the parametric model $m$ is chosen so that $m = m_n \to \infty$ as $n \to \infty$, the empirical M-ESWD from \eqref{eq: mede} converges to the population M-SWD, and the associated M-ESWE is also consistent for $\theta^* = \argmin_\theta \gwass(P, Q_\theta)$.

This result requires a stronger continuity assumption on the map $\theta \mapsto Q_\theta$, similar to Assumption 2.7 in \citet{bernton2019parameter}, and Assumption D4 in \cite{nadjahi2019asymptotic}.

\begin{theorem}[Consistency of M-ESWE]
\label{thm:M_ESWE_consistency}
Set $m = m_n \to \infty$, and assume that
(a) the map $\theta \mapsto Q_\theta$ is continuous w.r.t. the weak topology; and (b) if $\theta_n  \to \theta$, then $\EE[ \gwass(Q_{\theta_n}, Q_{\theta_n, m_n})]\to 0$ $\PP$-a.s. Then, the following hold.

\begin{enumerate}[(i)]
    \item $\inf_{\theta \in \Theta} \EE \big [ \gwass ( P_n, Q_{\theta, m_n} ) \,\big|\, X_{1:n} \big ] \to \inf_{\theta \in \Theta} \gwass (P, Q_\theta)$ as $n \to \infty$, $\PP$-a.s.;
    \item For any sequence $\{ \hat{\theta}_{n} \}_{n \in \mathbb{N}}$ of measurable estimators such that 
    \[
    \EE \big [ \gwass (P_{n},Q_{\hat{\theta}_{n}, m_n}) \,\big|\, X_{1:n} \big ] \le \inf_{\theta \in \Theta} \EE \big [ \gwass  (P_n,Q_{\theta, m}) \,\big|\, X_{1:n} \big ] + o_{\PP}(1),
    \] 
    the set of cluster points of $\{ \hat{\theta}_{n} (\omega) \}_{n \in \mathbb{N}}$ is included in $\argmin_{\theta \in \Theta} \gwass (P,Q_{\theta})$ $\PP$-a.s.;
    \item in particular, if $\argmin_{\theta \in \Theta} \gwass (P,Q_{\theta})$ is unique, i.e., $\argmin_{\theta \in \Theta} \gwass (P,Q_{\theta}) = \{ \theta^{\ast} \}$, then $\hat{\theta}_{n} \to \theta^{\ast}$ a.s..
\end{enumerate}

\end{theorem}

\begin{remark}[Continuity of $\theta \mapsto Q_\theta$]
The extra assumption (b) of Theorem~\ref{thm:M_ESWE_consistency} is weaker than Assumption 2.7 of \citet{bernton2019parameter} for $m_n = n$ since $\gwass(P,Q) \leq \wass(P, Q)$. A sufficient condition that ensures both (b) and Assumption 2.7 is $\sup_{\theta\in\Theta} \int |x|^q\,\dd Q_\theta < \infty$ for some $q > 1$. This allows bounding $\EE\big[ \wass(Q_{\theta_n}, Q_{\theta_n, n})\big]$ uniformly over $\theta$ via Theorem 1 of \cite{FournierGuillin2015}. Alternatively, assuming that the map $\theta \mapsto Q_\theta$ is continuous in $\gwass$ also ensures (b).
\end{remark}

The proof of Theorem~\ref{thm:M_ESWE_consistency} is given in Appendix~\ref{subsec:M_ESWE_consistency_proof}. Statement (i) can be established using epi-covergence of the map $\theta \mapsto \EE \big [ \gwass(P_n, Q_{\theta, m_n}) \big ] $, similar to  \citet[Theorem~3]{nadjahi2019asymptotic}. The last two statements follow directly from (i).

\medskip

Turning to limit distribution results, results are presented for the `well-specified' scenario, i.e., when $P=Q_{\theta^\ast}$ for some $\theta^\ast$ in the interior of $\Theta$. 
The limit distribution result for M-SWE (Theorem~\ref{thm:MSWD_inf_CLT} in Appendix ) was proven in \citet{Goldfeld2020limit_wass} via an adpatation of the norm differentiability based arguments of \citet{pollard1980minimum}. Combining Theorem~\ref{thm:MSWD_inf_CLT} with an approximation argument produces a limit distribution for the M-ESWE error in the regime $m = m_n \gg n$. To simplify the proof (given in Appendix~\ref{subsec:MESWD_proof}), a uniform moment bound over the parametric family $Q_\theta$ is assumed to hold. This result is stated below.

\begin{theorem}[M-ESWD limit distribution]
\label{thm:MESWD_inf_CLT}
Let $P$ satisfy the conditions of Theorem~\ref{thm: CLT for W1}. In addition, suppose that:
\begin{enumerate}[(i)]
    \item the map $\theta \mapsto Q_{\theta}$ is continuous relative to the weak topology;
    \item $P \ne Q_{\theta}$ for any $\theta \ne \theta^{\ast}$;
    \item there exists a vector-valued functional $D^{\mspace{1mu}\sigma}\mspace{-2mu}\in\mspace{-2mu} (\ell^\infty(\mathsf{Lip}_{1,0}))^{d_0}$ such that $\lip{Q^{\mspace{1mu}\sigma}_\theta\mspace{-2mu}-Q^{\mspace{1mu}\sigma}_{\theta^\ast}\mspace{-2mu}-\mspace{-2mu}\langle\theta-\theta^\ast\mspace{-3mu},\mspace{-2mu}D^{\mspace{1mu}\sigma}\rangle}\mspace{-2mu}$ $=o(\|\theta-\theta^\ast\|)$ as $\theta \to \theta^\ast$, where $\langle t,D^{\mspace{1mu}\sigma} \rangle:=\sum_{i=1}^{d_0} t_iD_i^{\mspace{1mu}\sigma}$ for $t\in\RR^{d_0}$;
    \item the derivative $D^{\mspace{1mu}\sigma}$ is \textit{nonsingular} in the sense that $\langle t,D^{\mspace{1mu}\sigma} \rangle\neq 0$, i.e., $\langle t,D^{\mspace{1mu}\sigma} \rangle\in\ell^\infty(\mathsf{Lip}_{1,0})$ is not the zero functional for all $0\neq t\in\RR^{d_0}$;
    \item $m = m_n \gg n$;
    \item $\sup_{\Theta} \int |x|^q\,\dd Q_\theta < \infty$ for some $q > 2(d+1)$.
\end{enumerate}
Then,
\[
\sqrt{n}\inf_{\theta\in\Theta} \EE \big [ \gwass(P_n,Q_{\theta, m_n}) \, \big| \, X_{1:n} \big ] \stackrel{d}{\to} \inf_{t \in \RR^{d_0}}\lip{G_P^{\mspace{1mu}\sigma}-\ip{t,D^{\mspace{1mu}\sigma}}},
\]
where $G_{P}^{\mspace{1mu}\sigma}$ is the Gaussian process from Theorem \ref{thm: CLT for W1}. 
\end{theorem}

Finally, we establish a limit distribution for the M-ESWE solution. The proof of Theorem~\ref{thm:MESWD_inf_CLT} shows that under certain assumptions, solution to the M-ESWE problem are also approximate solutions to the M-SWE problem. Combining this fact with a result on the limit distribution of approximate solutions to the M-SWE problem (Corollary~\ref{cor:unique_solution}), we have the following result (see Appendix~\ref{subsec:M_ESWE_limit_proof} for a proof).

\begin{theorem}[M-ESWE limit distribution]
\label{thm:M_ESWE_limit}
Under the conditions of Theorem~\ref{thm:MESWD_inf_CLT}, let $\hat\theta_n$ be a sequence 
of measurable estimators for $\theta^*$ satisfying
\[ \EE \big [ \gwass(P_n, Q_{\hat\theta_n, m_n}) \,\big|\, X_{1:n} \big ] \leq \inf_{\theta \in \Theta} \EE \big [ \gwass(P_n, Q_{\theta, m_n}) \,\big|\, X_{1:n} \big ] + o_{\PP}(n^{-1/2}). \]
Then, provided that $\argmin_{t \in \R^{d_{0}}} \lip{G_P^{\mspace{1mu}\sigma}-\ip{t,D^{\mspace{1mu}\sigma}}}$ is unique a.s., we have $\sqrt{n}(\hat{\theta}_{n} - \theta^*) \stackrel{d}{\to} \argmin_{t \in \R^{d_{0}}}\lip{G_P^{\mspace{1mu}\sigma}-\ip{t,D^{\mspace{1mu}\sigma}}}$.
\end{theorem}

\section{Proofs of Main Results}

This section presents proofs of our main results (Theorems \ref{thm: CLT for W1}, \ref{thm: SWD small sigma Berry-Esseen type rate}, \ref{thm: Empirical Bootstrap Berry-Esseen type rate} and \ref{thm: low dim rate for exp}). Derivations of auxiliary lemmas used in these proofs are deferred to the appropriate Appendix sections. Due to the technical nature and length of the proof of Theorem \ref{thm: Empirical Bootstrap Berry-Esseen type rate}, herein we only provide an outline of the argument and defer the full details to the Appendix.

\subsection{Proof of Theorem \ref{thm: CLT for W1}}\label{SUBSEC:CLT for W1 proof}
Recall that $\varphi_{\sigma}$ is the density function of $\cN(0,\sigma^2 \mathrm{I}_d)$, i.e., $\varphi_{\sigma}(x) = (2\pi \sigma^{2})^{-d/2} e^{-\|x\|^{2}/(2\sigma^{2})}$ for $x \in \R^{d}$. 
Since $P_{n} \ast \Gauss$ has a density given by
$x \mapsto n^{-1} \sum_{i=1}^{n} \gauss (x-X_{i}),$ we arrive at the expression
\begin{equation}
\gwass (P_{n},P) = \sup_{f \in \mathsf{Lip}_{1}} \left [ \frac{1}{n} \sum_{i=1}^{n} f\ast\gauss (X_{i}) - Pf\ast\gauss \right ]. 
\label{eq: smooth Wasserstein}
\end{equation}
Observing that the RHS of \eqref{eq: smooth Wasserstein} does not change if we replace $f$ with $f-f(x^{\ast})$ for any fixed $x^{\ast}$,
the problem boils down to showing that the function class $\cF_\sigma := \{f\ast\gauss:f \in\mathsf{Lip}_{1,0}\}$ is $P$-Donsker.

To that end we next derive a bound on the H\"older norm of $f_\sigma:=f\ast\gauss$, for $f\in\mathsf{Lip}_{1,0}$. The Donsker property will follow by metric entropy bounds for H\"older balls. For a vector $k = (k_1,\dots,k_d)$ of $d$ nonnegative integers, define the differential operator
\[
D^{k} = \frac{\partial^{|k|}}{\partial x_{1}^{k_{1}} \cdots \partial x_{d}^{k_{d}}},
\]
with $|k| = \sum_{i=1}^{d} k_{i}$. 
For a function $f: \mathcal{X} \to \R$ defined on a bounded set $\mathcal{X}\subset\R^d$, the H\"{o}lder norm of order $\alpha > 0$ is 
\[
\| f \|_\alpha := \| f \|_{\alpha,\mathcal{X}} := \max_{0 \leq |k| \leq \underline\alpha} \sup_{x \in \cX} \big|D^k f(x)\big| + \sup_{\overset{|k| = \underline\alpha}{x,y \in \mathcal{X}}} \frac{\big|D^k f_\sigma(x) - D^k f_\sigma(y)\big|}{\|x-y\|^{\alpha-\underline\alpha}},
\]
where $\underline\alpha$ is the greatest integer strictly smaller than $\alpha$ and the suprema are taken over the interior of $\mathcal{X}$. We have the following lemma (see Appendix \ref{SUBSEC: limit variable supp proofs} for the proof). 

\begin{lemma}[H\"older norm bound]
\label{lem: fractional derivatives}
For any $\alpha \ge 1$, $0 < \sigma \le 1$, $f \in \mathsf{Lip}_{1,0}$, and a bounded convex $\mathcal{X}\subset\RR^d$, we have
$\| f_\sigma\|_{\alpha,\cX}  \lesssim_{\alpha,d} \sigma^{-\alpha+1}\big(1 \vee \sup_{\cX} \|x\| \big)$. 
\end{lemma}

The following metric entropy bound from \citet[Theorem 2.7.1]{VaWe1996} will be subsequently used.

\begin{lemma}[Metric entropy bound for H\"{o}lder ball]
\label{lem: metric entropy}
Suppose $\mathcal{X}\subset\R^d$ is bounded and convex with nonempty interior. For $\alpha \geq 1$ and $M > 0$, let $C^{\alpha}_M(\cX)$ be the set of continuous functions $f:\cX\to\R$ with $\|f\|_\alpha \leq M$.
Then, for any $0 < \epsilon \le 1$, we have 
$\ \log N \big(\epsilon M,C_{M}^{\alpha}(\mathcal{X}), \| \cdot \|_{\infty} \big) \lesssim_{d,\alpha,\diam (\mathcal{X})} \epsilon^{-d/\alpha}$.
\end{lemma}

We are now in position to prove Theorem \ref{thm: CLT for W1}, which applies Theorem 1.1 in \citet{VaWe1996} to the function class $\cF_\sigma$ to show that it is $P$-Donsker. We begin with noting that the function class $\cF_\sigma$ has envelope $F(x):= \|x\| + \sigma \sqrt{d}$, which is $P$-integrable as required under the assumed moment conditions. Indeed, for any $f \in \mathsf{Lip}_{1,0}$, since $|f(y)| \le |f(0)| + \|y\| =\| y \|$, we have
\[
\begin{split}
|f_{\sigma}(x)| &\le \int \| y \| \gauss(x-y) \dd y \le \int \big(\| x \| + \| x-y \|\big) \gauss (x-y) \dd y \\
&\le \| x  \| + \int \| y \| \gauss (y) \dd y \le \| x \| + \left ( \int_{\R^{d}} \| y \|^{2} \gauss(y) \dd y \right )^{1/2} = \| x \| + \sigma \sqrt{d}. 
\end{split}
\]

Next, for each $j$, consider the restriction of $\cF_\sigma$ to $I_{j}$, denoted as ${\cF}_{j}= \{ f \mathds{1}_{I_{j}} : f \in {\cF_\sigma}\}$.
To invoke \citet[Theorem 1.1]{VaWe1996}, we have to verify that each $\cF_{j}$ is $P$-Donsker and to bound $\E[\| \GG_n \|_{{\cF}_{j}}]$, where $\GG_n := \sqrt{n}(P_{n}-P)$.
In view of Lemma~\ref{lem: fractional derivatives},  ${\cF}_{j}$ can be regarded as a subset of $C_{M'_j}^{\alpha}(I_{j})$ with $\alpha > d/2$ and $M_{j}' =\sigma^{-\alpha+1} C_{\alpha,d} \sup_{I_j}\|x\|$. 
Lemma~\ref{lem: metric entropy} then implies that the $L^{2}(Q)$-metric entropy of ${\cF}_{j}$, for any probability measure $Q\in\cP(\R^{d})$, can be bounded~as 
\[
\log N\left(\epsilon M_{j}' Q(I_{j})^{1/2}, {\cF}_{j}, L^{2}(Q) \right) \lesssim_{\alpha,d,K} \epsilon^{-d/\alpha},
\]
where $K = \sup_{j} \diam(I_j) < \infty$. For any choice of $\alpha > d/2$, the square root of the RHS is integrable (w.r.t. $\epsilon$) around $0$, so that $\cF_{j}$ is $P$-Donsker by Theorem 2.5.2 in \citet{VaWe1996}. By Theorem 2.14.1 in \citet{VaWe1996} we further obtain
\be\label{eq: moment}
\E[\| \GG_n \|_{{\cF}_{j}}] \lesssim_{\alpha,d,K} M_{j}' P(I_{j})^{1/2} \lesssim_{\alpha,d} \sigma^{-\alpha + 1}M_{j} P(I_{j})^{1/2},
\ee
with $M_{j} = 1\vee\sup_{I_{j}}\|x\|$. By assumption, the RHS is summable over $j$. 

By Theorem 1.1 in \citet{vanderVaart1996} we conclude that ${\cF_\sigma}$ is $P$-Donsker, which implies that there exists a tight version of $P$-Brownian bridge process $G_{P}$ in $\ell^{\infty}({\cF_\sigma})$ such that $\big(\GG_n(f)\big)_{f \in \cF_\sigma}$ converges weakly in $\ell^{\infty}({\cF_\sigma})$ to $G_{P}$. 
Finally, the continuous mapping theorem yields that
\[\sqrt{n} \gwass (P_{n},P) = \| \GG_n\|_{{\cF}_\sigma}\stackrel{w}{\to} \| G_{P} \|_{{\cF}_\sigma}= \big\|G_{P}^{\mspace{1mu}\sigma}\big\|_{\mathsf{Lip}_{1,0}},
\]
where $G_{P}^{\mspace{1mu}\sigma} (f):=G_{P} (f \ast \gauss)$. By construction, the Gaussian process $\big(G_{P}^{\mspace{1mu}\sigma}(f)\big)_{f \in \mathsf{Lip}_{1,0}}$ is tight in $\ell^{\infty}(\mathsf{Lip}_{1,0})$. The moment bound for $\alpha > d/2$ follows from summing up the moment bounds for~${\cF}_{j}$.

\subsection{Proof of Theorem \ref{thm: SWD small sigma Berry-Esseen type rate}}
\label{SUBSEC: SWD Berry-Esseen type rate proofs}

The statistic $\sqrt{n}\gwass(P_n,P)$ and its limit $L_P^{\mspace{1mu}\sigma}$ are, respectively, the suprema of the empirical process and its Gaussian process limit over the function class $\cF_\sigma$. Given a finite $L^2(P)$ $\delta$-net of functions covering $\cF_\sigma$, the difference in the distribution of these suprema can be bound by their difference w.r.t. the $\delta$-net plus approximation error terms. Formally, if $f_1, \dots, f_{N_{\sigma, \delta}}$ is a minimal $\delta$-net as above, where $N_{\sigma, \delta} = N\big(\delta, \cF_\sigma, L^2(P)\big)$, then, for any set $A \in \cB(\RR)$,
\begin{align*}\label{eq: berry esseen decomposition}
    \PP\Big(&\,\|\GG_n\|_{\cF_\sigma} \in A\Big) - \PP\Big(\|G_P\|_{\cF_\sigma} \in A^{3\eta}\Big)\\
    &\leq \underbrace{\PP\Big(\big|\|G_P\|_{\cF_\sigma} - \max_{1\leq i \leq N_{\sigma, \delta}} G_P(f_i)\big| > \eta \Big)}_{\text{(I)}} + \underbrace{\PP\Big(\big|\|\GG_n\|_{\cF_\sigma} - \max_{1\leq i \leq N_{\sigma, \delta}} \GG_n(f_i)\big| > \eta \Big)}_{\text{(II)}} \\
    &\qquad\quad+ \underbrace{\PP\Big(\max_{1\leq i \leq N_{\sigma, \delta}} \GG_n(f_i) \in A^{\eta}\Big) - \PP\Big(\max_{1\leq i \leq N_{\sigma, \delta}}G_P(f_i) \in A^{2\eta}\Big)}_{\text{(III)}}.\numberthis
\end{align*}

We wish to translate this inequality into a bound on the difference $V_{\sigma,n} - W_{\sigma,n}$, where $\big(V_{\sigma,n},W_{\sigma,n}\big)$ are a coupling of $\|\GG_n\|_{\cF_\sigma}$ and $\|G_P\|_{\cF_\sigma}$. Strassen's theorem allows us to ensure the existence of such couplings. (see \citet[Lemma 4.1]{ChChKa2014} for the version stated below).

\begin{lemma}[Strassen's theorem]\label{lem: Strassen's theorem}
Let $X,Y$ be real random variables. If $\PP(X \in A) \leq \PP(Y \in A^\delta) + \epsilon$ for all $A \in \calB(\RR)$, then there exists a coupling $(\widetilde{X},\widetilde{Y})$ of $X$ and $Y$ such that $\PP\big(|\widetilde{X} - \widetilde{Y}| > \delta\big) \leq \epsilon$. 
\end{lemma}
In fact, we will derive an $O_\PP$ rate for the difference $|V_n-W_n|$, where $V_n:= V_{n,\sigma_n,\delta_n,\eta_n}$ and $W_n:= W_{n,\sigma_n,\delta_n,\eta_n}$, for a smoothing parameter $\sigma=\sigma_n$ that is allowed to decay at a slow enough rate. To that end, first observe that for each fixed $\sigma > 0$, the class $\cF_\sigma$ is pre-Gaussian, which implies the existence of a tight version of the limit process $\big(G_P(f)\big),\ f\in\cF_\sigma$ and ensures measurability of $\|G_P\|_{\cF_\sigma}$. Since $\cF_\sigma$ is separable w.r.t. pointwise convergence, $\|\GG_n\|_{\cF_\sigma}$ is also measurable.

Thus, if we establish a non-asymptotic (in $n$) bound on the RHS of \eqref{eq: berry esseen decomposition} and characterizes the exact dependence on $\sigma$, we may set $\sigma = \sigma_n$ and still obtain $(V_n,W_n)$ on $(\Omega,\cA,\PP)$, whose difference is bounded in probability. Combining the bounds for each $n$, we get an $O_{\PP}$ rate for $|V_n - W_n|$. The rate bound on the Prokhorov distance follows more directly from \eqref{eq: berry esseen decomposition} without going through Lemma~\ref{lem: Strassen's theorem}.

We begin with a few technical lemmas that will be required throughout the proof; derivations of the lemmas are given in Appendix \ref{SUBSEC: SWD Berry-Esseen Rate supp proofs}. The first is a bound on $L^2$ entropy classic of functions with unbounded support. This provides control on $N_{\sigma,\delta}$ for a given discretization error level $\delta$.

\begin{lemma}[Example 19.9 in \citet{vanderVaart1998_asymptotic}]\label{lem: unbdd_l2_entropy}
\label{lem:partition_entropy}
Let $\{I_j\}_{j=1}^\infty$ be an enumeration of unit cubes in $\RR^d$ with vertices on integer lattice points. If $\cF$ is a function class such that for any $j\in\NN$ the restriction ${\cF}_{j}= \{ f \mathds{1}_{I_{j}} : f \in {\cF}\}$ belongs to the H\"{o}lder class $C^\alpha_{M'_j}(I_j)$, for $\alpha > 0$ and $M_j>0$, i.e., $\cF_j \subset C^\alpha_{M'_j}(I_j)$, then
\be\label{eq: L2P bracketing entropy bound}\log N_{[\,]}(\delta, \cF, L^2(P)) \lesssim_{\alpha,d} \left(\frac{1}{\delta}\right)^\frac{d}{\alpha} \left[\sum_{j=1}^\infty \big((M'_j)^2 P(I_j)\big)^\frac{d}{d+2\alpha}\right]^\frac{d+2\alpha}{2\alpha}.\ee
\end{lemma}

The next lemma relates the probability law of the maximum of (centered) i.i.d. random vectors to that of the maximum of i.i.d. Gaussians with same covariance matrices. This will enable bounding term (III) in \eqref{eq: berry esseen decomposition}.

\begin{lemma}[Theorem 3.1 in \citet{ChChKa2016} simplified]
    \label{lem: gauss_approx_discretized}
    Suppose that $X_1, \dots, X_n$ are i.i.d. zero mean $\RR^p$-valued random vectors $(p \geq 2)$ with covariance matrix $\Sigma$ and coordinate-wise finite absolute third moments, i.e., $\EE\big[|X_{ij}|^3\big] < \infty$, for all $1 \leq i \leq n$ and $1 \leq j \leq p$. Consider the statistic $Z:= \max_{1\leq j\leq p} \frac{1}{\sqrt{n}}\sum_{i=1}^n X_{ij}$. Let $Y\sim\cN(0,\Sigma)$ and define $\widetilde{Z}:= \max_{1\leq j\leq p} Y_j$. Then, for every $\eta > 0$ and $A\in\cB(\RR)$, we have
    \[\PP(Z \in A) \leq \PP(\widetilde{Z} \in A^{C_7 \eta}) + \frac{C_8 (\log  p)^2}{\eta^3 \sqrt{n}}\big(L_n + M_{n,X}(\eta) + M_{n,Y}(\eta)\big),\]
    where $C_7, C_8$ are universal positive constants, $L_n = \max_{1\leq j \leq p} \EE\left[|X_{ij}|^3\right]$, and
    \begin{align*}
    &M_{n,X}(\eta) =  \EE\left[\max_{1\leq j \leq p} |X_{ij}|^3 \ind_{\left\{\max_{1\leq j \leq p} |{X}_{ij}| > \eta\sqrt{n}/\log p\right\}}\right],\\
    &M_{n,Y}(\eta) = \EE\left[\max_{1\leq j \leq p} |Y_{j}|^3 \ind_{\left\{\max_{1\leq j \leq p} |{Y}_{j}| > \eta\sqrt{n}/\log p\right\}}\right].
    \end{align*}
\end{lemma}

We also use a Fuk-Nagev type inequality (Theorem 2, \citet{Adamczkak2010}) for empirical process concentration to bound the term (II) in \eqref{eq: berry esseen decomposition}.

\begin{lemma}[Fuk-Nagev type inequality]\label{lem: fuk-nagev}
Let $\cF$ be a function class $f : S \to \RR$ with a measurable envelope $F$, that is separable w.r.t. pointwise convergence, and let $P\in\cP(S)$. Suppose that $P f=0$ for all $f\in\cF$, and that $F \in L^{p\vee 2}(P)$. 
For $X_1,\ldots,X_n$ i.i.d. according to $P$ define $Z:=\sup_{f\in\cF}\big|\sum_{i=1}^nf(X_i)\big|$ and let $\sigma^2 > 0$ be any positive constant such that $\sigma^2 \geq \sup_{f \in \cF} Pf^2$. Then 
\[\PP\big(Z \geq (1 + \alpha)\EE[Z] + x\big) \leq e^{-cx^2/(n\sigma^2)} + \frac{c'\EE[M^p]}{x^p},\quad\forall\, \alpha,x > 0,\]
where $M := \max_{1\leq i \leq n} F(X_i)$ and $c,c'$ are positive constants that depend only on $p$ and $\alpha$.
\end{lemma}

Finally, we also use the following relation between moments of random variables and discrete sums approximating such moments.
\begin{lemma}[Sufficient moment condition] \label{lem: moment inequality}
Let $\{ \mspace{-1.5mu} I_j \mspace{-1.5mu} \}_{j=1}^\infty$ be an enumeration of unit cubes in $\RR^d$~with vertices on integer lattice points, disjointified such that they form a partition of $\RR^d$. If $P\|x\|_\infty^l = M < \infty$ for any $l > \frac{(d-1)(1-p_2) + p_1 +1}{p_2}$, then we have $\sum_{j=1}^\infty \big(\sup_{I_j}\|x\|\big)^{p_1} P(I_j)^{p_2}$ $\lesssim_{l,d,p_1,p_2}  M^{p_2} < \infty$.
\end{lemma}

We are now ready to prove Theorem \ref{thm: SWD small sigma Berry-Esseen type rate}.
\begin{proof}[Proof of Theorem \ref{thm: SWD small sigma Berry-Esseen type rate}]

\noindent\textbf{Proof of (i):} First, let $\sigma,\delta\in(0,1]$ be arbitrary; these parameters will be specified later. Recall that we fix a smoothness parameter $\alpha > d/2$ and let $a = \alpha - 1,\,b = d/(2\alpha)$. Note that, if $f_1, \dots, f_{N_{\sigma, \delta}}$ is a minimal $L^2(P)$ $\delta$-net of $\cF_\sigma$, then
\be\label{eq: discretization error bounds}
\begin{split}
\left|\|\GG_n\|_{\cF_\sigma} - \max_{1\leq i \leq N_{\sigma, \delta}} \GG_n (f_i)\right| &\leq \sup_{\substack{f,g \in \cF_\sigma: \\ \|f - g\|_{L^2(P)} < \delta}} \big|\GG_n(f-g)\big|,\\
\left|\|G_P\|_{\cF_\sigma} - \max_{1\leq i \leq N_{\sigma, \delta}}G_P (f_i)\right| &\leq \sup_{\substack{f,g \in \cF_\sigma: \\ \|f - g\|_{L^2(P)} < \delta}} \big|G_{P} (f) - G_{P} (g)\big|.
\end{split}
\ee

For the second inequality, we extend $G_P$ to the linear span of $\cF_\sigma$ so that its sample paths are linear, and the RHS may be written as a supremum of $|G_P(f-g)|$. We split the proof into the following steps:
\medskip

\noindent\underline{Step 1} (Bound on metric entropy). Set $K = \sum_{j=1}^\infty \big(M_j^2 P(I_j)\big)^{d/(d+2\alpha)}\big)$. Invoking Lemma~\ref{lem:partition_entropy} with $M_{j}' =\sigma^{-\alpha+1} C_{\alpha,d} \sup_{I_j}\|x\|$ (which is a feasible choice by Lemma~\ref{lem: fractional derivatives}), we obtain 
\be\label{eq: L2P metric entropy bound}\log N\big(\delta, \cF_{\sigma}, L^2(P)\big) \lesssim_{\alpha,d} K \sigma^{-d\frac{\alpha-1}{\alpha}} \delta^{-\frac{d}{ \alpha}} \lesssim_{\alpha,d,M,l} \sigma^{-2ab}\delta^{-2b},\ee
where $K \lesssim_{\alpha,d} M^{d/(d+2\alpha)}$ by Lemma~\ref{lem: moment inequality}.
\medskip

\noindent\underline{Step 2} (Empirical process discretization error). 
Let $\{I_j\}_{j=1}^\infty$ be a suitably disjointed enumeration of unit cubes on integer lattice points in $\R^d$. Define $M_j$ and $M'_j$ as in Theorem~\ref{thm: CLT for W1}, and let $\cF_\sigma(\delta):= \big\{f-g: f,g \in \cF_\sigma,\, \|f - g\|_{L^2(P)} < \delta\big\}$. Note that $\EE\big[\|\GG_n\|_{\cF_\sigma(\delta)}\big] \leq \sum_{j=1}^\infty \EE\big[\|\GG_n\|_{{\cF_j}(\delta)}\big]$, where ${\cF_j}(\delta)= \{f\ind_{I_j}: f \in \cF_\sigma(\delta)\}$. Also set ${\cF}_j = \{f\ind_{I_j}:f \in {\cF_\sigma}\}$. Clearly, $F_j(x) = M'_j\ind_{I_j}(x)$ is an envelope function for ${\cF}_j$. Then, using Lemma~\ref{lem: metric entropy} with $M = M'_j$ and $N = \alpha$ (which are feasible choices by Lemma~\ref{lem: fractional derivatives}), we have, for any probability measure $Q$ on a finite set, $\log N\big(\delta M'_j Q(I_j)^{1/2}, {\cF}_j, L^2(Q)\big) \lesssim_{\alpha,d} \delta^{-2b}$. Observing that
\be \label{eq:difference class entropy}N\left(2\delta M'_j Q(I_j)^{1/2}, {\cF}_j(\delta), L^2(Q)\right) \leq N\left(\delta M'_j Q(I_j)^{\frac{1}{2}}, {\cF}_j, L^2(Q)\right)^2,\ee
we obtain $\log N\big(2\delta M'_j Q(I_j)^{\frac{1}{2}}, \cF_j(\delta), L^2(Q)\big) \lesssim_{\alpha, d} \delta^{-2b}$.

For the difference class $\cF_j(\delta)$, $2F_j$ is an envelope function. Next, we compute the entropy integral $J(\delta, {\cF}_j(\delta), 2F_j)$ as
\be \label{eq:uniform_entropy_integral}J(\delta, \cF_j, 2F_j) \mspace{-3mu} := \mspace{-3mu} \int_0^\delta \mspace{-8mu} \sup_{Q \in \cP_{\mathsf{f}}(\R^d)} \sqrt{\log N\left(\epsilon\|2F_j\|_{L^2(Q)}, \cF_j(\delta), L^2(Q)\right)} \dd\epsilon \lesssim_{\alpha,d} \mspace{-3mu}\int_0^\delta \mspace{-4mu} \epsilon^{-b} \dd \epsilon \lesssim_{\alpha,d} \delta^{1-b},\ee
where $\cP_{\mathsf{f}}(\R^d)\subset\cP(\RR^d)$ denotes the set of discrete measures on finitely many points in $\R^d$. The local maximal inequality for uniform entropy \citep[Theorem 5.2]{ChChKa2014} then implies
\be\EE\big[\|\GG_n\|_{\cF_j(\delta)}\big] \lesssim_{\alpha,d} \left(M'_j P(I_j)^\frac{1}{2}\right)^b \delta^{1 - b} + \frac{1}{\sqrt{n}} (M'_j)^{1 + 2b}P(I_j)^b\delta^{-2b}.\ee

By Lemma~\ref{lem: moment inequality}, $\sum_{j=1}^\infty \big(M_j P(I_j)^{1/2}\big)^b \lesssim_{\alpha,d,l} M^{b/2} < \infty$ and $\sum_{j=1}^\infty M_j^{1+2b} P(I_j)^b \lesssim_{\alpha,d,l} M^b$. Hence, we sum over $j$ while using the fact that $M'_j\lesssim_{\alpha,d} \sigma^{-a} M_j$ to obtain  

\be\label{eq:emp_error_discretized}
\EE\big[\|\GG_n\|_{\cF_{\sigma}(\delta)}\big] \lesssim_{\alpha,d}\mspace{-3mu} \sum_{j=1}^\infty \left[\big(M'_j P(I_j)^\frac{1}{2}\big)^b \delta^{1 - b}\mspace{-5mu}+\mspace{-3mu} \frac{(M'_j)^{1 + 2b}P(I_j)^b}{\delta^{2b}\sqrt{n}} \right] \mspace{-5mu}\lesssim_{\alpha,d,M,l} \sigma^{-ab} \delta^{1-b} + \frac{\sigma^{-a(1+2b)} }{\delta^{2b}\sqrt{n}}.\mspace{-3mu}
\ee
\medskip

\noindent\underline{Step 3} (Gaussian process discretization error). 
Recall that $G_P$ is a mean zero Gaussian process on $\cF_\sigma$ with covariance kernel $\Cov\big(G_P(f),G_P(g)\big)=\Cov_P(f,g)$. 
By Dudley's theorem (see, for e.g., Theorem 11.17, \citet{LeTa1991}), modulo arguments relating the entropy of the difference class $\cF_\sigma(\delta)$ with that of $\cF_\sigma$ (as in \eqref{eq:difference class entropy}), we have
\be\label{eq:gaussian process discretization error}\EE\left[\|G_{P}\|_{\cF_\sigma(\delta)}\right] \lesssim \int_{0}^\delta \sqrt{\log N\big(\epsilon, \cF_\sigma, L^2(P)\big)} \dd\epsilon \lesssim_{\alpha,d, M,l} \sigma^{-ab}\delta^{1-b}.\ee
\medskip

\noindent\underline{Step 4} (Comparison between $\GG_n$ and $G_P$ after discretization). 
Let $f_1, \dots, f_{N_{\sigma, \delta}}$ be a minimal $L^2(P)$ $\delta$-net of $\cF_\sigma$ and define $Z = \max_{1\leq j \leq N_{\sigma,\delta}} \GG_n(f_j)$ and $\tilde{Z} = \max_{1\leq j\leq N_{\sigma,\delta}} G_{P} (f_j)$. Invoking Lemma~\ref{lem: gauss_approx_discretized}, the rest of this step focuses on bounding $L_n$, $M_{n,X}(\eta)$, and $M_{n,Y}(\eta)$. Let $X \sim P$. For the first two, set $X_{ij} = f_j(X_i) - Pf_j$ and note that $\EE\big[|X_{ij}|^3\big] \leq \EE\big[(\|X\| + \sigma\sqrt{d}\,)^3\big] \leq \EE\big[(\|X\| + \sqrt{d}\,)^3\big]$.

This gives us
\be\label{eq: L_bound} L_n \leq \max_{1\leq j \leq N_{\sigma,\delta}}\EE\big[|X_{ij}|^3\big] \leq \EE\left[\big(\|X\| + \sqrt{d}\,\big)^3\right] \lesssim_d M^{3/l}, \ee
\be\label{eq: Mnx_bound} M_{n,X}(\eta) \leq \EE\left[\max_{1\leq j \leq N_{\sigma,\delta}}|X_{ij}|^3\right] \leq \EE\left[\big(\|X\| + \sqrt{d}\,\big)^3\right] \lesssim_d M^{3/l}. \ee

We next bound $M_{n,Y}(\eta)$. Here $Y_{j}$ are $\cN(0,\sigma_j^2)$, where $\sigma_j^2 \leq \E\big[ (\|X\|+ \sigma\sqrt{d})^2\big]$, for all $1 \leq j \leq p$. Then, the following concentration inequality and moment bound for Gaussian random variables hold \citep[Example 2.1.19]{GiNi2016}:
\begin{align*}
    \PP\left(\max_{1\leq j \leq N_{\sigma, \delta}} |Y_{j}| > \EE\left[\max_{1\leq j \leq N_{\sigma, \delta}} |Y_{j}|\right] + t\right) &\leq  \exp\left(-\frac{t^2}{2\E\left[ \big(\|X\|+ \sigma\sqrt{d}\big)^2\right]}\right),\\
    \EE\left[\max_{1\leq j \leq N_{\sigma, \delta}} |Y_{j}|\right] &\leq \sqrt{2 \E\left[ \big(\|X\|+ \sigma\sqrt{d}\big)^2\right] \log (2N_{\sigma,\delta})}.
\end{align*}
Setting $l(n, \eta) = \eta\sqrt{n}/\log N_{\sigma, \delta} - \sqrt{2 \E\big[ (\|X\|+ \sigma\sqrt{d})^2\big] \log 2N_{\sigma,\delta}}$, we have for $l(n,\eta) > 0$,
\be\begin{split}\label{eq: Mny_bound}
M_{n,Y}(\eta) &= \EE\left[\max_{1\leq j \leq p} |Y_{j}|^3 \ind_{\left\{\max_{1\leq j \leq p} |Y_{j}| > \eta\sqrt{n}/\log N_{\sigma,\delta}\right\}}\right]\\
&\lesssim \int_{l(n, \eta)}^\infty t^2\exp\left(-\frac{t^2}{2\E\left[ \big(\|X\|+ \sigma\sqrt{d}\big)^2\right]}\right)\,\dd t\\
&\leq \int_{0}^\infty t^2\exp\left(-\frac{t^2}{2\E\left[ \big(\|X\|+ \sigma\sqrt{d}\big)^2\right]}\right)\,\dd t\ \lesssim_{d,M} 1.
\end{split}\ee

Hence, from \eqref{eq: L_bound}-\eqref{eq: Mny_bound}, for every $A\in\cB(\RR)$ and every $\eta > 0$ such that $l(n, \eta) > 0$, we have
\be\label{eq: gaussian comparison}\PP(Z \in A) < \PP\big(\tilde{Z} \in A^\eta\big) + \frac{C_{\alpha,d,M,l} (\log N_{\sigma, \delta})^2}{\eta^3 \sqrt{n}}, \ee
for a constant $C_{\alpha,d,M,l}$ depending only on $\alpha$, $d$, $l$, and $M$.
\medskip

\noindent\underline{Step 5} (Strassen's theorem).
By \eqref{eq: gaussian comparison} and Markov's inequality, we have (as in \eqref{eq: berry esseen decomposition}), for each $\eta\in(0,1]$ and $n$ such that $l(n, \eta) > 0$,
\begin{align*}
    &\PP\left(\sqrt{n}\gwassemp \in A\right)\\
    &\quad\leq \PP\left(L^{\mspace{1mu}\sigma}_P \in A^{(2+C_7)\eta}\right) + \PP\Big(\|\GG_n\|_{\cF_{\sigma}(\delta)} > \eta\Big)+ \PP\Big(\|G_{P}\|_{\cF_{\sigma}(\delta)} > \eta\Big) + \frac{C_{\alpha,d,M,l} (\log N_{\sigma,\delta})^2}{\eta^3 \sqrt{n}}\numberthis\label{eq: gwassemp set prob bound}\\
    &\quad\leq \PP\Big(L^{\mspace{1mu}\sigma}_P \in A^{(2+C_7)\eta}\Big) +\frac{\EE\big[\|\GG_n\|_{\cF_{\sigma}(\delta)}\big]}{\eta}+ \frac{\EE\big[\|G_P\|_{\cF_{\sigma}(\delta)}\big]}{\eta}+\frac{C_{\alpha,d,M,l} (\log N_{\sigma,\delta})^2}{\eta^3 \sqrt{n}}.
\end{align*}

Hence, by Strassen's theorem (Lemma~\ref{lem: Strassen's theorem}), for $\eta > 0$ such that $l(n,\eta) > 0$, there exists a coupling $\big(V_{n,\sigma,\delta,\eta},W_{n,\sigma,\delta,\eta}\big)$ of $\sqrt{n}\gwassemp$ and $L^{\mspace{1mu}\sigma}_P$, such that
\be\label{eq:strassen_interim}
    \PP\big(|V_{n,\sigma,\delta,\eta}- W_{n,\sigma,\delta,\eta}| > (2+C_7)\eta\big) \leq \frac{\EE\big[\|\GG_n\|_{\cF_{\sigma}(\delta)}\big]}{\eta} + \frac{C_{\alpha,d,M,l} (\log N_{\sigma,\delta})^2}{\eta^3 \sqrt{n}} + \frac{\EE\big[\|G_{P}\|_{\cF_{\sigma}(\delta)}\big]}{\eta}.
\ee

Substituting the bounds \eqref{eq: L2P metric entropy bound}, \eqref{eq:emp_error_discretized}, and \eqref{eq:gaussian process discretization error} in \eqref{eq:strassen_interim}, we obtain the following estimate
\be\label{eq: strassen bound}
\PP\big(|V_{n,\sigma,\delta,\eta} - W_{n,\sigma,\delta,\eta}| > (2+C_7)\eta\big)\lesssim_{\alpha,d,M,l} \frac{1}{\eta}\left[\sigma^{-ab} \delta^{1-b} + \frac{1}{\delta^{2b}\sqrt{n}}\sigma^{-a(1+2b)} \right] + \frac{1}{\eta^3}\frac{\sigma^{-4ab}}{\delta^{4b}\sqrt{n}}.\ee
\medskip

\noindent\underline{Step 6} (Rates for $\delta$ and $\eta$).
Note that \eqref{eq: strassen bound} holds for any $\sigma,\delta,\eta \in (0,1]$, and $n$ under the moment condition of the theorem and $l(n,\eta) > 0$. Let $\sigma = \sigma_n$ as in the statement of the theorem. 

Taking $\delta = n^{-1/(2(3+b))}\sigma^{-ab/(3+b)}$ and $\eta = \eta_n = K_n n^{-(1-b)/(2(3+b))} \sigma_n^{-4ab/(3+b)}$ for $K_n \geq 1$, we see that the RHS is $O(1/K_n)$ with $l(n,\eta_n) > 0$, for all sufficiently large $n$. For all such $n$, the random variables $V_n := V_{n,\sigma_n,\delta_n,\eta_n}$ and $W_n := W_{n,\sigma_n,\delta_n,\eta_n}$ satisfy
\[\PP\left(|V_n - W_n| \geq K_n n^{-\frac{1-b}{2(3+b)}} \sigma_n^{-\frac{4ab}{3+b}}\right) = \frac{1}{K_n}O(1).\]
For $K_n \to \infty$ we have the claim (i).
\medskip

\textbf{Proof of (ii)}. To control $\rho\big(\sqrt{n}\gwassempsmall,\,L_P^{\mspace{1mu}\sigma_n}\big)$, we start from \eqref{eq: gwassemp set prob bound} and bound the latter probability terms on RHS using concentration inequalities. 
\medskip

\noindent\underline{Step 1} (Empirical process concentration via Fuk-Nagev). To treat the empirical process term, we apply Lemma~\ref{lem: fuk-nagev} to $\cF_{\sigma_n}(\delta_n)$. Matching the lemma notation to our framework, we set $p = l > 4\alpha -d + 3$, $\cF = \{f-Pf\,:\,f \in \cF_{\sigma_n}(\delta_n)\}$, $F(x) = \|x\| + {\sigma_n}\sqrt{d}$, $\alpha = 1$, $Z = \sqrt{n}\|\GG_n\|_{\cF}$, $\sigma^2 = \delta_n^2$, $\EE[M^p] \lesssim_{d} \EE\big[\max_{1\leq i \leq n} \|X_i\|^p\big] \leq n\EE\big[\|X_i\|^p\big]$, and $x = \sqrt{n}\big(\eta_n-2\EE\big[\|\GG_n\|_{\cF_{\sigma_n}(\delta_n)}\big]\big)$. For $\eta_n > 2\EE\big[\|\GG_n\|_{\cF_{\sigma_n}(\delta_n)}\big]$, we then have, for all $n$ such that $\eta_n-2\EE\big[\|\GG_n\|_{\cF_{\sigma_n}(\delta_n)}\big] > 0$,
\[\begin{split}
\PP\Big(\|\GG_n\|_{\cF_{\sigma_n}(\delta_n)} > \eta_n\Big) &\lesssim_{\alpha,d,M,l} \underbrace{e^{-\frac{c_{\alpha,d}}{\delta_n^2}\left(\eta_n-2\EE\big[\|\GG_n\|_{\cF_{\sigma_n}(\delta_n)}\big]\right)^2}}_{\mathrm{(a)}} \\
&\qquad+ \underbrace{n^{-\frac{l-2}{2}}\Big(\eta_n-2\EE\big[\|\GG_n\|_{\cF_{\sigma_n}(\delta_n)}\big]\Big)^{-(4\alpha -d + 3)}}_{\mathrm{(b)}}.
\end{split}
\]

\medskip

Hence, if we choose $\delta_n$ and $\eta_n$ such that $\eta_n \geq 3\EE\big[\|\GG_n\|_{\sigma_n,\delta_n}\big]$, and $\eta_n/\delta_n \gtrsim n^{p'}$ for some $p' > 0$, where $p'$ depends only on $\alpha$ and $d$, then $\text{(a)} \lesssim_{\alpha,d,M,l,p} n^{-p}$ for any $p > 0$. We then also have $\text{(b)} \lesssim_{\alpha,d,M,l} n^{-(l-2)/2} \eta_n^{-l}$.

\noindent\underline{Step 2} (Gaussian concentration). 
Under the above rate conditions, the Borel-Sudakov-Tsirelson inequality \citep[Theorem 2.2.7]{GiNi2016} implies that, for any $p > 0$,
\begin{align*}
    \PP\Big(\|G_{P}\|_{\cF_{\sigma_n}(\delta_n)} > \eta_n\Big) &= \PP\Big(\|G_{P}\|_{\cF_{\sigma_n}(\delta_n)} > \EE[\|G_{P}\|_{\cF_{\sigma_n}(\delta_n)}] + \big(\eta_n-\EE[\|G_{P}\|_{\cF_{\sigma_n}(\delta_n)}]\big)\Big)\\
    &\leq e^{-\frac{1}{2\delta_n^2}\big(\eta_n-\EE\big[\|G_{P}\|_{\cF_{\sigma_n}(\delta_n)}\big]\big)^2}\\
    & \lesssim_{\alpha,d,M,l,p} n^{-p},\numberthis\label{eq:gaussian_process_dicretization_error_conc_bound}
\end{align*}
with the same parameters as before. 
\medskip

\noindent\underline{Step 3} (Combined bounds and optimal rates). Assume we have chosen $\delta_n$ and $\eta_n$ so that the rate conditions from Step 1 hold. Using the last bounds in Steps 1 and 2 we obtain, respectively, $\PP(\|\GG_n\|_{\cF_{\sigma_n}(\delta_n)} > \eta_n) \lesssim_{\alpha,d,M,l,p} n^{-(l-2)/2}\eta_n^{-l} + n^{-p}$ and $\PP(\|G_{P}\|_{\cF_{\sigma_n}(\delta_n)} > \eta_n) \lesssim_{\alpha, d, M, l, p} n^{-p}$, for any $p > 0$. Note that \eqref{eq: L2P metric entropy bound} further gives $\frac{C_{d,M,l} (\log N_{\sigma_n,\delta_n})^2}{\eta_n^3 \sqrt{n}} \lesssim_{\alpha,d,M,l} \frac{1}{\eta_n^3}\frac{\sigma_n^{-4ab} \delta_n^{-4b}}{\sqrt{n}}$. Inserting these bounds into  \eqref{eq: gwassemp set prob bound}, the Prokhorov distance is bounded as
\[
\rho\left(\sqrt{n}\gwassempsmall,L_P^{\mspace{1mu}\sigma}\right)\lesssim_{\alpha,d,M,l} \eta_n \wedge \frac{1}{\eta_n^3}\frac{\sigma_n^{-4ab} \delta_n^{-4b}}{\sqrt{n}} \wedge n^{-\frac{l-2}{2}}\eta_n^{-l}.
\]
To conclude, we set $\delta_n = n^{-1/8}$. Then, from \eqref{eq:gaussian process discretization error}, we can choose a constant $C_{\alpha,d,M,l}$ such that $\eta_n = C_{\alpha,d,M,l} n^{-(1-b)/8}\sigma_n^{-ab} \geq 3 \EE\big[\|G_P\|_{\cF_{\sigma_n}(\delta_n)}\big]$, which fulfills the rate conditions from Step 1 with $p' = b$, and yields
\[
\rho\left(\sqrt{n}\gwassempsmall,\,L_P^{\mspace{1mu}\sigma_n}\right) \lesssim_{\alpha,d,M,l} n^{-(1-b)/8} \sigma_n^{-ab},
\]
as claimed.
\end{proof}

\subsection{Proof (Outline) of Theorem \ref{thm: Empirical Bootstrap Berry-Esseen type rate}}
\label{SUBSEC: SWD Empirical Bootstrap proof outline}
The subsection presents an outline of the proof, with the full details available in Appendix \ref{SUBSEC: SWD Empirical Bootstrap Berry-Esseen proof}. The comparison between the empirical bootstrap process and the limit Gaussian process is split into two parts. First, we compare the limit Gaussian process with its bootstrap analogue, the multiplier bootstrap process. Then, we related the multiplier bootstrap process to the empirical bootstrap process.

The multiplier bootstrap process indexed by the smooth 1-Lipschitz class  $\cF_{\sigma}$ is $\GG_n^B(f) := n^{-1/2} \sum_{i=1}^n \xi_i \big(f\left(X_i\right) - P_n f\big),\ {f \in \cF_{\sigma}}$, where $\xi_i\iid\cN_1$ are independent of $(X_i)_{i\in \NN}$. We also introduce the shorthand $X_{1:n}:=(X_1, \dots, X_n)$. As in \eqref{eq: berry esseen decomposition}, we start by discretizing the three processes and then compare the distributions of the discretized maxima of $\GG_n^B$ and $G_P$. We~have
\begin{align*}\label{eq: empirical bootstrap berry esseen decomposition}
    \PP&\big(\|G_n^*\|_{\cF_\sigma} \in A\,\big|\,X_{1:n}\big) - \PP\left(\|\GG_P\|_{\cF_\sigma} \in A^{(3+C_7)\eta}\right)\\
    &\leq \underbrace{\PP\Big(\big|\|G_P\|_{\cF_\sigma} - \max_{1\leq j \leq N_{\sigma, \delta}} G_P(f_j)\big| > \eta \Big)}_\text{(I)} + \underbrace{\PP\left(\big|\|\GG_n^*\|_{\cF_\sigma} - \max_{1\leq j \leq N_{\sigma, \delta}} \GG_n^*(f_j)\big| > \eta \,\Big|\,X_{1:n}\right)}_{\text{(II)}}\\
    &\qquad\quad+\underbrace{\PP\left(\max_{1\leq j \leq N_{\sigma, \delta}} \GG_n^*(f_j) \in A^{\eta}\,\middle|\,X_{1:n}\right) - \PP\left(\max_{1\leq j \leq N_{\sigma, \delta}} \GG_n^B(f_j) \in A^{(1+C_7)\eta}\,\middle|\,X_{1:n}\right)}_{\text{(III)}}\\
    &\qquad\quad+\underbrace{\PP\left(\max_{1\leq j \leq N_{\sigma, \delta}}\GG_n^B(f_j) \in A^{(1+C_7)\eta}\middle|\,X_{1:n}\right) - \PP\left(\max_{1\leq j \leq N_{\sigma, \delta}} G_P(f_j) \in A^{(2+C_7)\eta}\right)}_{\text{(IV)}}.\numberthis
\end{align*}

Term (I) can be bounded using \eqref{eq:gaussian_process_dicretization_error_conc_bound}, while (II) is controlled using \citet[Lemma 3.6.6]{VaWe1996}, with the expectation bound supplied by the local maximal inequality for a poissonized empirical process.
Term (III) is handled by Lemma~\ref{lem: gauss_approx_discretized}, while a bound on (IV) can be obtained by controlling the difference between the sample and population covariance matrices. For the latter, we use the following result.
\begin{lemma}[Theorem 3.2, \citet{ChChKa2016}]
\label{lem: gaussian comparison}
Let $X = (X_1,\dots,X_p)^\intercal\sim \cN(\mu, \Sigma^X)$ and $Y = (Y_1,\dots, Y_p)^\intercal\sim\cN(\mu, \Sigma^Y)$ be Gaussian random vectors in $\RR^p$. 
Define $\Delta= \max_{1\leq j,k \leq p} \big|\Sigma_{j,k}^X - \Sigma_{j,k}^Y\big|$, $Z\mspace{-2mu}=\mspace{-2mu}\max_{1\leq j\leq p}\mspace{-2mu}X_j$, and $\tilde{Z}= \max_{1\leq j\leq p} Y_j$. Then, for every $\eta > 0$ and $A\in\cB(\RR)$, we have
\[\PP(Z \in A) \leq \PP\big(\tilde{Z} \in A^\eta\big) + C \eta^{-1}\sqrt{\Delta \log p},\]
where $C> 0$ is an universal constant.
\end{lemma}

It thus remains to convert \eqref{eq: empirical bootstrap berry esseen decomposition} to an absolute difference bound in probability. For this step, we use a conditional version of Strassen's theorem to ensure the existence of couplings $\big(V_{\sigma,n},W_{\sigma,n}\big)$, with $\law\left(V_{\sigma,n}\,\big|\,X_{1:n}\right) = \law\left(\|\GG_n^*\|_{\cF_\sigma}\middle|\,X_{1:n}\right)$ and $\law(W_{\sigma,n}) = \law\left(\|G_P\|_{\cF_\sigma}\right)$,
such that the difference $|V_{\sigma,n} - W_{\sigma,n}|$ is bounded in probability.

\begin{lemma}[Lemma 4.2, \citet{ChChKa2016}]
\label{lem: Conditional Strassen's theorem}
Let $V$ be a real-valued random variable defined on a probability space $(\Omega, \mathcal{A}, \PP)$, and let $\mathcal{C}$ be a countably generated sub $\sigma$-field of $\mathcal{A}$. Assume that there exists a uniform random variable on $[0, 1]$ independent of $\mathcal{C}\vee\sigma(V)$. Let $G(\cdot| \mathcal{C})$ be a regular conditional distribution on the Borel $\sigma$-field of $\RR$ given $\mathcal{C}$, and suppose that for some $\delta > 0$ and $\epsilon > 0$,
\[\EE\left[\sup_{A \in \cB[\RR]}\left(\PP\big(V \in A\,\big|\,\cC\big) - G\big(A^\delta\,\big|\,\mathcal{C}\big)\right)\right] \leq \epsilon.\]
Then there exists a random variable $W$ such that the conditional distribution of $W$ given $\mathcal{C}$ coincides with $G(\cdot|\mathcal{C})$ and
\[\PP\big(|V - W| > \delta\big) < \epsilon.\]
\end{lemma}

An exact characterization of the dependence on $\sigma$ will allow us to fix $\sigma = \sigma_n$, possibly vanishing as $n$ tends to infinity, in the bounds. The bound on Prokhorov distance will follow from \eqref{eq: empirical bootstrap berry esseen decomposition} by bounding the expectation of terms on the RHS.

\subsection{Proof of Theorem \ref{thm: low dim rate for exp}}
\label{SUBSEC: low dim rate proof}
We prove the following stronger result, from which Theorem \ref{thm: low dim rate for exp} follows. 
\begin{theorem}[Strengthening of Theorem \ref{thm: low dim rate for exp} via entropy condition]
\label{thm: low dim rate for exp 2}
Suppose $P\in\cP(\RR^d)$ satisfies
\begin{equation}
\label{eq: low dim entropy condition}
N_\epsilon(P, \epsilon^\frac{2\alpha s}{s - 2}) \leq K \epsilon^{-s}\big(\log(1/\epsilon)\big)^\beta,\quad \forall\,\epsilon > 0
\end{equation} 
and $P \|x\|^s = M < \infty$, for some constant $K > 0$, $\beta \geq 0$ and $\alpha>s/2 >1$. 
Then, we have
\[\EE\left[\sqrt{n}\gwass (P_n,P) \right] \lesssim_{\alpha,d,s,M,K,\beta} \sigma^{-\alpha+1}.\]
If instead $\alpha = s/2$, then
\be \label{eq: SWD rate low dim}\EE\left[\sqrt{n}\gwass (P_n,P)\right] \lesssim_{d,s,M,K,\beta} \sigma^{-s/2+1}(\log n)^\frac{\beta+3}{2}.\ee
In particular, if $\dim^*_{\mathsf{SW}_1}(P) < s$ and $P \|x\|^s = M < \infty$, then \eqref{eq: SWD rate low dim} holds with $\beta = 0$ and $\alpha = s/2$.
\end{theorem}

We again begin with a few technical lemmas that will be used in the proof. The first provides a bound on tail expectations of random variables in the quantile form. This result is later used to bound the $L^2(P)$ entropy of $\cF_\sigma$ by accounting for the probability mass excluded in \eqref{eq: low dim entropy condition}.


The next lemma gives a bound on the metric entropy of smooth function classes defined on sufficiently small balls. This will be used in conjunction with the entropy condition \eqref{eq: low dim entropy condition} to bound the entropy of the entire, unbounded, function class, as in \citet[Theorem 2.7.4]{vanderVaart1996}.
\begin{lemma}[Metric entropy of small balls]
\label{lem: metric entropy small ball}
For all $\epsilon \in (0,1)$, \(\log N\big(\epsilon, C_1^\alpha\big(B(\epsilon^{1/\alpha})\big), \|\cdot\|_\infty\big) \lesssim_{\alpha,d} \log\big(1 + \frac{2C_d}{\epsilon}\big),\)
where $C_d = (1 + e^d)$.
\end{lemma}

Next, we state an adaptation of \citet[Lemma 2.14.3]{VaWe1996} for controlling moments of the empirical process when the entropy of the indexing function class is not integrable around 0.
\begin{lemma}[Maximal inequality for non-integrable bracketing entropy]
\label{lem: non integrable l2 entropy}
Let $\cF$ be a function class symmetric about $0$ and containing the $0$ function, with a measurable envelope $F\in L^2(P)$. Then,
\[\begin{split}
    \EE\big[\|\GG_n\|_{\cF}\big] \mspace{-2mu} \lesssim \mspace{-2mu} \inf_{0<\eta<1/3} &\left\{\eta\sqrt{n} \|F\|_{L^2(P)} \mspace{-2mu} + \mspace{-2mu} \|F\|_{L^2(P)} \mspace{-3mu} \int_\eta^1 \mspace{-4mu}\sqrt{1 + \log N_{[\,]}\big(\epsilon\|F\|_{L^2(P)}, \cF, L^2(P)\big)}\dd\epsilon\right\}\\&\qquad\quad\ +\sqrt{n}PF\ind_{\left\{4F\geq\sqrt{n}\|F\|_{L^2(P)}\left(1 + \log N_{[\,]}\big(\|F\|_{L^2(P)}/2, \cF, L^2(P)\big)\right)^{-1/2}\right\}}.
\end{split}\]
\end{lemma}

Proofs of the above lemmas are given in Appendix \ref{sec: low dim case proofs}. We are now ready to prove Theorem \ref{thm: low dim rate for exp 2}.
By~\eqref{eq: low dim entropy condition}, there is a set $S\in\cB(\RR^d)$ such that
$P(S) > 1 - \epsilon^{2s/(s - 2) }$ and $N_{\epsilon^{\nicefrac{1}{\alpha}}}(S) \leq K\alpha^{-\beta}\epsilon^{-\nicefrac{s}{\alpha}}\big(\log(\frac{1}{\epsilon})\big)^{\beta}$. 
Let $x_1, \dots, x_m$ be a minimal $\epsilon^{1/\alpha}$-net of $S$. By Lemma~\ref{lem: fractional derivatives}, there is a constant $C_{\alpha,d}$, depending only on $\alpha$ and $d$, such that each $\cF_\sigma|_{B(x_j,\epsilon^{1/\alpha})}$ is a subset of $C_{M'_j}^\alpha\big(B(x_j,\epsilon^{1/\alpha})\big)$ for 
$M'_j = C_{\alpha,d}\sigma^{-\alpha+1}\big(1\vee\sup_{x \in B(x_j,\epsilon^{1/\alpha})}\|x\|\big)$. For every $j=1,\dots,m$, let $f_{j,1}, \dots, f_{j,p_j}$ be a minimal $\epsilon M'_j$-net of $\cF_\sigma|_{B(x_j,\epsilon^{1/\alpha})}$ w.r.t. $\|\cdot\|_\infty$. Consider now the disjointed balls:
\[B_1 = B\big(x_1, \epsilon^{1/\alpha}\big),\ B_j = B(x_j, \epsilon^{1/\alpha}) \setminus \left\{\bigcup\nolimits_{k=1}^{j-1} B_k\right\},\quad j=2,\dots,m.\]

Let $g(x) = \|x\|\wedge 1 + \sqrt{d}$, $F(x) = \sigma^{-\alpha+1}g(x)$. Construct the function brackets
\[\sum_{j=1}^m \big(f_{j,i_j}\ind_{B_j} \pm \epsilon M'_j\big) \pm F\ind_{\left\{\bigcup_{j=1}^n B_j\right\}^c},\quad i_j=1,\dots,p_j,\quad j=1,\ldots,m.\]
These cover $\cF_\sigma$, and have $L^2(P)$ bracket width bounded by:
\[\sqrt{4\epsilon^2 \sum_{j=1}^m \big(M_j^{'2} P(B_j)\big) \mspace{-2mu} + \mspace{-2mu} 4 P F^2\ind_{\left\{\bigcup_{j=1}^n B_j\right\}^c}} \mspace{-2mu} \leq \mspace{-2mu} 4\epsilon \sigma^{-\alpha+1} \mspace{-1mu}\sqrt{P g^2 + (P g^s)^{2/s}} \lesssim_{\alpha, M, s, d}  \mspace{-2mu} \epsilon \|F\|_{L^2(P)},\]
where the first inequality uses Lipschitzness of $g$ and H\"older's inequality, while the second is since $P\|x\|^{s}\mspace{-4mu}=\mspace{-4mu}M$. 

The  total number of covering brackets is $\prod_{j=1}^m p_j$. Lemma~\ref{lem: metric entropy small ball} implies $\log p_j \lesssim_{\alpha,d} \log\big(1 + 2C_d/\epsilon\big)$ for each $j = 1,\dots,m$. Since \eqref{eq: low dim entropy condition} implies $\log m \leq K\alpha^{-\beta}\epsilon^{-s/\alpha}\big(\log(1/\epsilon)\big)^\beta$, we further have
\[
\log N_{[\,]}\big(C\epsilon \|F\|_{L^2(P)}, \cF_\sigma, L^2(P)\big) \leq \sum_{j=1}^m \log p_j\lesssim_{\alpha,M,s,d,\beta} K\epsilon^{-s/\alpha} \log \mspace{-3mu} \left(1 + \frac{2C_d}{\epsilon}\right)\mspace{-3mu} \big(\log (1/\epsilon)\big)^\beta.
\]

For $\alpha > s/2$, the square root of the right hand side is integrable around 0, so by Theorem 2.14.1 in \citet{VaWe1996}, we get the first conclusion of Theorem \ref{thm: low dim rate for exp}. For $\alpha = s/2$, we have
\[
\log N_{[\,]}\big(C\epsilon \|F\|_{L^2(P)}, \cF_\sigma, L^2(P)\big) \lesssim_{\alpha,s,d,M} K\epsilon^{-2} \big(\log(2/\epsilon)\big)^{1+\beta},
\]
which is not integrable around 0; this is where Lemma~\ref{lem: non integrable l2 entropy} is needed. Observing that
\[
\sqrt{n} PF\ind_{\{F > \sqrt{n}C\|F\|_{L^2(P)}\}} \leq PF^2\ind_{\{F > \sqrt{n}C\|F\|_{L^2(P)}\}}/(C\|F\|_{L^2(P)}) \leq \|F\|_{L^2(P)}/C,\] where $C = \sqrt{1 + \log N_{[\,]}\big(\|F\|_{L^2(P)}/2, \cF_\sigma, L^2(P)\big)} \geq 1$, the lemma implies that
\[
\begin{split}
\EE & \left[\sqrt{n}\gwass(P_n,P)\right]  \\
&\ \ \lesssim  \inf_{0 < \gamma \leq 1/3} \|F\|_{L^2(P)}  \left ( 1  +  \gamma \sqrt{n}  +  \int_{\gamma}^1  \sqrt{1 + \log N_{[\,]}(\epsilon \|F\|_{L^2(P)}, \cF_\sigma, L^2(P))} \dd\epsilon \right).
\end{split}
\]
Choosing $\gamma = 1/\sqrt{9n}$ yields $\EE\big[\sqrt{n}\gwass(P_n,P)\big]\lesssim_{\alpha,s,d,\beta}  \|F\|_{L^2(P)} + \|F\|_{L^2(P)} (\log n)^{(\beta+3)/2}$, and the second conclusion follows by noting that $\|F\|_{L^2(P)} \lesssim_{\alpha,d,M} \sigma^{-\alpha+1}$. 

For the last statement, observe that if $\dim^*_{\mathsf{SW}_1}(P) < s$, then by definition of $\dim^*_{\mathsf{SW}_1}$ we have $\limsup_{\epsilon \to 0} d_\epsilon\Big(P,\epsilon^{s^2/(2(s-2))}\Big) \leq s$. This implies, for all small enough $\epsilon > 0$, that $N_\epsilon\big(P,\epsilon^{s^2/(s-2)}\big) \lesssim_{s} \epsilon^{-s}$, which is precisely condition \eqref{eq: low dim entropy condition} with $\alpha = s/2$ and $\beta = 0$.

\section{Concluding Remarks and Future Directions}\label{SEC:summary}

This work conducted a statistical study of the SWD $\gwass$, which convolves the considered distributions with an isotropic Gaussian kernel. Via KR duality, the 1-Wasserstein distance can be viewed as an IPM w.r.t. the class of 1-Lipschitz functions. The Gaussian convolution reduces the complexity of this class, allows to show that it is Donsker, and, in turn, enables a comprehensive statistical analysis of the empirical SWD. We established limit distribution, bootstrap consistency, concentration, and strong approximation results for the empirical smooth distance. Furthermore, we showed that its rate of convergence in expectation adapts to the intrinsic dimension through the smoothing parameter. This was used to derive sharp convergence rates for empirical $\wass$ (up to log factors) for unboundedly supported population distributions. We also studied applications to two-sample homogeneity testing and minimum SWD estimation on parametric families, proving consistency of tests for the former and deriving measurability, consistency, and limit distributions for the latter. 

While this work focused on statistical aspects of empirical SWD, an important future avenue is to develop an efficient algorithm for computing it. While any computational method for classic $\wass$ is also applicable for the smooth framework (by sampling the kernel), we target algorithms that are tailored to exploit the Gaussian convolution structure. The simplest form of the computational question reduces to computing $\wass$ between Gaussian mixtures, which we plan to explore in the future. Another useful idea on the computational front is to replace the (analytically convenient) Gaussian kernel with a compactly supported smoothing kernel, so as to have control over the domain in which distributions are supported. Other possible future directions include evaluating asymptotic power and relative efficiency for two-sample homogeneity testing based on $\gwass$, and modifying the statistic $W_{m,n}$ to test for independence of paired observations. While this paper extensively exploits the dual form of the smoothed transport problem, it would also be interesting to study the convergence rate of optimal empirical couplings, appropriately defined, to the set of optimal couplings of the population distributions. Properties of such smoothed couplings between measures, such as cross-correlations between the involved variables, are also of interest.

\section*{Acknowledgements}
Z. Goldfeld is supported by NSF grants CCF-1947801 and CCF-2046018 (CAREER), and the 2020 IBM Academic Award. K. Kato is partially supported by NSF grants DMS-1952306 and DMS-2014636.


\clearpage
\begin{appendix}

\section{Supplementary Proofs for Section \ref{sec: limit}}


\subsection{Supplementary Proofs for Theorem \ref{thm: CLT for W1}}
\label{SUBSEC: limit variable supp proofs}
We start from the following auxiliary lemma. 

\begin{lemma}[Uniform bound on derivatives]
\label{lem: derivative}
For any $f \in \mathsf{Lip}_{1}$ and any nonzero multi-index $k=(k_{1},\dots,k_{d})$, we have 
\[
\big|D^{k} f_{\sigma}(x) \big| \le  \sigma^{-|k|+1} \sqrt{(|k|-1)!},\quad \forall x\in\RR^d.
\]
\end{lemma}
\begin{proof}
Let $H_{m}(z)$ denote the Hermite polynomial of degree $m$ defined by 
\[
H_{m}(z) =(-1)^{m} e^{z^{2}/2} \left [  \frac{d^{m}}{dz^{m}} e^{-z^{2}/2}  \right ], \ m=0,1,\dots.
\]
Recall that for $Z \sim \cN(0,1)$, we have $\E[H_{m}(Z)^{2}] = m!$. A straightforward computation shows that 
\[
D_{x}^{k} \gauss(x-y) = \gauss (x-y) \left [ \prod_{j=1}^{d} (-1)^{k_{j}} \sigma^{-k_{j}} H_{k_{j}} \big((x_{j}-y_{j})/\sigma\big)\right ]
\]
for any multi-index $k=(k_{1},\dots,k_{d})$, where $D_{x}$ means that the differential operator is applied to $x$. Hence, we have
\[
\begin{split}
D^{k} f_{\sigma}(x) &= \int f(y) \gauss (x-y) \left [ \prod_{j=1}^{d} (-1)^{k_{j}} \sigma^{-k_{j}} H_{k_{j}} \big ((x_{j}-y_{j})/\sigma \big)\right ] \dd y \\
&=\int f(x-\sigma y) \varphi_1 (y) \left [ \prod_{j=1}^{d} (-1)^{k_{j}} \sigma^{-k_{j}} H_{k_{j}} (y_{j})\right ] \dd y,
\end{split}
\]
so that, by $1$-Lipschitz continuity of $f$, 
\[
\left|D^{k} f_{\sigma}(x)-D^{k} f_{\sigma}(x')\right| \le \| x-x' \| \int \varphi_1 (y) \left [ \prod_{j=1}^{d} \sigma^{-k_{j}} \big|H_{k_{j}} (y_{j})\big| \right ] \dd y.
\]
The integral on the RHS equals
\[
\prod_{j=1}^{d} \sigma^{-k_{j}} \E\Big[ \big|H_{k_{j}}(Z)\big| \Big]
\le  \prod_{j=1}^{d} \sigma^{-k_{j}} \sqrt{\E\left[\big|H_{k_{j}}(Z)\big|^{2}\right] }= \prod_{j=1}^{d} \sigma^{-k_{j}} \sqrt{k_{j}!} \le \sigma^{-|k|} \sqrt{|k|!},
\]
where $Z \sim \cN (0,1)$.
The conclusion of the lemma follows from induction on the size of $|k|$.
\end{proof}

\begin{proof}[Proof of Lemma~\ref{lem: fractional derivatives}]
For $\alpha = 1$,
\[
\| f_\sigma \|_{\alpha,\cX}  \leq \sup_{\cX}\big(\|x\| + \sigma \sqrt{d}\big) \vee 1 \lesssim_d 1 \vee \sup_{\cX} \|x\|.\]
For $\alpha > 1$,
\begin{align*}
    \sup_{\overset{|k| = \underline\alpha}{x,y \in \mathcal{X}}}\mspace{-4mu} \frac{\big|D^k \mspace{-3mu} f_\sigma(x) \mspace{-3mu}- D^k \mspace{-3mu} f_\sigma(y)\big|}{\|x-y\|^{\alpha-\underline\alpha}} \mspace{-3mu} &\leq \mspace{-3mu}\left(\sup_{\overset{|k| = \underline\alpha}{x,y \in \mathcal{X}}}\mspace{-4mu} \frac{\big|D^k \mspace{-3mu} f_\sigma(x) \mspace{-2mu}-\mspace{-2mu} D^k \mspace{-3mu} f_\sigma(y)\big|^{\alpha - \underline\alpha}}{\|x-y\|^{\alpha-\underline\alpha}} \mspace{-4mu} \right) \mspace{-4mu} \mspace{-3mu}\sup_{\overset{|k| = \underline\alpha}{x \in \mathcal{X}}}\mspace{-4mu}\big|D^k \mspace{-3mu} f_\sigma(x)\mspace{-3mu}-\mspace{-3mu} D^k \mspace{-3mu} f_\sigma(y)\big|^{1 - \alpha + \underline\alpha}\\
    &\leq 2\left(\sup_{\overset{|k| = \underline\alpha + 1}{x \in \mathcal{X}}} \big|D^k \mspace{-3mu} f_\sigma(x)\big|\right)^{\alpha - \underline\alpha} \sup_{\overset{|k| = \underline\alpha}{x \in \mathcal{X}}}\big|D^k \mspace{-3mu} f_\sigma(x)\big|^{1 - \alpha + \underline\alpha}\\
    &\leq (\sigma^{-\underline\alpha}\sqrt{\underline\alpha!})^{\alpha - \underline\alpha} (\sigma^{-\underline\alpha+1}\sqrt{(\underline\alpha-1)!})^{1 - \alpha + \underline\alpha}\\
    &= \sigma^{-\alpha + 1} \big(\sqrt{\underline\alpha!}\big)^{\alpha - \underline\alpha}\big(\sqrt{(\underline\alpha-1)!}\big)^{1 - \alpha + \underline\alpha}.
\end{align*}
Hence,
\begin{align*}
    \| f_\sigma \|_{\alpha,\cX}  &\le \max_{0 \leq |k| \leq \underline\alpha} \sup_{x \in \cX} \left|D^k \mspace{-1mu} f_\sigma(x)\right| + C_\alpha \sigma^{-\alpha+1}\\
    &\leq \left((\sup_\mathcal{X} \|x\| + \sigma\sqrt{d}) \vee \sigma^{-\underline\alpha + 1}\sqrt{(\underline\alpha - 1)!}\right) + C_\alpha \sigma^{-\alpha+1}\\
    &\lesssim_{\alpha,d} \sigma^{-\alpha + 1} \big(1 \vee \sup_\mathcal{X} \|x\|\big),
\end{align*}
completing the proof. 
\end{proof}

\subsection{Proof of Corollary \ref{cor: sharp gwass rate}}
\label{SUBSEC: sharp gwass rate proof}
The following version of Dudley's entropy integral bound will be used in the proof of Corollary~\ref{cor: sharp gwass rate}. 

\begin{lemma}[Expectation bound for non-integrable entropy]
\label{lem: non integrable linf entropy}
For any bounded function class $\cF$,
\[
\EE[\|\GG_n\|_{\cF}] \lesssim \EE\left[\inf_{\gamma > 0} \left\{\gamma\sqrt{n} + \int_\gamma^{\mspace{1mu}\sigma_n} \sqrt{\log N(\delta, \cF, L^2(P_n))} \ \dd\delta\right\}\right],
\]
where $\sigma_n = \sup_{\cF}\|f\|_{L^2(P_n)}$.
\end{lemma}
This lemma follows by applying \citet[Theorem 5.31]{van2018probability} after symmetrization and conditioning on $X_1, \dots, X_n$.

\begin{proof}[Proof of Corollary \ref{cor: sharp gwass rate}]
We follow the notation used in the proof of Theorem \ref{thm: CLT for W1}. 
If $\alpha=d/2$, then by Lemma~\ref{lem: non integrable linf entropy},
\[
\begin{split}
\E\big[\| \mathbb{G}_n \|_{{\cF_\sigma}_{j}}\big] &\lesssim \E \left [ \inf_{\gamma} \left \{ \gamma \sqrt{n} + \int_{\gamma}^{M_{j}' P_n(I_{j})^{1/2}} \frac{M_{j}' P_n(I_{j})^{1/2}}{\epsilon} \dd\epsilon \right \} \right] \\
&\leq \E \left [ M_{j}' P_n(I_{j})^{1/2} +M_{j}' P_n(I_{j})^{1/2} \log \sqrt{n}\right ] \\
&\lesssim M_j' P(I_j)^{1/2}\log n\\
&\lesssim_{d} \sigma^{-d/2 + 1} M_j P(I_j)^{1/2}\log n,
\end{split} 
\]
where the second inequality follows by choosing $\gamma = M_{j}' P_n(I_{j})^{1/2}/\sqrt{n}$. The desired moment bound follows by summing the final expression over $j$.
\end{proof}


\subsection{Proof of Lemma~\ref{lem: limit variable} and Corollary \ref{cor: bootstrap}}\label{SUPP: bootstrap proof}

\begin{proof}[Proof of Lemma~\ref{lem: limit variable}]
From the proof of Theorem \ref{thm: CLT for W1} and the fact that $\mathsf{Lip}_{1}$ is symmetric, we have $L_P^{\mspace{1mu}\sigma} = \| G_{P} \|_{{\cF_\sigma}}$. Since $G_{P}$ is a tight Gaussian process in $\ell^{\infty}({\cF_\sigma})$, ${\cF_\sigma}$ is totally bounded for the pseudometric $\mathsf{d}_{P} (f,g) = \sqrt{\Var_{P}(f_\sigma-g_\sigma)}$, and $G_{P}$ is a Borel measurable map into the space of $\mathsf{d}_{P}$-uniformly continuous functions $\calC_{u}({\cF_\sigma})$ equipped with the uniform norm $\| \cdot \|_{{\cF_\sigma}}$. 
Let $F$ denote the distribution function of $L_P^{\mspace{1mu}\sigma}$, and define $r_{0} := \inf \{ r \ge 0 : F(r) > 0 \}$.

From \citet[Theorem 11.1]{Davydov1998}, $F$ is absolutely continuous on $(r_{0},\infty)$, and there exists a countable set $\Delta \subset (r_{0},\infty)$ such that $F'$ is positive and continuous on $(r_{0},\infty) \setminus \Delta$.
The theorem however does not exclude the possibility that $F$ has a jump at $r_0$, and we will verify that (i) $r_0= 0$ and (ii) $F$ has no jump at $r=0$, which lead to the conclusion. 
The former follows from p. 57 in \citet{LeTa1991}.
The latter is trivial since
\[
F(0)-F(0-)=\PP\left(L_P^{\mspace{1mu}\sigma}=0\right) \le \PP\big(G_{P}(f) =0\big),\quad \forall f \in {\cF_\sigma}.
\]
As $G_{P}$ is Gaussian, we have $\PP\big(G_{P}(f) =0\big)=0$, for any $f$ such that $\Var_P(f) > 0$. If $P$ is not a point mass, there always exists such an $f \in \cF_\sigma$.
\end{proof}

\begin{proof}[Proof of Corollary \ref{cor: bootstrap}]
From Theorem 3.6.3 in \citet{VaWe1996} applied to the function class ${\mathcal{F}_\sigma}$ together with the continuous mapping theorem, we see that conditionally on $X_{1},X_{2},\dots$, one has
\[
\sqrt{n}\gwass\big(P_{n}^{B},P_{n}\big) = \big\| \sqrt{n}(P_{n}^{B} -P_{n})\big\|_{\cF_\sigma} \stackrel{w}{\to} L_{P}^{\mspace{1mu}\sigma}
\]
for $\PP$-almost every realization of $X_{1},X_{2},\dots$ The conclusion follows from the fact that the distribution function of $L_{P}^{\mspace{1mu}\sigma}$ is continuous (cf. Lemma~\ref{lem: limit variable}) and Polya's Theorem (cf. Lemma 2.11 in \citet{vanderVaart1998_asymptotic}).
\end{proof}


\subsection{Proof of Proposition \ref{prop: concentration}}\label{SUBSEC: concentration}

Cases (i) and (ii) follow from Theorems 4 and 2 in \citet{Adamczkak2008} and \citet{Adamczkak2010}, respectively, applied to the function class ${\cF_\sigma}$ using the envelope function ${F}(x) = \| x \| + \sigma \sqrt{d}$. We omit the details for brevity.
\qed

\subsection{Proof of Corollary \ref{cor: 2 sample limit}}\label{subsec: 2 sample limit proof}
If $\cF_\sigma$ is $P$-Donsker, $Q = P$, and $m/(m+n) \to \lambda$, then $W_{m,n}$ converges in distribution to the supremum of the tight Gaussian process $G_P$ over $\cF_\sigma$ \citep[Chapter 3.7]{VaWe1996}. The bootstrap convergence follows by \citet[Theorem 3.7.7]{vanderVaart1996}. Lastly, since the distribution of $L_H^{\mspace{1mu}\sigma}$ is also Lebesgue absolutely continuous under the above conditions (Lemma~\ref{lem: limit variable}) and $H = \lambda P + (1-\lambda) Q = P$ when $P = Q$, the final statement follows.
\qed

\subsection{Proof of Proposition~\ref{prop: two_sample_bootstrap_alt}}
\label{subsec: two_sample_bootstrap_alt}

\begin{proof}[Proof of Proposition~\ref{prop: two_sample_bootstrap_alt}]
The proof follows along similar lines as that of Proposition 1 in \citet{Dumbgen1993} by proving a joint unconditional limit distribution for the bootstrap and empirical processes on $\ell^\infty(\cF_\sigma)$. First, we need some notation. Letting $r_{m,n} := \sqrt{mn/(m + n)}$, we define the following processes on the Banach space $\ell^{\infty}(\cF_\sigma)$: $T(f) := (P-Q)f$, $\hat{T}_{m,n}(f) = (P_n - Q_m)f$, $\hat{T}_{m,n}^B(f) = (\tilde P_n^B - \tilde Q_m^B)f$, $B_{m,n} = r_{m,n} (\hat T_{m,n} - T)$ and $C_{m,n} = r_{m,n} ( \hat T_{m,n}^B - \hat T_{m,n} ) $. We will also use the shorthand notation $X_{1:n} = (X_1, \dots, X_n)$, $Y_{1:m} = (Y_1, \dots, Y_m)$. Finally, let $\Phi(\alpha)=\| \alpha \|_{\cF_{\sigma}}$ denote the supremum functional on $\ell^\infty(\cF_\sigma)$. By \citep[Theorem~2.1]{carcamo2020directional},  $\Phi$ is Hadamard directionally differentiable at $\alpha \neq 0$, and we denote its derivative by $\Phi_{\alpha}'$.  Using the above notation, we have
\[
\bar W^B_{m,n} = r_{m,n} \big ( \Phi(\hat T_{m,n}^B) - \Phi(\hat T_{m,n}) \big ).
\]

By Theorem~3.6.3 in \citet{vanderVaart1996} (or its simple extension to the two sample setting),  $ \law( C_{m,n} | X_{1:n}, Y_{1:m}) \overset{w}{\to} \law \left ( G_{P,Q,\lambda}^\sigma \right ) \ \PP$-a.s. Combined with the weak convergence of the two sample empirical process $B_{m,n}$, this yields the following joint weak convergence of $B_{m,n}$ and $C_{m,n}$ via \citet[Theorem 2.2]{kosorok2008bootstrapping}:
\[
(B_{m,n}, C_{m,n}) \overset{w}{\to} (G_{P,Q,\lambda}^\sigma, \tilde G_{P,Q,\lambda}^\sigma)
\]
on $\ell(\cF_\sigma)^2$. Then, by the continuous mapping theorem, we further have
\[
r_{m,n} \big ( (\hat T_{m,n}^B, \hat T_{m,n}) - (T,T) \big ) = (B_{m,n} + C_{m,n}, B_{m,n}) \overset{w}{\to} (G_{P,Q,\lambda}^\sigma + \tilde G_{P,Q,\lambda}^\sigma, G_{P,Q,\lambda}^\sigma)
\]
on $\ell^\infty(\cF)^2$. 
Now, it is not difficult to see that the bivariate map $(\Phi, \Phi)$ on $\ell^\infty(\cF)^2$ is also Hadamard directionally differentiable at $(\alpha,\alpha) \ne (0,0)$ with derivative $(\Phi_{\alpha}',\Phi'_{\alpha})$. Combining the above the facts and applying the extended functional delta method \citep{shapiro1990concepts,romisch2004delta}, we have
\[
r_{m,n} \big (   \Phi(\hat T_{m,n}^B) - \Phi(T), \Phi(\hat T_{m,n}) - \Phi(T) \big ) \overset{w}{\to} \big ( \Phi'_T(G_{P,Q,\lambda}^\sigma + \tilde G_{P,Q,\lambda}^\sigma),  \Phi'_T(G_{P,Q,\lambda}^\sigma  ) \big ),
\]
which implies that
\[
r_{m,n} \big ( \Phi(\hat T_{m,n}^B)- \Phi(\hat T_{m,n}) \big ) \overset{w}{\to} \Phi'_T(G_{P,Q,\lambda}^\sigma + \tilde G_{P,Q,\lambda}^\sigma) - \Phi'_T(G_{P,Q,\lambda}^\sigma  ).
\]
The final expression coincides with $\| G^{\mspace{1mu}\sigma}_{P,Q,\lambda} + \tilde{G}^{\mspace{1mu}\sigma}_{P,Q,\lambda}\|_{\bar{M}_\sigma} - \| G^{\mspace{1mu}\sigma}_{P,Q,\lambda}\|_{\bar{M}_\sigma}$ by Theorem 6.1 in \cite{carcamo2020directional}.
\end{proof}


\section{Supplementary Proofs for Section \ref{sec: Berry-Esseen type results}}
\subsection{Supplementary Proofs for Theorem \ref{thm: SWD small sigma Berry-Esseen type rate}}
\label{SUBSEC: SWD Berry-Esseen Rate supp proofs}
\begin{proof}[Proof of Lemma~\ref{lem:partition_entropy}]
See Theorem 2.7.4 in \citet{VaWe1996}.
\end{proof}
\begin{proof}[Proof of Lemma~\ref{lem: moment inequality}]
For $I_j \subset \big\{x\,:\,\|x\|_\infty \in (k-1,k] \big\} =: L_k$, $k \in \NN$, we have $\sup_{I_j}\|x\| \leq k\sqrt{d}$. Now,
\[\sum_{j=1}^\infty \left(\sup\nolimits_{I_j}\|x\|\right)^{p_1} P(I_j)^{p_2} = \sum_{k=1}^\infty \sum_{j: I_j \subset L_k} \left(\sup\nolimits_{I_j}\|x\|\right)^{p_1} P(I_j)^{p_2}\lesssim_d \sum_{k=1}^\infty k^{p_1} P_k,\]
    where $P_k := \sum_{j: I_j \subset L_k} P(I_j)^{p_2}$. Setting $l_k := \Card\{I_j\,:\,I_j \subset L_k\}$, we have
    \[l_k =(2k)^d - (2k-2)^d \lesssim_d k^{d-1},\]
    where the last inequality follows since the expression inside the bracket is a sequence that converges to $1$, and hence is bounded. By the power mean inequality
    \[\left(\sum_{j:I_j \subset L_k} \frac{P(I_j)^{p_2}}{l_k}\right)^\frac{1}{p_2} \leq \frac{P(L_k)}{l_k}\quad \implies\quad P_k \leq l_k^{1-p_2} P(L_k)^{p_2} \lesssim_d k^{(d-1)(1-p_2)} P(L_k)^{p_2}.\]
    Hence, a sufficient condition for  $\sum_{k=1}^\infty k^{p_1} P_k < \infty$ follows by observing that
    \[\sum_{k=1}^\infty k^{p_1 + (d-1)(1-p_2)}P(L_k)^{p_2} \lesssim \int_{1}^\infty t^{p_1 + (d-1)(1- p_2)} \PP(\|X\|_\infty > t)^{p_2}\ \dd t.\]
    If $\EE[\|X\|_\infty^l] = M < \infty$ for any $l > \frac{(d-1)(1-p_2) + p_1 +1}{p_2}$, then Markov's inequality implies that the above integral is bounded by $M^{p_2}$ up to a constant depending on $l$, $p_1$, $p_2$, and $d$.
\end{proof}

\subsection{Proof of Theorem \ref{thm: Empirical Bootstrap Berry-Esseen type rate}}
\label{SUBSEC: SWD Empirical Bootstrap Berry-Esseen proof}

Recall the decomposition from \eqref{eq: empirical bootstrap berry esseen decomposition}. 
Term (I) has already been bounded in \eqref{eq:gaussian_process_dicretization_error_conc_bound}. We bound terms (II)-(IV) in the following steps. Fix $n \in \NN$ and $\sigma, \delta \in (0,1]$.

\noindent\underline{Step 1} (Discretization error of supremized empirical bootstrap (II)). By \citet[Lemma 3.6.6]{VaWe1996}, we have,
\[\EE\left[\|\GG_n^*\|_{\cF_\sigma(\delta)}|\middle|\,X_{1:n}\right] \lesssim \frac{1}{\sqrt{n}}\EE\left[\left\|\sum_{i=1}^n (N_i - N'_i) \delta_{X_i}\right\|_{\cF_\sigma(\delta)}\middle|\,X_{1:n}\right],\]
where $N_i,N'_i$ are i.i.d. Poisson(1/2) random variables, independent of $X_{1:n}$. Consider the modified function class:
\[\overline{\cF_\sigma(\delta)} := \{(x,w) \mapsto wf(x) : f \in \cF_\sigma(\delta), w \in \ZZ\}.\]
With some abuse of notation, let $\GG_n$ denote also the empirical process corresponding to the product probability law $P \otimes \law(N_i - N'_i)$. Then, for the marginal expectation of the discretization error above, we have
\[
\EE\left[\|\GG_n^*\|_{\cF_\sigma(\delta)}\right] \lesssim \frac{1}{\sqrt{n}}\EE\left[\left\|\sum_{i=1}^n (N_i - N'_i) \delta_{X_i}\right\|_{\cF_\sigma(\delta)}\right] = \EE\left[\|\GG_n\|_{\overline{\cF_{\sigma}(\delta)}}\right].
\]

To bound this marginal expectation, we will apply local maximal inequality on restrictions of process to compact partitioning sets (as in Step 2 of the proof of Theorem \ref{thm: SWD small sigma Berry-Esseen type rate}) and sum up the resulting bounds. Now, let $\overline{\cF_{j,k}(\delta)} := \big\{(x,w) \mapsto f(x,w)\ind_{I_j \times \{k\}}(x,|w|) : j \in \NN, k \in \NN\big\}$, where $I_j$ are an enumeration of unit cubes on integer lattice points in $\R^d$. Note that $F_{j,k}(x) = kM'_j\ind_{I_j \times \{k\}}(x,|w|)$ is an envelope function for this class. As in \eqref{eq:uniform_entropy_integral}, the entropy integral is bounded by
\be \label{eq:poiss_uniform_entropy_integral}
\begin{split}
J\big(\delta, \overline{\cF_{j,k}(\delta)}, F_{j,k}\big) &:= \int_0^\delta \sup_{Q \in \cP_{\mathsf{f}}(\R^{d+1})} \sqrt{\log N\left(\epsilon\|2F_j\|_{L^2(Q)}, \overline{\cF_{j,k}(\delta)}, L^2(Q)\right)} \dd\epsilon \\ &\lesssim_{\alpha,d} \int_0^\delta \epsilon^{-b} \dd \epsilon \lesssim_{\alpha,d} \delta^{1-b},
\end{split}
\ee
which, via \citet[Theorem 5.2]{ChChKa2014}, yields

\begin{align*}
\label{eq:poiss_emp_error_discretized}
\EE\big[\|\GG_n^*\|_{\cF_\sigma(\delta)}\big] \lesssim \EE\big[\|\GG_n\|_{\overline{\cF_{\sigma}(\delta)}}\big] &\lesssim_{\alpha,d} \sum_{j=1}^\infty \sum_{k=1}^\infty \left[\big(kM'_j p_k^{\frac{1}{2}} P(I_j)^\frac{1}{2}\big)^b \delta^{1 - b} \mspace{-3mu}+\mspace{-3mu} \frac{(kM'_j)^{1 + 2b} p_k^b P(I_j)^b}{\delta^{2b}\sqrt{n}} \right]\\
&\lesssim_{\alpha,d} \sum_{j=1}^\infty \left(M'_j P(I_j)^\frac{1}{2}\right)^b \delta^{1 - b} + \frac{1}{\sqrt{n}} (M'_j)^{1 + 2b}P(I_j)^b\delta^{-2b}\\
&\lesssim_{\alpha,d,M,L} \sigma^{-ab} \delta^{1-b} + \frac{\sigma^{-a(1+2b)} \delta^{-2b}}{\sqrt{n}}.\numberthis
\end{align*}

\noindent\underline{Step 2} (Comparison between $\GG_n^*$ and $\GG_n^B$ after discretization (III)). 
We now compare the discretized empirical and multiplier bootstrap processes conditioned on $X_{1:n}$, using Lemma~\ref{lem: gauss_approx_discretized}. To that end, let $f_1, \dots, f_{N_{\sigma,\delta}}$ be an $L^2(P)$ $\delta$-net of $\cF_\sigma$. Conditioned on $X_{1:n}$, for a bootstrap resample $X_1^B,\ldots,X_n^B$, the random vectors $\big(f_1(X_i^B) - P_n f, \ldots, f_{N_{\sigma,\delta}}(X_i^B) - P_n f_{N_{\sigma,\delta}}\big)$, $1\leq i \leq n$, are i.i.d. with conditional mean zero and conditional covariances given by $\Cov\big(f_j(X_i^N), f_k(X_i^B)\,\big|\,X_{1:n}\big) = \Cov_{P_n}(f_j,f_k)$. For the multiplier bootstrap, conditioned on $X_{1:n}$, the distribution of $\GG_n^B(f), f\in\cF_\sigma$ is that of a Gaussian process with covariance function $\Cov\big(\GG_n^B(f), \GG_n^B(g)\big) = \Cov_{P_n}(f,g)$. Consequently, $\big(\GG_n^B(f_1), \dots, \GG_n^B(f_{N_{\sigma,\delta}})\big)$ and $\big(f_1(X_i^B), \dots, f_{N_{\sigma,\delta}}(X_i^B)\big)$ have the same conditional covariance matrix.

The terms $L_n$ and $M_{n,X}(\eta)$ in Lemma~\ref{lem: gauss_approx_discretized} can be bounded by $(2/n)\sum_{i=1}^n \big(\|X_i\| + \sigma\sqrt{d}\big)^3$. We next bound $M_{n,Y}(\eta)$. Here, conditioned on $X_{1:n}$, $Y_{j} := n^{-1/2} \sum_{i=1}^n \eta_i \big(f_j(X_i) - P_n f_j\big)$ are $\cN(0,\sigma_j^2)$, where $\sigma_j^2 \leq n^{-1} \sum_{i=1}^n\big(\|X_i\|+ \sigma\sqrt{d}\big)^2$, for all $1 \leq j \leq p$. Then, we have the following concentration inequality and moment bound for Gaussian random variables \citep[Example 2.1.19]{GiNi2016}:
\begin{align*}
    \PP\left(\max_{1\leq j \leq N_{\sigma, \delta}} |Y_{j}| > \EE\left[\max_{1\leq j \leq N_{\sigma, \delta}} |Y_{j}\,|\middle|X_{1:n}\right] + t\,\middle|X_{1:n}\right) &\leq  \exp\left(-\frac{t^2}{\frac{2}{n} \sum_{i=1}^n\big(\|X_i\|+ \sigma\sqrt{d}\big)^2}\right),\\
    \EE\left[\max_{1\leq j \leq N_{\sigma, \delta}} |Y_{j}|\,\middle|X_{1:n}\right] &\leq \sqrt{ \frac{2}{n} \sum_{i=1}^n\big(\|X_i\|+ \sigma\sqrt{d}\big)^2 \log (2N_{\sigma,\delta})}.
\end{align*}
To bound $M_{n,Y}(\eta)$, we integrate the above tail probability bound. To that end, consider the event 
\[E := \left\{\frac{\eta\sqrt{n}}{\log N_{\sigma,\delta}} \geq \sqrt{ \frac{2}{n} \sum_{i=1}^n\big(\|X_i\|+ \sigma\sqrt{d}\big)^2 \log (2N_{\sigma,\delta})}\right\},\]
and note that on $E$, $M_{n,Y}(\eta) \leq C\big((1/n)\sum_{i=1}^n (\|X_i\| + \sigma\sqrt{d})^2\big)^{3/2}$. Taking $C_d = C_8 (1 \vee C)$, this implies
\[\begin{split}
\PP\left(\left\{\max_{1 \leq j \leq N_{\sigma,\delta}} \GG_n^*(f_i) \in A^\eta\right\} \cap E\middle| X_{1:n} \right) &\leq \PP\left(\max_{1 \leq j \leq N_{\sigma,\delta}} \GG_n^B(f_i) \in A^{(1+C_7) \eta} \middle| X_{1:n} \right) \\
&\qquad\qquad+ \frac{C_d(\log N_{\sigma,\delta})^2}{\eta^3\sqrt{n}} \left[\frac{3}{n}\sum_{i=1}^n (\|X_i\| + \sigma\sqrt{d})^3\right].
\end{split}\]
Adding a term for $E^c$, we further obtain
\be\label{eq:bootstrap_comparison}
\begin{split}
\PP\left(\max_{1 \leq j \leq N_{\sigma,\delta}} \GG_n^*(f_i) \in A^\eta \middle| X_{1:n} \right) &\leq \PP\left(\max_{1 \leq j \leq N_{\sigma,\delta}} \GG_n^B(f_i) \in A^{(1+C_7) \eta} \middle| X_{1:n} \right) \\
&\quad+ \frac{C_d(\log N_{\sigma,\delta})^2}{\eta^3\sqrt{n}} \left[\frac{3}{n}\sum_{i=1}^n (\|X_i\| + \sigma\sqrt{d})^3\right] + \PP\left(E^c\middle|\,X_{1:n}\right).
\end{split}\ee
\medskip 

Finally, observe that $\PP(E^c|X_{1:n})=\mathds{1}_{E^c}$ since $E$ is $\sigma(X_{1:n})$-measurable, and Markov's inequality yields
\be\label{eq:bad_set_prob_bound} \PP(E^c) \leq \PP\left(\frac{1}{n}\sum_{i=1}^n \left(\|X_i\| + \sqrt{d}\right)^2 > \frac{n \eta^2}{2(\log (2N_{\sigma,\delta}))^3}\right) \lesssim_{\alpha,d,M,l} \frac{(\log (2N_{\sigma,\delta}))^3}{n \eta^2}.\ee

\noindent\underline{Step 3} (Comparison between $\GG_n^B$ and $G_P$ after discretization (IV)). To apply Lemma~\ref{lem: gaussian comparison}, let $Z = \max_{1\leq i \leq N_{\sigma, \delta}} G_{P} (f_i)$, and $\tilde Z = \max_{1\leq i \leq N_{\sigma, \delta}}  \GG_n^B (f_i)$. By Lemma~\ref{lem: gaussian comparison}, for all $\eta > 0$, and a constant $C_9$ depending only on $d$,
\be\label{eq:multiplier_gaussian_comparison}
\PP\big(\tilde{Z} \in A^{(1+C_7)\eta}\big|X_{1:n}\big) \leq \PP\big(Z \in A^{(2+C_7)\eta}\big) + C_9 \eta^{-1} \sqrt{\Delta(X_{1:n}) \log N_{\sigma, \delta}},
\ee
where $\Delta(X_{1:n})=\max_{1\leq j,k \leq N_{\sigma,\delta}} \big|\Sigma_{j,k}^X(X_{1:n}) - \Sigma_{j,k}^Y\big|$ with $\Sigma^X(X_{1:n})$ as the conditional covariance matrix of the vector $\big(\GG_n^B(f_1), \dots, \GG_n^B(f_{N_{\sigma,\delta}})\big)$, and $\Sigma_Y$ the unconditional covariance matrix of $\big(G_P(f_1), \dots, G_P(f_{N_{\sigma,\delta}})\big)$.
To bound $\Delta(X_{1:n})$, note that
\begin{align*}
    \Sigma_{j,k}^Y \mspace{-3mu} - \mspace{-3mu} \Sigma_{j,k}^X(X_{1:n}) &= \Cov_{P}(f_j, f_k) \mspace{-3mu} - \mspace{-3mu} \Cov_{P_n}(f_j, f_k)\\
    &= P f_j f_k \mspace{-3mu} - \mspace{-3mu} P f_j P f_k \mspace{-3mu} - \mspace{-3mu} P_n f_j f_k + P_n f_j P_n f_k\\
    &= \mspace{-3mu} - \mspace{-3mu}(P_n \mspace{-3mu} - \mspace{-3mu}P) f_j f_k \mspace{-3mu} + \mspace{-3mu} (P_n \mspace{-3mu}-\mspace{-3mu} P)f_j (P_n \mspace{-3mu}-\mspace{-3mu} P) f_k \mspace{-3mu} + \mspace{-3mu} P f_k (P_n \mspace{-3mu}-\mspace{-3mu} P)f_j \mspace{-3mu} + \mspace{-3mu} P f_j(P_n \mspace{-3mu}-\mspace{-3mu} P)f_k.
\end{align*}
Consider the class $(\cF_\sigma)^{\times 2} = \{fg: f,g \in \cF_\sigma\}$, and denote its restriction to the cube $I_j$ by $(\cF_j)^{\times 2}$. With this notation, we have  $\PP$-a.s.
\be\label{eq: Delta bound interim}
\Delta(X_{1:n}) \leq \sum_{j=1}^\infty\frac{1}{\sqrt{n}}\left[\|\GG_n\|_{(\cF_j)^{\times 2}} + \|\GG_n\|_{\cF_j}^2 + 2\Big(\sup_{\cF_j} P|f|\Big) \|\GG_n\|_{\cF_j}\right].
\ee

Given \eqref{eq: Delta bound interim}, we bound $\EE\big[\Delta(X_{1:n})\big]$ using empirical process techniques. We start by controlling the metric entropy of $(\cF_j)^{\times 2}$. Recall that from Lemma~\ref{lem: metric entropy} we have
\[
\log N\big(\epsilon M, C_M^N(\chi), \| \cdot\|_\infty\big) \lesssim_{d, N, \text{diam}(\chi)} \epsilon^{-d/N},
\]
where $C_M^N(\chi)$ is the H\"{o}lder ball of smoothness $N$ and radius $M$. For  $(\cF_\sigma)^{\times 2}$, we take $N =  \alpha$ and $M = 2^{\lceil \alpha \rceil} (M'_j)^2$, and for any finitely supported measure $Q$ obtain
\[
\log N\big(\epsilon(M'_j)^2 Q(I_j)^{\frac{1}{2}}, (\cF_j)^{\times 2}, L^2(Q)\big) \lesssim_{d,\alpha} \epsilon^{-\frac{d}{\alpha}}.
\]
Then, using Theorem 2.14.1 in \citet{VaWe1996}, the expectation of the first term in the RHS of \eqref{eq: Delta bound interim} is bounded~as
\[
\EE\big[\|\GG_n\|_{(\cF_\sigma)^{\times 2}}\big] = \sum_{j=1}^\infty \EE\big[\|\GG_n\|_{(\cF_j)^{\times 2}}\big] \lesssim_{\alpha,d} \sum_{j=1}^\infty (M'_j)^{2}P(I_j)^{\frac{1}{2}}.
\]
For the second term, \eqref{eq: moment} combined with \citet[Theorem 2.14.5]{VaWe1996} implies 
\[
\EE\left[\|\GG_n\|_{\cF_j}^2\right] \lesssim_{\alpha,d} \sum_{j=1}^\infty (M'_j)^2 P(I_j),
\]
while the third term can be bounded by \eqref{eq: moment}. Combined, this gives
\be
\label{eq: Delta bound}
\EE\big[\Delta(X_{1:n})\big] \lesssim_{\alpha,d} \frac{1}{\sqrt{n}}\left[\sum_{j=1}^\infty (M'_j)^{2}P(I_j)^{\frac{1}{2}} + \sum_{j=1}^\infty (M'_j)^{2}P(I_j)\right] \lesssim_{\alpha,d,M,l} \sigma^{-2 a } n^{-\frac{1}{2}}.
\ee
\medskip 

\noindent\underline{Step 4} (Conclusion using Strassens's theorem).
By the previous steps, we have the following bounds on the terms in RHS of \eqref{eq: empirical bootstrap berry esseen decomposition}. First, by the Borel-Sudakov-Tsirelson inequality \citet[Theorem 2.2.7]{GiNi2016}, the Gaussian process discretization error is bounded by (see \eqref{eq:gaussian_process_dicretization_error_conc_bound})
\[
\text{(I)} \leq  \PP\Big(\|G_{P}\|_{\cF_{\sigma}(\delta)} > \eta\Big) \leq e^{-\frac{1}{2\delta^2}\big(\eta-\EE[\|G_{P}\|_{\cF_{\sigma}(\delta)}]\big)^2},
\]
which holds for every $\eta \geq C_{\alpha,d,M,l} \sigma^{-ab}\delta^{1-b} \geq \EE\big[\|G_{P}\|_{\cF_{\sigma}(\delta)}\big]$ by \eqref{eq:gaussian process discretization error}. Second, by Markov's inequality and \eqref{eq:poiss_emp_error_discretized}, we have the following bound for the empirical bootstrap process discretization error:
\[\EE\big[\text{(II)}\big] \lesssim_{\alpha,d,M,l} \frac{1}{\eta} \left[\sigma^{-ab} \delta^{1-b} + \frac{\sigma^{-a(1+2b)} \delta^{-2b}}{\sqrt{n}}\right].\]
Third, starting from \eqref{eq:bootstrap_comparison}, we bound $\log N_{\sigma,\delta}$ and $\PP(E^c)$ using \eqref{eq: L2P metric entropy bound} and \eqref{eq:bad_set_prob_bound}, respectively. This yield the following inequality for the comparison between empirical and multiplier bootstraps:
\[
\EE\big[\text{(III)}\big]  \lesssim_{\alpha,d,M,l} \frac{\sigma^{-4ab}\delta^{-4b}}{\sqrt{n}\eta^3} + \frac{\sigma^{-6ab}\delta^{-6b}}{n\eta^2}.
\]
Finally, bounding $\log N_{\sigma,\delta}$ and $\EE\big[\Delta(X_{1:n})\big]$ in \eqref{eq:multiplier_gaussian_comparison} using \eqref{eq: L2P metric entropy bound} and \eqref{eq: Delta bound}, respectively, we obtain
\[\EE[\text{(IV)}] \lesssim_{\alpha,d,M,l} \frac{\sigma^{-a(1+b)}\delta^{-b} n^{-1/4}}{\eta}.\]

As earlier, let $\sigma = \sigma_n$. Choosing $\eta = \eta_n = KC_{\alpha,d,M,l}n^{-\frac{1-b}{2(3+b)}}\sigma^{-a}$ for $K > 1$ and $\delta = \delta_n = n^{-1/(6+2b)}\sigma_n^{-a}$ gives
\[
\EE\big[\text{(I)} + \text{(II)} + \text{(III)} + \text{(IV)}\big] = O(1)/K.
\]
Choosing $K = K_n \to \infty$, Lemma~\ref{lem: Conditional Strassen's theorem} yields the claim (i). For the claim (ii), we set $\delta_n \mspace{-2mu}=\mspace{-2mu} n^{-1/(6+2b)}\sigma_n^{-a}$ and $\eta_n = 2C_{\alpha,d,M,l} n^{-(1-b)/(4(3+b))}\sigma^{-a/2}$, to obtain
\[\rho\big(\law\left(\|\GG_n^*\|_{\cF_\sigma}\middle|\,X_{1:n}\right), \law\left(\|G_P\|_{\cF_\sigma}\right)\big) \leq \eta_n \wedge \big[\text{(I)+(I)+(III)+(IV)}\big],\]
where $\EE\left[\text{(I)+(I)+(III)+(IV)}\right] \lesssim_{\alpha,d,M,l} n^{-(1-b)/(4(3+b))}\sigma^{-a/2}$. Hence,
\[
\rho\big(\law\left(\|\GG_n^*\|_{\cF_\sigma}\middle|\,X_{1:n}\right), \law\left(\|G_P\|_{\cF_\sigma}\right)\big) = O_\PP\big(n^{-(1-b)/(4(3+b))}\sigma^{-a/2}\big),
\]
completing the proof.
\qed

\section{Supplementary Proofs for Section \ref{sec: low dim case}}
\label{sec: low dim case proofs}

\subsection{Supplementary Proofs for Theorem \ref{thm: low dim rate for exp}}

\begin{proof}[Proof of Lemma~\ref{lem: metric entropy small ball}]
The proof is identical to that of Theorem 2.7.1 in \citet{VaWe1996}, except for the fact that the $\epsilon^{1/\alpha}$-net of $B(\epsilon^{1/\alpha})$ is just one point: namely the origin. The details are omitted for brevity.
\end{proof}

\begin{proof}[Proof of Lemma~\ref{lem: non integrable l2 entropy}]

Without loss of generality let $\|F\|_{L^2(P)} = 1$.
Choose integers $q_0 = 0, q_2$ such that $2^{-q_2-2} < \gamma < 2^{-q_2-1}$. Construct a nested sequence of partitions $\cF = \cup_{i=1}^{N_q} \cF_{qi}$ such that for $q=q_0=0$, we have the entire function class as our partition, and for each integer $q \geq q_0$, $\cF_{qi}$ satisfies
\[\Big\|\sup_{f,g \in \cF_{qi}} |f - g|^*\Big\|_{L^2(P)} < 2^{-q}\|F\|_{L^2(P)},\]
where $\cF_{01}$ consists of the entire function class. $N_q$ can be chosen to satisfy 
\[\log N_q \leq \sum_{r=1}^{q} \log N_{[\,]}\big(2^{-q}\|F\|_{L^2(P)}, \cF, \|\cdot\|_{L^2(P)}\big).\]

Fix a function $f_{qi}$ in each partition $\cF_{qi}$, with $f_{01} = 0$, and let
\[\pi_q f = f_{qi},\ \ \text{for } f \in \cF_{qi},\]
\[\triangle_q f = \sup_{f,g \in \cF_{qi}} |f - g|^*,\ \ \text{for } f \in \cF_{qi}.\]
Define the following quantities:
\begin{align*}
    a_q &= \frac{2^{-q}}{\sqrt{1 + \log N_{q+1}}},\\
    A_{q-1}f &= \ind\{\triangle_{q_0}f \leq \sqrt{n}a_{q_0},\dots,\triangle_{q-1}f \leq \sqrt{n}a_{q-1}\},\\
    B_q f &= \ind\{\triangle_{q_0}f \leq \sqrt{n}a_{q_0},\dots,\triangle_{q-1}f \leq \sqrt{n}a_{q-1}, \triangle_{q-1}f > \sqrt{n}a_{q-1},\}\\
    B_{q_0} f &= \ind\{\triangle_{q_0} f > \sqrt{n} a_{q_0}\}.
\end{align*}

The following decomposition then holds for each $f \in \cF$
(see \citet[Page 241]{VaWe1996}).
\begin{equation}
    \begin{split}
         (f - \pi_{q_0}f)&B_{q_0}f\ +\ \sum_{q_0 + 1}^{q_2} (f - \pi_q f) B_q f+\sum_{q_0 + 1}^{q_2} (\pi_q f - \pi_{q-1} f) A_{q-1} f +\ (f - \pi_{q_2}f)A_{q_2}f.
    \end{split}
\end{equation}

In our case $\pi_{q_0} f = 0$ for each $f \in \cF$. Applying $\GG_n$ to both sides and taking supremum over the function class, we get:

\begin{equation}
\label{eq:emp_process_decomp_nonintegrable}
    \begin{split}
        \EE\left[\|\GG_n\|_\cF\right] \leq \EE&\left[\sup_{f \in \cF}\ 
        \GG_n f B_{q_0}f\right]\ +\ \sum_{q=1}^{q_2} \EE\left[\sup_{f \in \cF} \GG_n(f - \pi_q f) B_q f\right]\\
        &\ +\ \sum_{q=1}^{q_2} \EE\left[\sup_{f \in \cF} \GG_n(\pi_q f - \pi_{q-1} f) A_{q-1} f\right] +\ \EE\left[\sup_{f \in \cF} \GG_n(f - \pi_{q_2}f)A_{q_2}f\right].
    \end{split}
\end{equation}
The first term in the display above can simply be bounded as follows:
\[
\EE\left[\sup_{f \in \cF} \GG_n f B_{q_0}f\right] \leq 2\sqrt{n}PF\ind_{\{2F > \sqrt{n}a_{q_0}\}}.
\]
Treatment of the second and third terms in \eqref{eq:emp_process_decomp_nonintegrable} is identical to that in Lemma 2.14.3 in \citet{VaWe1996}. 
For the final term, note that
\[
\GG_n(f - \pi_{q_2}f)A_{q_2}f \leq \sqrt{n}P_n \triangle_{q_2}f A_{q_2}f + \sqrt{n}P \triangle_{q_2}f A_{q_2}f \leq 2\sqrt{n}a_{q_2} \leq 8 \gamma\sqrt{n}.
\]
\end{proof}

\subsection{Proof of Lemma~\ref{lem: affine support entropy bound}}
Recall that $P$ is supported on the hyperplane $H_s = \{x_0 + x : x \in V_s\}$, where $V_s$ is an $s$-dimensional subspace of $\RR^d$. Suppose $K_\epsilon$ is large enough so that $P(\|X - x_0\| \geq K_\epsilon) < \epsilon^{s^2/(2(s - 2))}$. Then,  $B(x_0, K_\epsilon) \cap H_s$ can be covered by $\big(1+K_\epsilon/\epsilon\big)^s \lesssim_{s,d} \big(K_\epsilon/\epsilon\big)^s$ many $\epsilon$-balls. 

Let $Z = \|X - x_0\|$. Then, $\|Z\|_{\psi_\beta} = A < \infty$ by the assumptions of the lemma. Hence, we have $\PP(Z \geq Ax) \leq C_\beta e^{-x^\beta}$. Consequently, we may choose $K_\epsilon \lesssim_{s,A,\beta} \log(1/\epsilon)^{1/\beta}$ and obtain 
\[
N_\epsilon(P, \epsilon^{s^2/(2(s-2))}) \lesssim_{s, A, \beta} \epsilon^{-s} \log(1/\epsilon)^{1/\beta}.
\]
This satisfies condition \eqref{eq: low dim entropy condition}, and also gives
\[
\limsup_{\epsilon\to0} \frac{\log N_\epsilon(P, \epsilon^{s^2/(s-2)})}{-\log \epsilon} \leq s,
\]
which implies $\dim^*_{\mathsf{SW}_1}(P) \leq s$.

\subsection{Proof of Corollary \ref{cor: sharp wass rate low dim}} Note that $\tilde{s} := \big(\dim_{\mathsf{SW1}}^*(P) + s\big)/2 > \dim_{\mathsf{SW1}}^*(P)$. Since $\tilde{s}$ depends only on $s$ and $P$, by Theorem \ref{thm: low dim rate for exp}, we have
\[
\EE\left[\sqrt{n}\gwassemp\right] \lesssim_{d,s,P} \sigma^{-\tilde{s}/2+1}(\log n)^{3/2}.
\]
Now,
\begin{align*}
    \EE\left[\wass(P_n,P)\right] &\leq \inf_{\sigma > 0} \Bigg\{\EE\left[\gwassemp\right] + 2\sigma\sqrt{d}\Bigg\}\\
    &\lesssim_{d,s,P} \inf_{\sigma > 0} \Bigg\{\sigma^{-\tilde{s}/2+1}(\log n)^{3/2}/\sqrt{n} + 2\sigma\sqrt{d}\Bigg\}\\
    &\lesssim_{d,s,P} \left(\frac{(\log n)^3}{n}\right)^{1/\tilde{s}}\\
    &\lesssim_{d,s,P} n^{-1/s}.
\end{align*}


\section{Additional Material and Proofs for Section \ref{sec: applications}}
\subsection{Technical Tools and Results on Minimum Smooth Wasserstein Estimation}\label{subsec:MDE_appen}
\label{subsec: mde}

This section contains the results on M-SWE that previously appeared in \citet{Goldfeld2020limit_wass}. Proofs of these result are similar to the ones on M-ESWE included in this paper, and hence omitted. See Appendix B of \citet{Goldfeld2020limit_wass} for details.

\medskip
As in Section \ref{subsec: mede}, where M-ESWE was treated, we assume that $P,Q_\theta\in\cP_1(\RR^d)$, for all $\theta\in\Theta$, and that $\Theta \subset \R^{d_{0}}$ is compact with nonempty interior. 
Before stating results on M-SWE, we reproduce here a technical lemma concerning lower-semicontinuity of $\gwass$ that originally appeared in Appendix B, \citet{Goldfeld2020limit_wass}, and is required for the proofs of results on M-SWE and M-ESWE.

\begin{lemma}[Continuity of $\gwass$]\label{LEMMA:MSWD_lsc}
The smooth Wasserstein distance $\gwass$ is lower semicontinuous (l.s.c.) relative to the weak convergence on $\cP(\RR^d)$ and continuous in $\wass$. Explicitly, (i) if $\mu_{k} \stackrel{w}{\to} \mu$ and $\nu_{k} \stackrel{w}{\to} \nu$,  then
\[
    \liminf_{k\to\infty}\gwass(\mu_k,\nu_{k})\geq\gwass(\mu,\nu); 
\]
and (ii) if $\wass(\mu_{k},\mu) \to 0$ and $\wass(\nu_{k},\nu) \to 0$,  then 
\begin{equation}
    \lim_{k\to\infty}\gwass(\mu_k,\nu_{k})=\gwass(\mu,\nu). 
\end{equation}
\end{lemma}

We will also need a lemma specifying conditions for the existence of minimizers (cf. p. 73, \cite{santambrogio2015}).

\begin{lemma}[Weierstrass criterion for the existence of minimizers]
\label{lem: Weierstrass}
Suppose $\mathcal{X}$ is \\a compact metric space, and let $f: \mathcal{X} \to \R \cup \{ +\infty \}$ be l.s.c. (i.e., $\liminf_{x \to \overline{x}} f(x) \ge f(\overline{x})$ for any $\overline{x} \in \mathcal{X}$). Then, $\argmin_{x \in \mathcal{X}} f(x)$ is nonempty. 
\end{lemma}

Next, we state results on measurability and consistency of M-SWE. The first result shows that $\hat{\theta}_n \in \argmin_{\theta \in \Theta} \gwass(P_n,Q_\theta)$ is measurable. 

\begin{theorem}[M-SWE measurability]
\label{thm:MSWD_measurable}
Assume that the map $\theta \mapsto Q_{\theta}$ is continuous relative to the weak topology, 
i.e., $Q_{\theta}\stackrel{w}{\to} Q_{\overline{\theta}}$ whenever $\theta \to \overline{\theta}$ in $\Theta$. Then, for every $n\in\NN$, there exists a measurable function $\omega \mapsto \hat{\theta}_{n} (\omega)$ such that $\hat{\theta}_{n}(\omega)\in\argmin_{\theta\in\Theta}\gwass\big(P_n(\omega),Q_\theta\big)$ for every $\omega \in \Omega$ (this also implies that $\argmin_{\theta\in\Theta}\gwass\big(P_n(\omega),Q_\theta\big)$ is nonempty).
\end{theorem}

The next result establishes consistency of the M-SWE. Its proof relies on \citet[Theorem 7.33]{rockafellar2009variational}, and can be found in \citet{Goldfeld2020limit_wass}. 

\begin{theorem}[M-SWE consistency]\label{thm:MSWD_inf_argmin} Assume that the map $\theta \mapsto Q_{\theta}$ is continuous relative to the weak topology.
Then the following hold:
\begin{enumerate}[(i)]
    \item $\infgwassn\to\infgwass$ a.s.;
    \item there exists an event with probability one on which the following holds: for any sequence $\{ \hat{\theta}_{n} \}_{n \in \mathbb{N}}$ of measurable estimators such that 
    \[
    \gwass (P_{n},Q_{\hat{\theta}_{n}}) \mspace{-2.5mu} \le \mspace{-2.5mu} \inf_{\theta \in \Theta}\gwass (P_n,Q_{\theta}) + o(1),
    \]
    the set of cluster points of $\{ \hat{\theta}_{n} \}_{n \in \mathbb{N}}$ is included in $\argmin_{\theta \in \Theta} \gwass (P,Q_{\theta})$;
    \item in particular, if $\argmin_{\theta \in \Theta} \gwass (P,Q_{\theta})$ is unique, i.e., $\argmin_{\theta \in \Theta} \gwass (P,Q_{\theta}) = \{ \theta^{\ast} \}$, then $\hat{\theta}_{n} \to \theta^{\ast}$ a.s.
\end{enumerate}
\end{theorem}

Following measurability and consistency of the M-SWE, the subsequent results specify the limit distributions of the M-SWE and the associated MSWD under norm differentiability assumptions \citep{pollard1980minimum}. Results are presented for the `well-specified' scenario, i.e., when $P=Q_{\theta^\ast}$ for some $\theta^\ast$ in the interior of $\Theta$. Some definitions are needed to make the connection to the setting of \citet{pollard1980minimum}. 
With any $Q\in\cP_1(\RR^d)$, associate the functional $Q^{\mspace{1mu}\sigma}:\mathsf{Lip}_{1,0}\to\RR$ defined by $Q^{\mspace{1mu}\sigma}(f):=Q(f\ast\gauss)=(Q\ast\Gauss)(f)$. Note that $\left\|Q^{\mspace{1mu}\sigma}\right\|_{\mathsf{Lip}_{1,0}}:=\sup_{f\in\mathsf{Lip}_{1,0}}\left|Q^{\mspace{1mu}\sigma}(f)\right|$ is finite as $Q\in\cP_1(\RR^d)$ and $|(f\ast\varphi_{\sigma})(x)| \le \|x\| + \sigma \sqrt{d}$ for any $f \in \mathsf{Lip}_{1,0}$. Consequently, $Q^{\mspace{1mu}\sigma}\in\ell^\infty(\mathsf{Lip}_{1,0})$ for any $Q\in\cP_1(\RR^d)$. Finally, observe that $\gwass(P,Q)=\lip{P^{\mspace{1mu}\sigma}-Q^{\mspace{1mu}\sigma}}$, for any $P,Q\in\cP_1(\RR^d)$ (cf. Section~\ref{SUBSEC:CLT for W1 proof}).

\medskip

The next result specifies the limit distribution of the (scaled) infimized SWD. This is central to deriving the limiting M-SWE distribution (cf. Theorem \ref{thm:MSWD_argmin_CLT} and Corollary \ref{cor:unique_solution} below). 

\begin{theorem}[M-SWE error limit distribution]\label{thm:MSWD_inf_CLT}
Let $P$ satisfy the conditions of Theorem \ref{thm: CLT for W1}. In addition, suppose that:
\begin{enumerate}[(i)]
    \item the map $\theta \mapsto Q_{\theta}$ is continuous relative to the weak topology;
    \item $P \ne Q_{\theta}$ for any $\theta \ne \theta^{\ast}$.
    \item there exists a vector-valued functional $D^{\mspace{1mu}\sigma}\mspace{-2mu}\in\mspace{-2mu} (\ell^\infty(\mathsf{Lip}_{1,0}))^{d_0}$ such that $\lip{Q^{\mspace{1mu}\sigma}_\theta\mspace{-2mu}-Q^{\mspace{1mu}\sigma}_{\theta^\ast}\mspace{-2mu}-\mspace{-2mu}\langle\theta-\theta^\ast\mspace{-3mu},\mspace{-2mu}D^{\mspace{1mu}\sigma}\rangle}\mspace{-2mu}$ $=o(\|\theta-\theta^\ast\|)$ as $\theta \to \theta^\ast$, where $\langle t,D^{\mspace{1mu}\sigma} \rangle:=\sum_{i=1}^{d_0} t_iD_i^{\mspace{1mu}\sigma}$ for $t\in\RR^{d_0}$;
    \item the derivative $D^{\mspace{1mu}\sigma}$ is \textit{nonsingular} in the sense that $\langle t,D^{\mspace{1mu}\sigma} \rangle\neq 0$, i.e., $\langle t,D^{\mspace{1mu}\sigma} \rangle\in\ell^\infty(\mathsf{Lip}_{1,0})$ is not the zero functional for all $0\neq t\in\RR^{d_0}$.
\end{enumerate}
Then, $\sqrt{n}\inf_{\theta\in\Theta}\gwass(P_n,Q_\theta) \stackrel{w}{\to} \inf_{t \in \RR^{d_0}}\lip{G_P^{\mspace{1mu}\sigma}-\ip{t,D^{\mspace{1mu}\sigma}}},$ where $G_{P}^{\mspace{1mu}\sigma}$ is the Gaussian process from Theorem \ref{thm: CLT for W1}. 
\end{theorem}


\begin{remark}[Primitive conditions for norm differentiability (iii)]
Suppose that the family $\{ Q_{\theta} \}_{\theta \in \Theta}$ is dominated by a common Borel measure $\nu$ on $\R^{d}$, and let $q_{\theta}$ denote the density of $Q_{\theta}$ with respect to $\nu$, i.e., $\dd Q_{\theta} = q_{\theta} \dd\nu$. Then, $Q_{\theta} \ast \Gauss$ has Lebesgue density $x \mapsto \int \varphi_{\sigma}(x-t) q_{\theta}(x) \dd \nu (t)$. Assume that $q_{\theta}$ admits the Taylor expansion $q_{\theta}(x) = q_{\theta^{\ast}} (x) + \dot{q}_{\theta^{\ast}}(x) \cdot (\theta - \theta^{\ast}) + r_{\theta}(x) \cdot (\theta - \theta^{\ast})$ with $r_{\theta}(x) = o(1)$ as $\theta \to \theta^{\ast}$. Then, one may verify that Condition (iii) holds with $D^{\mspace{1mu}\sigma}(f) = \int f(x)\int \varphi_{\sigma}(x-t) \dot{q}_{\theta^{\ast}}(t) \dd \nu (t) \dd x = \int (f \ast\varphi_{\sigma}) (t) \dot{q}_{\theta^{\ast}} (t) \dd\nu (t)$, for $f \in \mathsf{Lip}_{1,0}$, provided that $\int (1+\|t\|) \| \dot{q}_{\theta^{\ast}}(t) \| \dd \nu(t) < \infty$ and $\int (1+\| t \|) \| r_{\theta}(t) \| \dd\nu(t) = o(1)$ (use the fact that $|f(t)| \le \| t \|$, for any $f \in \mathsf{Lip}_{1,0}$).
\end{remark}

The second limit distribution result on M-SWE concerns convergence in distribution of the M-SWE. Optimally, the limit distribution of $\sqrt{n}(\hat{\theta}_{n}-\theta^\ast)$, for some measurable estimator $\hat{\theta}_{n} \in \argmin_{\theta \in \Theta} \gwass (P_{n},Q_{\theta})$, is the object of interest. However, a limit is guaranteed to exist only when the (convex) function $t \mapsto \lip{G_P^{\mspace{1mu}\sigma}-\ip{t,D^{\mspace{1mu}\sigma}}}$ has a unique minimum a.s. (see Corollary \ref{cor:unique_solution} below). To avoid this stringent assumption, 
the next result considers the set of approximate minimizers $\hat{\Theta}_n:=\big\{ \theta\in\Theta: \gwass(P_n,Q_\theta)\leq \inf_{\theta'\in\Theta}\gwass(P_n,Q_{\theta'})+\lambda_n/\sqrt{n} \big\}$ and establishes a convergence result in the space of compact convex sets with the Hausdorff topology.
To this end, let \[K(L,\beta):=\left\{t\in\RR^{d_0}: \lip{L-\ip{t,D^{\mspace{1mu}\sigma}}}\leq \inf_{t'\in\RR^d}\lip{L-\ip{t',D^{\mspace{1mu}\sigma}}}\mspace{-8mu}+\mspace{-2mu}\beta\right\}.\]
Lemma 7.1 of \citet{pollard1980minimum} ensures that for any $\beta\geq 0$, $L \mapsto K(L,\beta)$ is a measurable map~from $\ell^{\infty}(\mathsf{Lip}_{1,0})$ into  $\mathfrak{K}$ -- the class of all compact, convex, and nonempty subsets of $\RR^{d_0}$ -- endowed with the Hausdorff topology, i.e., the topology induced by the metric $d_\mathsf{H}(K_1,K_2):=\inf\big\{\delta>0:$ $K_1^{\delta}\supset K_2,\ K_2^{\delta}\supset K_1\big\}$, where $K^{\delta}$ is the $\delta$-fattening of $K$. 

\begin{theorem}[M-SWE limit distribution]\label{thm:MSWD_argmin_CLT}
Let the conditions of Theorem \ref{thm:MSWD_inf_CLT} hold.\\Then, there exists a sequence of nonnegative reals  $\beta_{n} \downarrow 0$ such that the following hold:
\begin{enumerate}[(i)]
    \item $\PP_*\big( \hat{\Theta}_n\subset \theta^\ast+n^{-1/2}K(\GG_n^{\mspace{1mu}\sigma},\beta_n\big)\big)\to 1$, where $\GG_n^{\mspace{1mu}\sigma}:=\sqrt{n}(P_n^{\mspace{1mu}\sigma}-P^{\mspace{1mu}\sigma})$ is the (smooth) empirical process and $\PP_*$ denotes inner probability.
\item $K(\GG_n^{\mspace{1mu}\sigma},\beta_{n} )\stackrel{w}{\to} K(G_P^{\mspace{1mu}\sigma},0)$ as $\mathfrak{K}$-valued random variables.
\end{enumerate}
\end{theorem}

Given Theorem \ref{thm:MSWD_inf_CLT}, the proof of Theorem \ref{thm:MSWD_argmin_CLT} follows by repeating of the argument from \citet[Section 7.2]{pollard1980minimum}. 
If $\argmin_{t \in \R^{d_{0}}} \lip{G_P^{\mspace{1mu}\sigma}-\ip{t,D^{\mspace{1mu}\sigma}}}$ is unique a.s. (a nontrivial assumption), then Theorem \ref{thm:MSWD_argmin_CLT} simplifies as~follows.\footnote{Note that $\argmin_{t \in \R^{d_{0}}} \lip{G_P^{\mspace{1mu}\sigma}-\ip{t,D^{\mspace{1mu}\sigma}}}\neq \emptyset$ provided that $D^{\mspace{1mu}\sigma}$ is nonsingular, since the latter guarantees that $\lip{G_P^{\mspace{1mu}\sigma}-\ip{t,D^{\mspace{1mu}\sigma}}} \to \infty$ as $\| t \| \to \infty$. }

\begin{corollary}[Simplified M-SWE limit distribution]
\label{cor:unique_solution}
Under the conditions of Theorem \ref{thm:MSWD_inf_CLT}, let $\{ \hat{\theta}_{n} \}_{n \in \mathbb{N}}$ be a sequence measurable estimators such that 
\[
\gwass (P_{n},Q_{\hat{\theta}_{n}})\mspace{-3mu}\le\mspace{-3mu}\inf_{\theta \in \Theta}\mspace{-2mu}\gwass (P_{n},Q_{\theta}) + o_{\PP}(n^{-\frac{1}{2}}).
\]
Then, provided that $\argmin_{t \in \R^{d_{0}}} \lip{G_P^{\mspace{1mu}\sigma}-\ip{t,D^{\mspace{1mu}\sigma}}}$ is unique a.s., we have $\sqrt{n}(\hat{\theta}_{n} - \theta^{\ast}) \stackrel{w}{\to} \argmin_{t \in \R^{d_{0}}}\lip{G_P^{\mspace{1mu}\sigma}-\ip{t,D^{\mspace{1mu}\sigma}}}$.
\end{corollary}

The last result of this section is a high probability generalization bound for generative modeling via M-SWE, in accordance to the framework from \citet{arora2017generalization,zhang2017discrimination}. The goal is to control the gap between the $\gwass$ loss attained by approximate, possibly suboptimal, empirical minimizers and the population loss $\inf_{\theta\in\Theta}\gwass (P,Q_\theta)$. Upper bounding this gap by the rate of empirical convergence, the concentration result from Proposition \ref{prop: concentration} together with the triangle inequality imply the following.

\begin{corollary}[M-SWE generalization error bound]
\label{cor:M-SWE-generalization-error}
Let $P$ be compactly supported, and suppose that $\hat{\theta}_n$ is an estimator such that $\gwass(P_n,Q_{\hat{\theta}_n}) \leq \inf_{\theta \in \Theta} \gwass(P_n,Q_\theta) + \epsilon$, for some $\epsilon>0$. We have
\[
\PP\left(\gwass(P,Q_{\hat{\theta}_n}) -\inf_{\theta \in \Theta}\gwass(P,Q_\theta) > \epsilon +t \right) \leq C e^{-c nt^2},
\]
for constants $C,c $ independent of $n$ and $t$.
\end{corollary}

\subsection{Proof of Theorem~\ref{thm:M_ESWE_measurable}}
\label{subsec:M_ESWE_measurable_proof}
We will need an extension of Lemma~\ref{LEMMA:MSWD_lsc} to the expected SWD.
\begin{lemma}[Continuity of ESWD]
\label{lem:eswd_continuity}
The expected SWD map $(P,Q) \mapsto \EE[\gwass(P,Q_m)]$ is lower semicontinuous (l.s.c.) relative to the weak convergence on $\cP(\RR^d)$ and continuous in $\wass$. Explicitly, (i) if $\mu_{(k)} \stackrel{w}{\to} \mu$ and $\nu_{(k)} \stackrel{w}{\to} \nu$,  then
\[
    \liminf_{k\to\infty}\EE[\gwass(\mu_{(k)},\nu_{(k),m})]\geq\EE[\gwass(\mu,\nu_m)],
\]
where $\nu_{(k),m}, \nu_m$ are empirical measures constructed from $m$ i.i.d. samples from $\nu_{(k)}$ and $\nu$ respectively;
and (ii) if $\wass(\mu_{(k)},\mu) \to 0$ and $\wass(\nu_{(k)},\nu) \to 0$,  then 
\begin{equation}
    \lim_{k\to\infty}\EE[\gwass(\mu_{(k)},\nu_{(k),m})]=\EE[\gwass(\mu,\nu_m)]. 
\end{equation}
\end{lemma}

\begin{proof}[Proof of Lemma~\ref{lem:eswd_continuity}]
Since $Q_{(k)} \overset{w}{\to} Q$, we also have convergence of the corresponding m-fold product measures, i.e., $Q_{(k)}^m \overset{w}{\to} Q^m$. Then, by Skorohod representation theorem, there exist random variables $Y_{k1}, \dots, Y_{km} \overset{iid}{\sim} Q_{(k)}$ and $Y_1, \dots, Y_m \overset{iid}{\sim} Q$ so that $(Y_{k1}, \dots, Y_{km}) \to (Y_{1}, \dots, Y_{m})$ $\PP$-a.s.. This further implies that $Q_{(k)m} = \frac{1}{m} \sum_{i=1}^m \delta_{Y_{ki}} \overset{w}{\to} Q_m = \frac{1}{m} \sum_{i=1}^m \delta_{Y_i}$ $\PP$-a.s..  By Lemma~\ref{LEMMA:MSWD_lsc}, this implies 
\[
\liminf_{k \to \infty} \gwass(P_{(k)}, Q_{(k)m}) \overset{w}{\to} \gwass(P, Q)\ \ \PP-a.s..
\]
Taking expectation on both sides and applying Fatou's lemma, we have the first statement.
For (ii), note that if $W_1(Q_{(k)}, Q) \to 0$, there exists a coupling of $(Y_k, Y)$ of $Q_{(k)}$ and $Q$ so that $\EE[|Y_{k}  - Y|] \to 0$. Hence, we can again define $Y_{k1}, \dots, Y_{km} \sim Q_{(k)}$ i.i.d. and $Y_1, \dots, Y_m \sim Q$ i.i.d.., so that $\EE[\wass(Q_{(k)m}, Q_m)] \to 0$ as $k \to \infty$, where $Q_{(k)m}, Q_m$ are defined as earlier. Further, by triangle inequality,
\[
\big | \gwass(P_{(k)}, Q_{(k)m}) - \gwass(P, Q_m) \big | \leq \gwass(P_{(k)}, P) + \gwass(Q_{(k)m}, Q_m).
\]
Taking expectation on both sides and letting $k \to \infty$, we have the second statement.
\end{proof}

\begin{proof}[Proof of Theorem~\ref{thm:M_ESWE_measurable}]
By Lemma~\ref{lem: Weierstrass}, since $\Theta$ is compact with non-empty interior and the map the map $\theta \mapsto \EE \big [\gwass (P_n,Q_{\theta}) \,|\, X_{1:n} \big ](\omega)$ is lower-semicontinuous (Lemma~\ref{lem:eswd_continuity}), $\argmin_{\theta \in \Theta} \gwass (P_n(\omega),Q_{\theta})$ is non-empty.
As in the proof of Theorem~\ref{thm:MSWD_measurable}, it suffices to prove that the joint map $(x,\theta) \mapsto \EE_{Y_{1:m}} \big [\gwass(P_n(x), Q_{\theta,m} \big ]$ is continuous. Now, for fixed $x$ the above map is lower-semicontinuous in $\theta$ by Lemma~\ref{lem:eswd_continuity}, and hence measurable in $\theta$. For fixed $\theta$, taking $x_k \to x$, we have that $P_n(x_k) \to P_n(x)$ in $\wass$, so that by statement (ii) of Lemma~\ref{lem:eswd_continuity}, the above map is continuous in $x$. This shows that the above map is jointly measurable by Lemma 4.5.1 in \citet{aliprantis2006}, which completes the proof. 
\end{proof}

\subsection{Proof of Theorem~\ref{thm:M_ESWE_consistency}}
\label{subsec:M_ESWE_consistency_proof}

\noindent\underline{Proof of (i)}: As in the proof of Theorem~\ref{thm:MSWD_inf_argmin} (cf. Appendix B.3, \citet{Goldfeld2020limit_wass}), we will apply Theorem 7.31 of \citet{rockafellar2009variational} to prove the convergence of infimums. Define:
\[
\GG_n(\theta) = \GG_n(\theta, \omega) := \begin{cases}
\EE[\gwass(P_n, Q_{\theta, m}) | X_{1:n} ] (\omega), & \theta \in \Theta,\\
+\infty, &\ \theta \in \RR^{d_0} \setminus \Theta;
\end{cases}
\]
\[
g(\theta) = \begin{cases}
\gwass(P,Q_\theta), &\ \theta \in  \Theta,\\
+\infty, &\ \theta \in \RR^{d_0} \setminus \Theta.
\end{cases}
\]
For every fixed $\omega$ in a $\PP$-a.s. event, $\GG_n(., \omega)$ and $g$ are lower-semicontinuous in $\theta$, level bounded and proper. Hence, by Theorem 7.31 in \citet{rockafellar2009variational}, it only remains to show that $\GG_n$ epi-converges to $g$, i.e., it remains to verify that
\begin{enumerate}[(a)]
    \item $\limsup_n \inf_{\theta \in \cU} \GG_n(\theta) \leq \inf_{\theta \in \cU} g(\theta)$ for every open subset $\cU \subset \Theta$.
    \item $\liminf_n \inf_{\theta \in \cK} \GG_n(\theta) \geq \inf_{\theta \in \cK} g(\theta)$ for every compact subset $\cK \subset \Theta$.
\end{enumerate}
For (a), first note that if $\cU \cap \Theta = \emptyset$, both sides reduce to $+\infty$, and thus the ineqality trivially holds. Otherwise, choose a convergent sequence $\theta_n \in \cU$ (possible since $\Theta$ is compact: note that the limit $\theta^*$ may not  may not be in $\cU$) so that $\gwass(P,Q_{\theta_n}) \to \inf_{\theta \in \cU} \gwass(P, Q_{\theta})$. Then, 
\[\EE \big [ \gwass(P_n, Q_{\theta_n, m}) \,|\, X_{1:n} \big ] \leq \gwass(P_n, P) + \gwass(P, Q_{\theta_n}) + \EE \big [ \gwass(Q_{\theta_n}, Q_{\theta_n, m}) \,|\, X_{1:n} \big ].\]
The right hand side of the above display converges $\PP$-a.s. to $\gwass(P,Q_\theta^*)$. The statement (a) then follows by taking $\limsup_{n \to \infty}$ on both sides.

For (b), note that if $\cK \cap \Theta = \emptyset$, the statement again follows. Otherwise, we can assume $\cK \subset \Theta$ w.l.o.g.. Choose $\theta_n(\omega) \in \argmin_{\theta\in \Theta} \EE[\gwass(P_n, Q_{\theta, m}) | X_{1:n}]$. Then,
\[ \liminf_n \inf_{\theta \in \Theta} \EE \big [ \gwass(P_n, Q_{\theta, n}) \,|\, X_{1:n} \big ] = \EE \big [ \gwass(\theta_n, Q_{\theta_n, m}) \,|\, X_{1:n} \big ]. \]
For every subsequential limit of $\EE [ \gwass(\theta_n, Q_{\theta, n}) | X_{1:n} ]$, there exists a further subsequence $n_{k_l}$ such that $\theta_{n_{k_l}} \to \theta^* \in \cK$, so that $Q_{\theta_{n_{k_l}}} \overset{w}{\to} Q_{\theta^*}$. This implies that the limit of the subsequence $\lim_k \EE [ \gwass(\theta_{n_k}, Q_{\theta_{n_k}, m}) | X_{1:n_k} ] = \lim_l \EE [ \gwass(\theta_{n_{k_l}}, Q_{\theta_{n_{k_l}}, m}) | X_{1:n_{k_l}} ]$ can be bounded as
\begin{align*}
\lim_l \EE \big [ \gwass(P_{n
_{k_l}}, Q_{\theta_{n_{k_l}}, m}) \,|\, X_{1:n} \big ] &\geq \liminf_l \left ( \gwass (P_{n_{k_l}}, Q_{\theta_{n_{k_l}}}) - \EE \big [ \gwass(Q_{\theta_{n_{k_l}}}, Q_{\theta_{n_{k_l}}}) \big ] \right ) \\
&\geq \gwass(P, Q_{\theta^*}),
\end{align*}
where the last inequality follows from Lemma~\ref{LEMMA:MSWD_lsc}. Since $\gwass(P, Q_{\theta^*}) \geq \inf_{\theta \in \cK} \gwass(P, Q_\theta)$, the result follows.

\noindent\underline{Proof of (ii):} By Fatou's lemma, for any such cluster point $\theta^* = \lim_k \hat\theta_{n_k}$,
\[
\liminf_n \EE \big [ \gwass(P_n, Q_{m, \hat\theta_n}) \,|\, X_{1:n} \big ] \leq \EE \big [ \lim_k \gwass(P_{n_k}, Q_{m(n_k), \hat\theta_{n_k}}) \,|\, X_{1:n} \big ] \mspace{-2mu} = \mspace{-2mu} \gwass(P, Q_{\theta^*})\ \PP-a.s..
\]
However, the left most side of the above display is equal to $\inf_{\theta \in \Theta} \gwass(P, Q_\theta)$ by (i), and hence, (ii) follows.

\noindent\underline{Proof of (iii):} This directly folows from (ii).
\hfill\qed

\subsection{Proof of Theorem~\ref{thm:MESWD_inf_CLT}}
\label{subsec:MESWD_proof}
Consider $\tilde\theta_n \in \argmin_{\theta \in \Theta} \EE \big [ \gwass(P_n, Q_{\theta, m}) \,|\, X_{1:n} \big ]$ and $\hat\theta_n \in \argmin_{\theta\in\Theta} \gwass(P_n, Q_\theta)$. We will show that
\be\label{eq:inf_MESWD_approx}
\EE \big [ \gwass(P_n, Q_{\tilde\theta_n,m}) \,|\, X_{1:n} \big ] \leq \inf_{\theta \in \Theta} \gwass(P_n, Q_{\theta}) + o_{\PP}\left(\frac{1}{\sqrt{n}}\right).
\ee
The result will then follow by Theorem~\ref{thm:MSWD_inf_CLT} combined with Slutsky's theorem.

We have
\begin{align*}
    \gwass(P_n, Q_{\tilde\theta_n}) &\leq \inf_{\theta\in\Theta} \EE \big [ \gwass(P_n, Q_{\theta, m}) \,|\, X_{1:n} \big ] + \EE \big [ \gwass(Q_{\tilde\theta_n}, Q_{\tilde\theta_n, m} \,|\, X_{1:n} \big ]\\
    &\leq \inf_{\theta \in \Theta} \Big [ \gwass(P_n, Q_\theta) + \EE \big [ \gwass(Q_\theta, Q_{\theta, m}) \big ] \Big ] + \EE \big [ \gwass(Q_{\tilde\theta_n}, Q_{\tilde\theta_n, m}) \,|\, X_{1:n} \big ]\\
    &\leq \inf_{\theta \in \Theta} \gwass(P_n, Q_\theta) + \EE \big [ \gwass(Q_{\hat\theta_n}, Q_{\hat\theta_n, m}) \,|\, X_{1:n} \big ] + \EE \big [ \gwass(Q_{\tilde\theta_n}, Q_{\tilde\theta_n, m}) \,|\, X_{1:n} \big ].
\end{align*}
Under the assumption that $\sup_\theta \int |x|^q\,\dd Q_\theta < \infty$ for $q > 2(d+1)$, the last two terms are $O_\PP(1/\sqrt{m_n})$ by \eqref{EQ:expected_bound} (cf. Remark~\ref{rem:moment_condition}). Since $m_n \gg n$, the result follows.
\hfill\qed

\subsection{Proof of Theorem~\ref{thm:M_ESWE_limit}}
\label{subsec:M_ESWE_limit_proof}
Under the given conditions, \eqref{eq:inf_MESWD_approx} holds as in the proof of Theorem~\ref{thm:MESWD_inf_CLT}, so that any estimator $\hat\theta_n$ satisfying the conditions of Theorem~\ref{thm:M_ESWE_limit} also satisfies the conditions of Corollary~\ref{cor:unique_solution}. This completes the proof.
\hfill\qed

\end{appendix}

\bibliographystyle{plainnat}
\bibliography{ref_combined}

\end{document}